\documentclass{article}
\usepackage[a4paper, total={6.3in, 9.5in}]{geometry}
\usepackage{subscript}
\usepackage{titlesec}
\usepackage{titletoc} 
\titleformat{\section}[block]{\bfseries\large}{\thesection}{1em}{}
\usepackage{float}
\usepackage{amsmath,amsfonts}
\usepackage{mathrsfs}  
\usepackage{amsthm}
\usepackage{amssymb}
\usepackage[USenglish]{babel}
\usepackage{enumerate}
\usepackage[shortlabels]{enumitem}

\usepackage{cite}

\usepackage{xcolor}

\usepackage{todonotes}
\usepackage[final, pdftex, pdfpagelabels, pdfstartview = {FitH}, bookmarks, plainpages = false, linktoc=all,linkcolor=black, citecolor=blue, urlcolor=blue,
filecolor=black]{hyperref}
\usepackage{cleveref}
\usepackage{caption}
\usepackage[labelformat=simple]{subcaption}

\usepackage{dynkin-diagrams}
\usepackage{standalone}
\usepackage{lipsum}
\usepackage{tikz}
\usepackage{tikz-cd}
\usepackage[title]{appendix}

\usepackage{oldgerm}
\usepackage{multicol}

\theoremstyle{plain}
\newtheorem{theorem}{Theorem}[section]
\newtheorem{lemma}[theorem]{Lemma}
\newtheorem{prop}[theorem]{Proposition}
\newtheorem{cor}[theorem]{Corollary}

\newtheorem{conjecture}[theorem]{Conjecture}
\newtheorem{lettertheorem}{Theorem}

\theoremstyle{definition}
\newtheorem{definition}[theorem]{Definition}

\newtheorem{remark}[theorem]{Remark}
\newtheorem{example}[theorem]{Example}

\theoremstyle{remark}

\usepackage{newpxtext}
\newcommand\varfrak[1]{\mathord{\text{\textgoth{#1}}}}

\makeatletter
\newcommand{\ie}{\emph{i.e. }\@ifnextchar.{\!\@gobble}{}}
\newcommand{\eg}{\emph{e.g.}\@ifnextchar.{\!\@gobble}{}}
\newcommand{\etc}{etc\@ifnextchar.{}{.\@}}
\newcommand{\zfrac}{  \mathbb{Z} \left[    \frac{1}{(n+1)!}   \right]  }

\makeatother

\newcommand{\perm}{\widetilde{S}_{n+1}(\mathbb{Z})}
\newcommand{\li}[1]{   \mathcal{I}_{a}(#1)} 
\newcommand{\lip}[1]{   \mathcal{I}_{a}^+(#1)} 
\newcommand{\lr}[2]{   \langle  #1 , #2\rangle} 
\newcommand*\circled[1]{\tikz[baseline=(char.base)]{
				\node[shape=circle,draw,inner sep=2pt] (char) {#1};}}

\renewcommand{\tilde}{\widetilde}

\newcommand{\bbN}{\mathbb{N}}

\newcommand{\bbR}{\mathbb{R}}
\newcommand{\bbZ}{\mathbb{Z}}

\newcommand{\id}{\text{id}}

\DeclareMathOperator{\vol}{Vol}
\DeclareMathOperator{\relvol}{relVol}

\numberwithin{equation}{section}

\title{Paper BOAT\\[0.5em] \large \textit{Bruhat Order and Alcovic Tilings}}

\author{ Federico Castillo\thanks{Pontificia Universidad Catolica, Santiago, Chile}, Damian de la Fuente\thanks{LAMFA, Universit\'e de Picardie Jules Verne, Amiens, France}, Nicolas Libedinsky\thanks{Universidad de Chile, Santiago, Chile}  and David Plaza\thanks{Universidad de Talca, Talca, Chile}  }

\date{ }

\begin{document}

\maketitle

\begin{abstract}
%We produce a formula for the size of lower Bruhat intervals for elements in the dominant cone of an affine Weyl group of type $A$. Additionally, we conjecture a generalization of this formula to all affine Weyl groups within the lowest two-sided Kazhdan-Lusztig cell.
We derive a formula for computing the size of lower Bruhat intervals for elements in the dominant cone of an affine Weyl group of type~$A$. This enumeration problem is reduced to counting lattice points in certain polyhedra. Our main tool is a decomposition—or tiling—of each interval into smaller, combinatorially tractable pieces, which we call paper boats. We also conjecture a generalization of this formula to all affine Weyl groups, restricted to elements in the lowest two-sided Kazhdan–Lusztig cell, which contains almost all of the elements.
\end{abstract}

\maketitle

\section*{A note on the title}

The title of this paper reflects two aspects of this work, though the connection between them emerged unexpectedly. 
Initially, we referred to the sets that we introduce as ``paper boats'' due to their visual resemblance in certain cases to actual paper boats (see the green objects in \Cref{fig:introA}). 
Subsequently, we observed that this name could also serve as an acronym for ``Bruhat Order and Alcovic Tilings'', which is the central theme of the paper (see \Cref{fig:introA,fig:introB}).

\section{Introduction}

In \cite{castillo2023size}, we initiated the study of the sizes of lower Bruhat intervals in affine Weyl groups and uncovered intriguing connections with convex geometry. 
In that work, we derived formulas for a set of elements, denoted by  $\varfrak{A}$ in this introduction, which plays a significant role in representation theory. 

\begin{figure}[hbt!]
\captionsetup[subfigure]{labelformat=simple}
\centering

\begin{subfigure}{0.3\textwidth}
\centering
\begin{tikzpicture}[x=0.75cm, y=0.75cm,>=Latex]
% el origen esta en  (0,0)++(60:9)++(4,0) 

% \clip (-0.5,-1.4) rectangle + (16.3,18.5);
\clip (3,6) rectangle + (11,12);

% \filldraw[lightgray] (0,0)++(60:9)++(4,0) ++ (60:7)++(120:3)--++ (240:1) --++ (-1,0)--++ (240:1) --++ (-1,0) --++ (240:1) --++ (-1,0) --++ (240:1) --++ (-1,0) --++ (240:1) --++ (-1,0) --++ (240:1) --++ (-1,0) --++ (240:1) --++ (-1,0) --++ (-60:1) --++ (240:1)--++ (-60:1) --++ (240:1)--++ (-60:1) --++ (240:1)--++ (1,0) --++ (-60:1) --++ (1,0) --++ (-60:1) --++ (1,0) --++ (-60:1) --++ (1,0) --++ (-60:1) --++ (1,0) --++ (-60:1) --++ (1,0) --++ (-60:1)  --++ (1,0) --++ (-60:1) --++ 
% (60:1) --++ (1,0) --++ 
% (60:1) --++ (1,0) --++ 
% (60:1) --++ (1,0) --++
% (120:1) --++ (60:1) --++
% (120:1) --++ (60:1) --++
% (120:1) --++ (60:1) --++
% (120:1) --++ (60:1) --++
% (120:1) --++ (60:1) --++
% (120:1) --++ (60:1) --++
% (120:1) --++ (60:1) --++
% (-1,0)  --++ (120:1) --++
% (-1,0)  --++ (120:1) --++
% (-1,0)  --++ (120:1) --++
% (-1,0)  
% -- cycle;

\filldraw[green]  (0,0)++(60:9)++(4,0)++(60:8)++(120:1)--++(-1,0)--++(120:1)--++(240:1)--++(-1,0)--++(240:1)--++
(-1,0)--++(240:1)--++(-1,0)--++(-60:1)--++(1,0)--++(-60:1)--++(2,0)--++(60:1)--++(-1,0)--++(120:1)--++(1,0)--++
(60:1)--++(1,0)--++(60:1)
--cycle;

\filldraw[red]  (0,0)++(60:9)++(4,0)++(60:2)--++(60:2)--++(-1,0)--++(120:1)--++(240:1)--++(-1,0)--++(-60:1)--++
(1,0)
--cycle;

\filldraw[red]  (0,0)++(60:9)++(4,0)++(60:5)--++(60:2)--++(-1,0)--++(120:1)--++(240:1)--++(-1,0)--++(-60:1)--++
(1,0)
--cycle;

\filldraw[blue]  (0,0)++(60:9)++(4,0)--++(60:2)--++(120:1)--++(240:1)--++(-1,0)
--cycle;

\filldraw[orange]  (0,0)++(60:9)++(4,0)--++(120:2)--++(1,0)--++(60:1)--++(-1,0)--++(120:1)--++(240:1)
--cycle;

\filldraw[orange]  (0,0)++(60:9)++(4,0)--++(120:5)--++(1,0)--++(60:1)--++(-1,0)--++(120:1)--++(240:1)
--cycle;

\filldraw[magenta]  (0,0)++(60:9)++(4,0)--++(120:3)--++(60:1)--++(1,0)--++(60:1)--++(-1,0)--++(120:1)--++
(240:1)--++(-1,0)
--cycle;

\filldraw[cyan]  (0,0)++(60:9)++(4,0)++(60:4)--++(60:1)--++(120:1)--++(240:1)--++(-1,0)--++(-60:1)
--cycle;

\filldraw[cyan]  (0,0)++(60:9)++(4,0)++(60:7)--++(60:1)--++(120:1)--++(240:1)--++(-1,0)--++(-60:1)
--cycle;

\foreach \i in {0,...,8} {
     \foreach \j in {-3,...,12}{
  \draw[gray, thick](0:0)++(1+3*\i,0)++(60:\j)++(120:\j)--++(60:1)--++(120:1)--++(-1,0)--++(240:1)--++(300:1)--++(1,0);
  }}
  \foreach \i in {0,...,5} {
     \foreach \j in {-3,...,12} {
  \draw[gray, thick](0:0)++(60:1)++(1,0)++(1+3*\i,0)++(60:\j)++(120:\j)--++(60:1)--++(120:1)--++(-1,0)--++(240:1)--++(300:1)--++(1,0);
  }}
\foreach \i in {0,...,6} {
     \foreach \j in {-3,...,12}{
  \draw[gray, thick](0:0)++(60:1)++(1+3*\i,0)++(60:\j)++(120:\j)--++(60:1)--++(120:1)--++(-1,0)--++(240:1)--++(300:1)--++(1,0);
  }}  
  \foreach \i in {0,...,5} {
     \foreach \j in {-3,...,12} {
  \draw[gray, thick](0:0)++(2,0)++(1+3*\i,0)++(60:\j)++(120:\j)--++(60:1)--++(120:1)--++(-1,0)--++(240:1)--++(300:1)--++(1,0);
  }}
\foreach \i in {0,...,6} {
     \foreach \j in {-3,...,12}{
  \draw[gray, thick](0:0)++(-1,0)++(60:1)++(1+3*\i,0)++(60:\j)++(120:\j)--++(60:1)--++(120:1)--++(-1,0)--++(240:1)--++(300:1)--++(1,0);
  }}  
  \foreach \i in {0,...,5} {
     \foreach \j in {-3,...,12}{
  \draw[gray, thick](0:0)++(1,0)++(1+3*\i,0)++(60:\j)++(120:\j)--++(60:1)--++(120:1)--++(-1,0)--++(240:1)--++(300:1)--++(1,0);
  }}

\draw[black, ultra thick]  (0,0)++(60:9)++(4,0)++(-9,0)--++(17,0);
\draw[black, ultra thick]  (0,0)++(60:9)++(4,0)++(120:11)--++(300:18);
\draw[black, ultra thick]  (0,0)++(60:9)++(4,0)++(60:11)--++(240:18);

\draw[black, ultra thick] (0,0)++(60:9)++(4,0) ++(60:7) ++(120:3) --++ (240:1) --++ (-1,0) --++ (-60:1) 
--++ (2,0) --++(60:1) --++ (-1,0) --++ (120:1); 
\draw[black, ultra thick] (0,0)++(60:9)++(4,0) ++(60:5) ++(120:4) --++ (240:1) --++ (-1,0) --++ (-60:1) 
--++ (2,0) --++(60:1) --++ (-1,0) --++ (120:1);
\draw[black, ultra thick] (0,0)++(60:9)++(4,0) ++(60:5) ++(120:4)++(240:1)++(-1,0) --++ (240:1) --++ (-1,0) --++ (-60:1) --++ (2,0) --++(60:1) --++ (-1,0) --++ (120:1);
\draw[black, ultra thick] (0,0)++(60:9)++(4,0) ++(60:5) ++(120:4)++(240:1)++(-1,0)++(-60:1)++(1,0) --++ (240:1) --++ (-1,0) --++ (-60:1) --++ (2,0) --++(60:1) --++ (-1,0) --++ (120:1);
\draw[black, ultra thick] (0,0)++(60:9)++(4,0)++(120:6)++(60:1)--++(240:1)--++(-60:1)--++(1,0)--++(60:1) ;
\draw[black, ultra thick] (0,0)++(60:9)++(4,0)++(120:6)++(60:1)++(-60:3)--++(240:1)--++(-60:1)--++(1,0)--++(60:1)--++(-1,0)--++(120:1) ;
\draw[black, ultra thick] (0,0)++(60:9)++(4,0)++(120:6)++(60:1)++(-60:3)--++(1,0)--++(60:1)--++(-60:1)--++(1,0)--++(60:1)--++(120:1);
\draw[black, ultra thick] (0,0)++(60:9)++(4,0) ++(60:8)--++(120:1); 
\draw[black, ultra thick] (0,0)++(60:9)++(4,0) ++(60:7)--++(-1,0)--++(120:1); 
\draw[black, ultra thick] (0,0)++(60:9)++(4,0) ++(60:2)--++(120:1);

%  PUNTOS
\filldraw [white] (0,0)++(60:9)++(4,0) ++ (60:6)++(120:2)  ++ (-60:0) circle (3pt);
\draw[black ,very thick]  (0,0)++(60:9)++(4,0) ++ (60:6)++(120:2) ++ (-60:0)  circle(3pt) ;

\filldraw [white] (0,0)++(60:9)++(4,0) ++ (60:4) ++ (120:3) circle (3pt);
\draw[black ,very thick]  (0,0)++(60:9)++(4,0) ++ (60:4) ++ (120:3)  circle(3pt) ;

\filldraw [white] (0,0)++(60:9)++(4,0) ++ (60:7) circle (3pt);
\draw[black ,very thick]  (0,0)++(60:9)++(4,0) ++ (60:7)  circle(3pt) ;

\filldraw [white] (0,0)++(60:9)++(4,0) ++ (60:4) circle (3pt);
\draw[black ,very thick]  (0,0)++(60:9)++(4,0) ++ (60:4)  circle(3pt) ;

\filldraw [white] (0,0)++(60:9)++(4,0) ++ (60:1) circle (3pt);
\draw[black ,very thick]  (0,0)++(60:9)++(4,0) ++ (60:1)  circle(3pt) ;

\filldraw [white] (0,0)++(60:9)++(4,0) ++ (60:2)++(120:4) circle (3pt);
\draw[black ,very thick]  (0,0)++(60:9)++(4,0) ++ (60:2)++(120:4)  circle(3pt) ;

\filldraw [white] (0,0)++(60:9)++(4,0) ++(120:5) circle (3pt);
\draw[black ,very thick]  (0,0)++(60:9)++(4,0) ++(120:5)  circle(3pt) ;

\filldraw [white] (0,0)++(60:9)++(4,0) ++(60:5)++(120:1) circle (3pt);
\draw[black ,very thick]  (0,0)++(60:9)++(4,0) ++(60:5)++(120:1)  circle(3pt) ;

\filldraw [white] (0,0)++(60:9)++(4,0) ++(60:3)++(120:2) circle (3pt);
\draw[black ,very thick]  (0,0)++(60:9)++(4,0) ++(60:3)++(120:2)  circle(3pt) ;

\filldraw [white] (0,0)++(60:9)++(4,0) ++(60:1)++(120:3) circle (3pt);
\draw[black ,very thick]  (0,0)++(60:9)++(4,0) ++(60:1)++(120:3)  circle(3pt) ;

\filldraw [white] (0,0)++(60:9)++(4,0) ++(60:2)++(120:1) circle (3pt);
\draw[black ,very thick]  (0,0)++(60:9)++(4,0) ++(60:2)++(120:1)  circle(3pt) ;

\filldraw [white] (0,0)++(60:9)++(4,0) ++(120:2) circle (3pt);
\draw[black ,very thick]  (0,0)++(60:9)++(4,0) ++(120:2)  circle(3pt) ;

\end{tikzpicture}
\caption{Tiling the dominant part of a lower interval by paper boats.}
\label{fig:introA}
\end{subfigure}
\hfill
\begin{subfigure}{0.3\textwidth}
\centering
\begin{tikzpicture}[x=0.75cm, y=0.75cm,>=Latex]
% el origen esta en  (0,0)++(60:9)++(4,0) 

%\clip (-0.5,-1.4) rectangle + (16.3,18.5);

\clip (3,6) rectangle + (11,12);

% \filldraw[lightgray] (0,0)++(60:9)++(4,0) ++ (60:7)++(120:3)--++ (240:1) --++ (-1,0)--++ (240:1) --++ (-1,0) --++ (240:1) --++ (-1,0) --++ (240:1) --++ (-1,0) --++ (240:1) --++ (-1,0) --++ (240:1) --++ (-1,0) --++ (240:1) --++ (-1,0) --++ (-60:1) --++ (240:1)--++ (-60:1) --++ (240:1)--++ (-60:1) --++ (240:1)--++ (1,0) --++ (-60:1) --++ (1,0) --++ (-60:1) --++ (1,0) --++ (-60:1) --++ (1,0) --++ (-60:1) --++ (1,0) --++ (-60:1) --++ (1,0) --++ (-60:1)  --++ (1,0) --++ (-60:1) --++ 
% (60:1) --++ (1,0) --++ 
% (60:1) --++ (1,0) --++ 
% (60:1) --++ (1,0) --++
% (120:1) --++ (60:1) --++
% (120:1) --++ (60:1) --++
% (120:1) --++ (60:1) --++
% (120:1) --++ (60:1) --++
% (120:1) --++ (60:1) --++
% (120:1) --++ (60:1) --++
% (120:1) --++ (60:1) --++
% (-1,0)  --++ (120:1) --++
% (-1,0)  --++ (120:1) --++
% (-1,0)  --++ (120:1) --++
% (-1,0)  
% -- cycle;

\filldraw[green]  (0,0)++(60:9)++(4,0)++(60:7)++(120:1)--++(120:1)--++(-1,0)--++(240:1)--++(-1,0)--++(240:1)--++(-1,0)--++(240:1)--++(-60:1)--++(1,0)--++(-60:1)--++(1,0)--++(60:1)--++(120:1)--++(60:1)--++(1,0)
--cycle;

\filldraw[red]  (0,0)++(60:9)++(4,0)++(60:2)--++(60:1)--++(120:1)--++(-1,0)--++(240:1)--++(-60:1)--++(1,0)
--cycle;
\filldraw[red]  (0,0)++(60:9)++(4,0)++(60:5)--++(60:1)--++(120:1)--++(-1,0)--++(240:1)--++(-60:1)--++(1,0)
--cycle;

\filldraw[blue]  (0,0)++(60:9)++(4,0)--++(60:2)--++(-1,0)--++(240:1)--++(-60:1)
--cycle;

\filldraw[orange]  (0,0)++(60:9)++(4,0)++(120:1)--++(60:1)--++(120:1)--++(-1,0)
--cycle;

\filldraw[orange]  (0,0)++(60:9)++(4,0)++(120:4)--++(60:1)--++(120:1)--++(-1,0)
--cycle;

\filldraw[magenta]  (0,0)++(60:9)++(4,0)--++(120:3)--++(1,0)--++(60:1)--++(120:1)--++(-1,0)--++(240:1)
--cycle;

\filldraw[cyan]  (0,0)++(60:9)++(4,0)++(60:5)--++(-1,0)--++(240:1)--++(-60:1)
--cycle;

\filldraw[cyan]  (0,0)++(60:9)++(4,0)++(60:8)--++(-1,0)--++(240:1)--++(-60:1)
--cycle;

% \filldraw[cyan]  (0,0)++(60:9)++(4,0)++(60:7)--++(60:1)--++(120:1)--++(240:1)--++(-1,0)--++(-60:1)
% --cycle;

\foreach \i in {0,...,8} {
     \foreach \j in {-3,...,12}{
  \draw[gray, thick](0:0)++(1+3*\i,0)++(60:\j)++(120:\j)--++(60:1)--++(120:1)--++(-1,0)--++(240:1)--++(300:1)--++(1,0);
  }}
  \foreach \i in {0,...,5} {
     \foreach \j in {-3,...,12} {
  \draw[gray, thick](0:0)++(60:1)++(1,0)++(1+3*\i,0)++(60:\j)++(120:\j)--++(60:1)--++(120:1)--++(-1,0)--++(240:1)--++(300:1)--++(1,0);
  }}
\foreach \i in {0,...,6} {
     \foreach \j in {-3,...,12}{
  \draw[gray, thick](0:0)++(60:1)++(1+3*\i,0)++(60:\j)++(120:\j)--++(60:1)--++(120:1)--++(-1,0)--++(240:1)--++(300:1)--++(1,0);
  }}  
  \foreach \i in {0,...,5} {
     \foreach \j in {-3,...,12} {
  \draw[gray, thick](0:0)++(2,0)++(1+3*\i,0)++(60:\j)++(120:\j)--++(60:1)--++(120:1)--++(-1,0)--++(240:1)--++(300:1)--++(1,0);
  }}
\foreach \i in {0,...,6} {
     \foreach \j in {-3,...,12}{
  \draw[gray, thick](0:0)++(-1,0)++(60:1)++(1+3*\i,0)++(60:\j)++(120:\j)--++(60:1)--++(120:1)--++(-1,0)--++(240:1)--++(300:1)--++(1,0);
  }}  
  \foreach \i in {0,...,5} {
     \foreach \j in {-3,...,12}{
  \draw[gray, thick](0:0)++(1,0)++(1+3*\i,0)++(60:\j)++(120:\j)--++(60:1)--++(120:1)--++(-1,0)--++(240:1)--++(300:1)--++(1,0);
  }}

\draw[black, ultra thick]  (0,0)++(60:9)++(4,0)++(-9,0)--++(17,0);
\draw[black, ultra thick]  (0,0)++(60:9)++(4,0)++(120:11)--++(300:18);
\draw[black, ultra thick]  (0,0)++(60:9)++(4,0)++(60:11)--++(240:18);
\draw[black, ultra thick]  (0,0)++(60:9)++(4,0)++(60:1)++(120:1)--++(120:1);
\draw[black, ultra thick]  (0,0)++(60:9)++(4,0)++(60:1)++(120:4)--++(120:1);
\draw[black, ultra thick]  (0,0)++(60:9)++(4,0)++(60:2)++(120:2)--++(120:1);
\draw[black, ultra thick]  (0,0)++(60:9)++(4,0)++(60:3)++(120:0)--++(120:1);
\draw[black, ultra thick]  (0,0)++(60:9)++(4,0)++(60:3)++(120:3)--++(120:1);
\draw[black, ultra thick]  (0,0)++(60:9)++(4,0)++(60:4)++(120:1)--++(120:1);
\draw[black, ultra thick]  (0,0)++(60:9)++(4,0)++(60:5)++(120:2)--++(120:1);
\draw[black, ultra thick]  (0,0)++(60:9)++(4,0)++(60:6)++(120:1)--++(60:1);

\draw[black, ultra thick] (0,0)++(60:9)++(4,0) ++(60:8)--++(-1,0)--++(120:1)--++(-1,0)--++(240:1)--++(-1,0)--++(240:1)--++(-1,0) --++(240:1)--++(-1,0);  

\draw[black, ultra thick] (0,0)++(60:9)++(4,0) ++(60:6)--++(120:1)--++(-1,0)--++(240:1)--++(-1,0)--++(240:1)--++(-1,0) --++(240:1); 

\draw[black, ultra thick] (0,0)++(60:9)++(4,0) ++(60:5)--++(-1,0)--++(240:1)--++(-1,0)--++(240:1)--++(-1,0); 
\draw[black, ultra thick] (0,0)++(60:9)++(4,0) ++(60:2)--++(-1,0)--++(240:1);

%  PUNTOS
\filldraw [white] (0,0)++(60:9)++(4,0) ++ (60:6)++(120:2)  ++ (-60:0) circle (3pt);
\draw[black ,very thick]  (0,0)++(60:9)++(4,0) ++ (60:6)++(120:2) ++ (-60:0)  circle(3pt) ;

\filldraw [white] (0,0)++(60:9)++(4,0) ++ (60:4) ++ (120:3) circle (3pt);
\draw[black ,very thick]  (0,0)++(60:9)++(4,0) ++ (60:4) ++ (120:3)  circle(3pt) ;

\filldraw [white] (0,0)++(60:9)++(4,0) ++ (60:7) circle (3pt);
\draw[black ,very thick]  (0,0)++(60:9)++(4,0) ++ (60:7)  circle(3pt) ;

\filldraw [white] (0,0)++(60:9)++(4,0) ++ (60:4) circle (3pt);
\draw[black ,very thick]  (0,0)++(60:9)++(4,0) ++ (60:4)  circle(3pt) ;

\filldraw [white] (0,0)++(60:9)++(4,0) ++ (60:1) circle (3pt);
\draw[black ,very thick]  (0,0)++(60:9)++(4,0) ++ (60:1)  circle(3pt) ;

\filldraw [white] (0,0)++(60:9)++(4,0) ++ (60:2)++(120:4) circle (3pt);
\draw[black ,very thick]  (0,0)++(60:9)++(4,0) ++ (60:2)++(120:4)  circle(3pt) ;

\filldraw [white] (0,0)++(60:9)++(4,0) ++(120:5) circle (3pt);
\draw[black ,very thick]  (0,0)++(60:9)++(4,0) ++(120:5)  circle(3pt) ;

\filldraw [white] (0,0)++(60:9)++(4,0) ++(60:5)++(120:1) circle (3pt);
\draw[black ,very thick]  (0,0)++(60:9)++(4,0) ++(60:5)++(120:1)  circle(3pt) ;

\filldraw [white] (0,0)++(60:9)++(4,0) ++(60:3)++(120:2) circle (3pt);
\draw[black ,very thick]  (0,0)++(60:9)++(4,0) ++(60:3)++(120:2)  circle(3pt) ;

\filldraw [white] (0,0)++(60:9)++(4,0) ++(60:1)++(120:3) circle (3pt);
\draw[black ,very thick]  (0,0)++(60:9)++(4,0) ++(60:1)++(120:3)  circle(3pt) ;

\filldraw [white] (0,0)++(60:9)++(4,0) ++(60:2)++(120:1) circle (3pt);
\draw[black ,very thick]  (0,0)++(60:9)++(4,0) ++(60:2)++(120:1)  circle(3pt) ;

\filldraw [white] (0,0)++(60:9)++(4,0) ++(120:2) circle (3pt);
\draw[black ,very thick]  (0,0)++(60:9)++(4,0) ++(120:2)  circle(3pt) ;

\end{tikzpicture}
\caption{Tiling the dominant part of a lower interval by paper boats.}
\label{fig:introB}
\end{subfigure}
\hfill
\begin{subfigure}{0.3\textwidth}
\centering
\begin{tikzpicture}[x=0.75cm, y=0.75cm,>=Latex]
% el origen esta en  (0,0)++(60:9)++(4,0) 

%\clip (-0.5,-1.4) rectangle + (16.3,18.5);
\clip (3,6) rectangle + (11,12);

% \filldraw[lightgray] (0,0)++(60:9)++(4,0) ++ (60:7)++(120:3)--++ (240:1) --++ (-1,0)--++ (240:1) --++ (-1,0) --++ (240:1) --++ (-1,0) --++ (240:1) --++ (-1,0) --++ (240:1) --++ (-1,0) --++ (240:1) --++ (-1,0) --++ (240:1) --++ (-1,0) --++ (-60:1) --++ (240:1)--++ (-60:1) --++ (240:1)--++ (-60:1) --++ (240:1)--++ (1,0) --++ (-60:1) --++ (1,0) --++ (-60:1) --++ (1,0) --++ (-60:1) --++ (1,0) --++ (-60:1) --++ (1,0) --++ (-60:1) --++ (1,0) --++ (-60:1)  --++ (1,0) --++ (-60:1) --++ 
% (60:1) --++ (1,0) --++ 
% (60:1) --++ (1,0) --++ 
% (60:1) --++ (1,0) --++
% (120:1) --++ (60:1) --++
% (120:1) --++ (60:1) --++
% (120:1) --++ (60:1) --++
% (120:1) --++ (60:1) --++
% (120:1) --++ (60:1) --++
% (120:1) --++ (60:1) --++
% (120:1) --++ (60:1) --++
% (-1,0)  --++ (120:1) --++
% (-1,0)  --++ (120:1) --++
% (-1,0)  --++ (120:1) --++
% (-1,0)  
% -- cycle;

\foreach \i in {0,...,8} {
     \foreach \j in {-3,...,12}{
  \draw[gray, thick](0:0)++(1+3*\i,0)++(60:\j)++(120:\j)--++(60:1)--++(120:1)--++(-1,0)--++(240:1)--++(300:1)--++(1,0);
  }}
  \foreach \i in {0,...,5} {
     \foreach \j in {-3,...,12} {
  \draw[gray, thick](0:0)++(60:1)++(1,0)++(1+3*\i,0)++(60:\j)++(120:\j)--++(60:1)--++(120:1)--++(-1,0)--++(240:1)--++(300:1)--++(1,0);
  }}
\foreach \i in {0,...,6} {
     \foreach \j in {-3,...,12}{
  \draw[gray, thick](0:0)++(60:1)++(1+3*\i,0)++(60:\j)++(120:\j)--++(60:1)--++(120:1)--++(-1,0)--++(240:1)--++(300:1)--++(1,0);
  }}  
  \foreach \i in {0,...,5} {
     \foreach \j in {-3,...,12} {
  \draw[gray, thick](0:0)++(2,0)++(1+3*\i,0)++(60:\j)++(120:\j)--++(60:1)--++(120:1)--++(-1,0)--++(240:1)--++(300:1)--++(1,0);
  }}
\foreach \i in {0,...,6} {
     \foreach \j in {-3,...,12}{
  \draw[gray, thick](0:0)++(-1,0)++(60:1)++(1+3*\i,0)++(60:\j)++(120:\j)--++(60:1)--++(120:1)--++(-1,0)--++(240:1)--++(300:1)--++(1,0);
  }}  
  \foreach \i in {0,...,5} {
     \foreach \j in {-3,...,12}{
  \draw[gray, thick](0:0)++(1,0)++(1+3*\i,0)++(60:\j)++(120:\j)--++(60:1)--++(120:1)--++(-1,0)--++(240:1)--++(300:1)--++(1,0);
  }}

\draw[black, ultra thick]  (0,0)++(60:9)++(4,0)++(-9,0)--++(17,0);
\draw[black, ultra thick]  (0,0)++(60:9)++(4,0)++(120:11)--++(300:18);
\draw[black, ultra thick]  (0,0)++(60:9)++(4,0)++(60:11)--++(240:18);

\draw[black, ultra thick] (0,0)++(60:9)++(4,0)++(60:7)--++(150:1.8)--++(210:5.16)--++(-60:5.055)--++(60:7);

%  PUNTOS
\filldraw [green] (0,0)++(60:9)++(4,0) ++ (60:6)++(120:2)  ++ (-60:0) circle (5pt);
\draw[black ,very thick]  (0,0)++(60:9)++(4,0) ++ (60:6)++(120:2) ++ (-60:0)  circle(5pt) ;

\filldraw [green] (0,0)++(60:9)++(4,0) ++ (60:4) ++ (120:3) circle (5pt);
\draw[black ,very thick]  (0,0)++(60:9)++(4,0) ++ (60:4) ++ (120:3)  circle(5pt) ;

\filldraw [cyan] (0,0)++(60:9)++(4,0) ++ (60:7) circle (5pt);
\draw[black ,very thick]  (0,0)++(60:9)++(4,0) ++ (60:7)  circle(5pt) ;

\filldraw [cyan] (0,0)++(60:9)++(4,0) ++ (60:4) circle (5pt);
\draw[black ,very thick]  (0,0)++(60:9)++(4,0) ++ (60:4)  circle(5pt) ;

\filldraw [blue] (0,0)++(60:9)++(4,0) ++ (60:1) circle (5pt);
\draw[black ,very thick]  (0,0)++(60:9)++(4,0) ++ (60:1)  circle(5pt) ;

\filldraw [green] (0,0)++(60:9)++(4,0) ++ (60:2)++(120:4) circle (5pt);
\draw[black ,very thick]  (0,0)++(60:9)++(4,0) ++ (60:2)++(120:4)  circle(5pt) ;

\filldraw [orange] (0,0)++(60:9)++(4,0) ++(120:5) circle (5pt);
\draw[black ,very thick]  (0,0)++(60:9)++(4,0) ++(120:5)  circle(5pt) ;

\filldraw [red] (0,0)++(60:9)++(4,0) ++(60:5)++(120:1) circle (5pt);
\draw[black ,very thick]  (0,0)++(60:9)++(4,0) ++(60:5)++(120:1)  circle(5pt) ;

\filldraw [green] (0,0)++(60:9)++(4,0) ++(60:3)++(120:2) circle (5pt);
\draw[black ,very thick]  (0,0)++(60:9)++(4,0) ++(60:3)++(120:2)  circle(5pt) ;

\filldraw [magenta] (0,0)++(60:9)++(4,0) ++(60:1)++(120:3) circle (5pt);
\draw[black ,very thick]  (0,0)++(60:9)++(4,0) ++(60:1)++(120:3)  circle(5pt) ;

\filldraw [red] (0,0)++(60:9)++(4,0) ++(60:2)++(120:1) circle (5pt);
\draw[black ,very thick]  (0,0)++(60:9)++(4,0) ++(60:2)++(120:1)  circle(5pt) ;

\filldraw [orange] (0,0)++(60:9)++(4,0) ++(120:2) circle (5pt);
\draw[black ,very thick]  (0,0)++(60:9)++(4,0) ++(120:2)  circle(5pt) ;

\end{tikzpicture}
\caption{Polyhedra associated to the tilings and lattice points.}
\label{fig:introC}
\end{subfigure}
\caption{Some examples of tilings of lower Bruhat intervals}
\label{fig:Intro2}
\end{figure}

Let $\mathfrak{B}$ denote the dominant elements, \ie, those whose alcoves belong to the dominant chamber, and let $\mathfrak{C}$ represent the lowest two-sided Kazhdan-Lusztig cell. 
These sets satisfy the inclusion $\varfrak{A} \subset \mathfrak{B} \subset \mathfrak{C}$ and are depicted in \Cref{fig:introD} for the affine Weyl group of type $A_2$. 
The set $\varfrak{A}$ corresponds to the green triangles, $\mathfrak{B}$ corresponds to the green and orange triangles, and $\mathfrak{C}$ corresponds to the green, orange and gray triangles. 
The red triangle corresponds to the identity element. 
It does not belong to any of the sets $\varfrak{A}$,   $\mathfrak{B}$ and $\mathfrak{C}$.

\begin{figure}[H]
\captionsetup[subfigure]{labelformat=simple}
\centering

\begin{subfigure}{0.4\textwidth}
\centering
\begin{tikzpicture}[x=0.75cm, y=0.75cm,>=Latex]
% el origen esta en  (0,0)++(60:9)++(4,0) 

\clip (3,2) rectangle + (11,11);

\filldraw[lightgray] (0,0)++(60:9)++(4,0) ++ 
(60:11)--++(-11,0)--++(-60:11)  
-- cycle;

\filldraw[lightgray] (0,0)++(60:9)++(4,0) ++ 
(-1,0) --++(120:11)--++(-3,0)--++(0,-9.5)  
-- cycle;

\filldraw[lightgray] (0,0)++(60:9)++(4,0) ++ 
(240:1) --++(-9,0)--++(0,-10)--++(3.2,0)  
-- cycle;

\filldraw[lightgray] (0,0)++(60:9)++(4,0) ++ 
(240:1) ++ (-60:1)--++(240:9)--++(9,0)--++(120:9)  
-- cycle;

\filldraw[lightgray] (0,0)++(60:9)++(4,0) ++ 
(-60:1)--++(-60:10)--++(3,0)--++(0,8.7)--++(9,0)  
-- cycle;

\filldraw[lightgray] (0,0)++(60:9)++(4,0) ++ 
(1,0)--++(7,0)--++(0,10)--++(-1.2,0)  
-- cycle;

\filldraw[orange] (0,0)++(60:9)++(4,0) ++ 
(60:11)--++(-11,0)--++(-60:11)  
-- cycle;

% \filldraw[green] (0,0)++(60:9)++(4,0) ++ 
% (60:1)--++(-1,0)--++(-60:1)  
% -- cycle;

\foreach \i in {1,...,11} {
     \foreach \j in {1,...,\i}{
\filldraw[green] (0,0)++(60:9)++(4,0) ++ (1,0)++
(60:\i)++(-\j,0)--++(-1,0)--++(-60:1)  
-- cycle;
}}

\filldraw[red] (0,0)++(60:9)++(4,0) ++ 
(240:1)--++(1,0)--++(120:1)  
-- cycle;

\foreach \i in {0,...,8} {
     \foreach \j in {-3,...,12}{
  \draw[gray, thick](0:0)++(1+3*\i,0)++(60:\j)++(120:\j)--++(60:1)--++(120:1)--++(-1,0)--++(240:1)--++(300:1)--++(1,0);
  }}
  \foreach \i in {0,...,5} {
     \foreach \j in {-3,...,12} {
  \draw[gray, thick](0:0)++(60:1)++(1,0)++(1+3*\i,0)++(60:\j)++(120:\j)--++(60:1)--++(120:1)--++(-1,0)--++(240:1)--++(300:1)--++(1,0);
  }}
\foreach \i in {0,...,6} {
     \foreach \j in {-3,...,12}{
  \draw[gray, thick](0:0)++(60:1)++(1+3*\i,0)++(60:\j)++(120:\j)--++(60:1)--++(120:1)--++(-1,0)--++(240:1)--++(300:1)--++(1,0);
  }}  
  \foreach \i in {0,...,5} {
     \foreach \j in {-3,...,12} {
  \draw[gray, thick](0:0)++(2,0)++(1+3*\i,0)++(60:\j)++(120:\j)--++(60:1)--++(120:1)--++(-1,0)--++(240:1)--++(300:1)--++(1,0);
  }}
\foreach \i in {0,...,6} {
     \foreach \j in {-3,...,12}{
  \draw[gray, thick](0:0)++(-1,0)++(60:1)++(1+3*\i,0)++(60:\j)++(120:\j)--++(60:1)--++(120:1)--++(-1,0)--++(240:1)--++(300:1)--++(1,0);
  }}  
  \foreach \i in {0,...,5} {
     \foreach \j in {-3,...,12}{
  \draw[gray, thick](0:0)++(1,0)++(1+3*\i,0)++(60:\j)++(120:\j)--++(60:1)--++(120:1)--++(-1,0)--++(240:1)--++(300:1)--++(1,0);
  }}

\draw[black, ultra thick]  (0,0)++(60:9)++(4,0)++(-9,0)--++(17,0);
\draw[black, ultra thick]  (0,0)++(60:9)++(4,0)++(120:11)--++(300:22);
\draw[black, ultra thick]  (0,0)++(60:9)++(4,0)++(60:11)--++(240:22);

\end{tikzpicture}
\caption{Illustration of the sets $\varfrak{A}$,   $\mathfrak{B}$ and $\mathfrak{C}$ for the affine Weyl group of type $A_2$.}
\label{fig:introD}
\end{subfigure}
\hfill
\begin{subfigure}{0.4\textwidth}
\centering
\begin{tikzpicture}[x=0.75cm, y=0.75cm,>=Latex]
% el origen esta en  (0,0)++(60:9)++(4,0) 

\clip (3,2) rectangle + (11,11);

\filldraw[lightgray] (0,0)++(60:9)++(4,0) ++ 
(60:11)--++(-11,0)--++(-60:11)  
-- cycle;

\filldraw[red] (0,0)++(60:9)++(4,0) ++ 
(240:1)--++(1,0)--++(120:1)  
-- cycle;

\filldraw[cyan] (0,0)++(60:9)++(4,0) --++ 
(60:1)--++(120:1)--++(240:1)  
-- cycle;

\foreach \i in {0,...,8} {
     \foreach \j in {-3,...,12}{
  \draw[gray, thick](0:0)++(1+3*\i,0)++(60:\j)++(120:\j)--++(60:1)--++(120:1)--++(-1,0)--++(240:1)--++(300:1)--++(1,0);
  }}
  \foreach \i in {0,...,5} {
     \foreach \j in {-3,...,12} {
  \draw[gray, thick](0:0)++(60:1)++(1,0)++(1+3*\i,0)++(60:\j)++(120:\j)--++(60:1)--++(120:1)--++(-1,0)--++(240:1)--++(300:1)--++(1,0);
  }}
\foreach \i in {0,...,6} {
     \foreach \j in {-3,...,12}{
  \draw[gray, thick](0:0)++(60:1)++(1+3*\i,0)++(60:\j)++(120:\j)--++(60:1)--++(120:1)--++(-1,0)--++(240:1)--++(300:1)--++(1,0);
  }}  
  \foreach \i in {0,...,5} {
     \foreach \j in {-3,...,12} {
  \draw[gray, thick](0:0)++(2,0)++(1+3*\i,0)++(60:\j)++(120:\j)--++(60:1)--++(120:1)--++(-1,0)--++(240:1)--++(300:1)--++(1,0);
  }}
\foreach \i in {0,...,6} {
     \foreach \j in {-3,...,12}{
  \draw[gray, thick](0:0)++(-1,0)++(60:1)++(1+3*\i,0)++(60:\j)++(120:\j)--++(60:1)--++(120:1)--++(-1,0)--++(240:1)--++(300:1)--++(1,0);
  }}  
  \foreach \i in {0,...,5} {
     \foreach \j in {-3,...,12}{
  \draw[gray, thick](0:0)++(1,0)++(1+3*\i,0)++(60:\j)++(120:\j)--++(60:1)--++(120:1)--++(-1,0)--++(240:1)--++(300:1)--++(1,0);
  }}

\draw[black, ultra thick]  (0,0)++(60:9)++(4,0)++(-9,0)--++(17,0);
\draw[black, ultra thick]  (0,0)++(60:9)++(4,0)++(120:11)--++(300:22);
\draw[black, ultra thick]  (0,0)++(60:9)++(4,0)++(60:11)--++(240:22);

\foreach \i in {1,...,5} {
\draw[black, ultra thick](0,0)++(60:9)++(4,0)++(60:\i)--++(120:6);
}

\foreach \i in {1,...,5} {
\draw[black, ultra thick](0,0)++(60:9)++(4,0)++(120:\i)--++(60:6);
}

\filldraw [fill=green] (0,0)++(60:9)++(4,0) ++ (60:0)++(120:0) circle (3pt);
\filldraw [fill=green] (0,0)++(60:9)++(4,0) ++ (60:0)++(120:1) circle (3pt);
\filldraw [fill=green] (0,0)++(60:9)++(4,0) ++ (60:0)++(120:2) circle (3pt);
\filldraw [fill=green] (0,0)++(60:9)++(4,0) ++ (60:0)++(120:3) circle (3pt);
\filldraw [fill=green] (0,0)++(60:9)++(4,0) ++ (60:0)++(120:4) circle (3pt);
\filldraw [fill=green] (0,0)++(60:9)++(4,0) ++ (60:0)++(120:5) circle (3pt);

\filldraw [fill=green] (0,0)++(60:9)++(4,0) ++ (60:1)++(120:0) circle (3pt);
\filldraw [fill=green] (0,0)++(60:9)++(4,0) ++ (60:1)++(120:1) circle (3pt);
\filldraw [fill=green] (0,0)++(60:9)++(4,0) ++ (60:1)++(120:2) circle (3pt);
\filldraw [fill=green] (0,0)++(60:9)++(4,0) ++ (60:1)++(120:3) circle (3pt);
\filldraw [fill=green] (0,0)++(60:9)++(4,0) ++ (60:1)++(120:4) circle (3pt);

\filldraw [fill=green] (0,0)++(60:9)++(4,0) ++ (60:2)++(120:0) circle (3pt);
\filldraw [fill=green] (0,0)++(60:9)++(4,0) ++ (60:2)++(120:1) circle (3pt);
\filldraw [fill=green] (0,0)++(60:9)++(4,0) ++ (60:2)++(120:2) circle (3pt);
\filldraw [fill=green] (0,0)++(60:9)++(4,0) ++ (60:2)++(120:3) circle (3pt);

\filldraw [fill=green] (0,0)++(60:9)++(4,0) ++ (60:3)++(120:0) circle (3pt);
\filldraw [fill=green] (0,0)++(60:9)++(4,0) ++ (60:3)++(120:1) circle (3pt);
\filldraw [fill=green] (0,0)++(60:9)++(4,0) ++ (60:3)++(120:2) circle (3pt);

\filldraw [fill=green] (0,0)++(60:9)++(4,0) ++ (60:4)++(120:0) circle (3pt);
\filldraw [fill=green] (0,0)++(60:9)++(4,0) ++ (60:4)++(120:1) circle (3pt);

\filldraw [fill=green] (0,0)++(60:9)++(4,0) ++ (60:5)++(120:0) circle (3pt);

\end{tikzpicture}
\caption{Tiling the dominant region by translations of the fundamental parallelepiped.}
\label{fig:introE}
\end{subfigure}
\caption{}
\label{fig:Intro3}
\end{figure}

In this paper, we propose a conjectural formula (see Equation \eqref{eq: lowerint implied by conjs}) for the size of lower Bruhat intervals associated with elements in $\mathfrak{C}$ and prove it in type $A$ for elements in $\mathfrak{B}$. 
To illustrate the relative sizes of these sets in type $A_n$, consider selecting a ``random'' element from the affine Weyl group. 
The probability of such an element belonging to $\varfrak{A}$, $\mathfrak{B}$, and $\mathfrak{C}$ is 
$$\frac{1}{n!(n+1)!}, \quad \frac{1}{(n+1)!} \quad \mbox{and} \qquad 1, $$
respectively.
Thus, our conjecture effectively provides an answer for almost the entirety of the affine Weyl group, and our proof for $\mathfrak{B}$ extends our prior results \cite{castillo2023size} to encompass a set that is $n!$ times larger in type $A_n$.

Another interesting aspect of this work is the conjecture (partially established in this paper) that the resulting formulas, when restricted to some specific sets, are polynomials. 
This is a scenario that exceeds expectations in both simplicity and elegance. 
To elaborate briefly, let $\Pi^+$ denote the fundamental parallelepiped spanned by the fundamental coweights. 
In \eqref{eq: param c plus}, we prove  that $\mathfrak{B}$  can be tiled using the set $\{\Pi^+ +\lambda\mid\lambda\in (\Lambda^{\vee})^+  \}$, where $(\Lambda^{\vee})^+$ is the set of dominant coweights. 
This tessellation is illustrated in \Cref{fig:introE}, where $\Pi^+$ is highlighted in light blue and the points corresponds to the elements in $(\Lambda^{\vee})^+$. 
The aforementioned polynomiality phenomenon asserts that, for a fixed alcove  $x\subset \Pi^+$, the size of the Bruhat interval $[\id,x+\lambda]$ is  a polynomial in the coordinates of $\lambda\in (\Lambda^{\vee})^+$ expressed in the basis of fundamental coweights. 

Three observations are worth highlighting.  
First, within each tile $\Pi^+ +\lambda$ there is exactly one element of $\varfrak{A}$, so this polynomiality phenomenon extends that of  
\cite{castillo2023size}. 
Second, the emergence of polynomiality from this tiling provides yet another instance of the deep connection between alcove geometry (via the Bruhat order) and Euclidean geometry. 
Finally, and in alignment with the previous observation, we have verified computationally---quite unexpectedly---that all polynomials that arise in this context, even when extended to the full cell $\mathfrak{C}$, are \emph{geometric polynomials}. 
That is, they are linear combinations (with coefficients that do not depend on $\lambda$) of the volumes of the faces of the permutohedron $\mathsf{P}(\lambda)$, which is defined as the convex hull of the $W_\mathrm{f}$-orbit of $\lambda$. 
This was proven for $\varfrak{A}$ in \cite{castillo2023size}, and this observation strongly suggests that the tessellation using $\Pi^+$ introduced here is the natural and correct framework for this study.
Let us briefly explain the 
main contributions of this paper. 

\subsection*{Bruhat order criterion in alcovic terms}
Elements in $\varfrak{A}$ have nice algebraic properties that arbitrary elements do not share; see \Cref{lemma: thetas are maximal}. 
As a consequence \cite[Eq. (3.3)]{castillo2023size}, their lower Bruhat intervals can essentially be reduced to ideals in the dominance order.
The latter is significantly more accessible than the former, which simplifies matters considerably. 
To extend our results beyond $\varfrak{A}$, we require a criterion for determining whether $x \leq y$, not merely as abstract elements of the affine Weyl group, but specifically as alcoves.
We establish a criterion for the dominant region in type $ A $: comparing alcoves reduces to comparing their corresponding vertices with respect to the dominance order. 
More precisely we have the following proposition. 

\begin{prop}\label{prop:order}
Let $W_\mathrm{aff}$ be an affine Weyl group of type $A$, $ V $ the set of vertices of the fundamental alcove, and
$ x, y \in \mathfrak{B} $.
Then,
\begin{equation*}
x \leq y  \qquad \mbox{if and only if} \qquad x(v) \leq y(v), \ \forall \ v\in V,
\end{equation*}
where the left-hand relation denotes the (strong) Bruhat order and the right-hand relation denotes the dominance order.  
\end{prop}

This is the content of \Cref{prop: main tool}. 
The criterion for the entire group was originally mentioned by Schremmer in \cite[\S 3]{schremmer2024affine}. 
We warmly thank him for sharing a sketch of the proof,  which we have expanded into the detailed version presented in 
\Cref{sec: criterio}.

\subsection*{Tiling lower intervals} 

To describe our next result, we first introduce some notation.  
For $ w \in W_\mathrm{aff}$, where $W_\mathrm{aff}$ is an affine Weyl group, we denote by $ \mathcal{A}_w $ the corresponding alcove.
Let $ (\Lambda^{\vee})^+ $ denote the set of dominant coweights.
We define  
\[
\mathcal{F} = \{ a \in W_\mathrm{aff} \mid \mathcal{A}_a \subset -\Pi^+ \},
\]
and for $\lambda\in (\Lambda^{\vee})^+$ and $a\in\mathcal{F}$, we let $ \theta_a(\lambda) $ be the unique element in $ W_\mathrm{aff} $ satisfying  
$ \mathcal{A}_{\theta_a(\lambda)} = \mathcal{A}_{aw_0} + \lambda$.

The region $ \mathfrak{B} $ (resp. $ \varfrak{A} $) consists of elements in $ W_\mathrm{aff} $ of the form $ \theta_a(\lambda) $ (resp. $ \theta_{\mathrm{id}}(\lambda) $).  
Consider the Bruhat intervals of the form  
\[
\mathcal{I}_a(\lambda) := [\mathrm{id},\theta_a(\lambda)] = \{ x \in W_\mathrm{aff} \mid x \leq \theta_a(\lambda) \}.
\]
The following definition introduces the central combinatorial object of this paper.  

\begin{definition}  
The \emph{Paper Boat} associated with $ \lambda\in (\Lambda^{\vee})^+ $ and $ a\in \mathcal{F} $ is defined as  
\begin{equation}
PB_a(\lambda) := \mathcal{I}_a(\lambda)\setminus \bigcup_{\mu \lessdot \lambda} \mathcal{I}_a(\mu),
\end{equation}
where $ \mu \lessdot \lambda $ denotes a covering relation in the poset $ (\Lambda^{\vee})^+ $ ordered by dominance.
\end{definition}  

A key observation is that lower intervals for elements in $ \mathfrak{B} $ are invariant under the left action of $ W_\mathrm{f} $. This allows us to focus on their dominant part. The fundamental building blocks of our tiling are the Paper Boats.  
This structure is illustrated in \Cref{fig:introA,fig:introB}.  

The next theorem establishes that Paper Boats provide a tessellation of lower Bruhat intervals for elements in $ \mathfrak{B} $.
For the rest of the Introduction, all results stated will be for the affine Weyl group $W_\mathrm{aff}$ of type $A_n$.

\begin{lettertheorem}\label{thmA}
Let $ \lambda $ be a dominant coweight, and let $ a \in \mathcal{F} $.  
Then we have the following decomposition:  
\begin{equation}\label{eq: tiling}
\mathcal{I}_a(\lambda) = \bigsqcup_{\mu \leq \lambda} PB_a(\mu),
\end{equation}
where the union runs over all $ \mu \in (\Lambda^{\vee})^+ $ satisfying $ \mu \leq \lambda $ in the dominance order.  
\end{lettertheorem}

In our previous work, we established a tiling \cite[Theorem A]{castillo2023size} for lower Bruhat intervals $\mathcal{I}_\mathrm{id}(\lambda)$ (\ie for elements in $\varfrak{A}$).
This tiling led to a formula for the size of these intervals due to two key properties.
First, among all intervals there is only one type of tile, consisting of the alcoves of the finite Weyl group $ W_\mathrm{f} $.
Second, for a given $\mathcal{I}_\mathrm{id}(\lambda)$, the number of tiles in the corresponding tiling corresponds to the number of lattice points in the permutohedron $\mathsf{P}(\lambda)$.

The tiling presented in \Cref{thmA}, this time for general elements in $\mathfrak{B}$, retains similar properties.
The first difference is that now there is more than one type of tile.
In \Cref{def: zones} we introduce a finite partition $\mathcal{Z}$ of the dominant coweights.
The following theorem shows that, among all possible tilings,  this new tiling features a finite number of distinct types of tiles.

\begin{lettertheorem}\label{thmB}
Let $a\in\mathcal{F}$ and $Z\in\mathcal{Z}$.
If $\lambda,\mu\in Z$, then the dominant alcoves in $PB_a(\lambda)$ are a translation of the dominant alcoves in $PB_a(\mu)$.
In particular, $|PB_a(\lambda)|=|PB_a(\mu)|.$
\end{lettertheorem}

It is worth noting that the definition of Paper Boats can be generalized  to all affine Weyl groups and to all elements of $ \mathfrak{C} $.  
We conjecture that there is a tiling as in \Cref{thmA} with a property as in \Cref{thmB}, at this level of generality (see Conjectures \ref{conj: lowerint as pbs} and \ref{conj: paperboat zones}).

\subsection*{Formulas for sizes of intervals.}

Combining \Cref{thmA} and \Cref{thmB}, we derive the following theorem and our main counting formula.
For a dominant coweight $\lambda$, let $\mathcal{L}_\lambda$ be the component of $\Lambda^\vee$ (modulo the coroot lattice) containing $\lambda$.
We let $\mathfrak{I}_\lambda$ be the dominant lattice points (within $\mathcal{L}_\lambda$) in the permutohedron $\mathsf{P}(\lambda)$, see \Cref{eq: ideal as latticepoints}.

\begin{lettertheorem}\label{thmC}
Let $\lambda\in(\Lambda^\vee)^+$ and let $a\in\mathcal{F}$.
Then
\begin{equation*}
|\li{\lambda}| = \sum_{Z\in\mathcal{Z}} c_a(Z) |\mathfrak{I}_\lambda \cap Z |,
\end{equation*}
where $c_a(Z)$ is $|PB_a(\mu)|$ for any $\mu\in Z$.
\end{lettertheorem}

As we mentioned above, the first difference between our previous tiling for elements in $\varfrak{A}$ \cite{castillo2023size} and our new tiling for elements in $\mathfrak{B}$ \eqref{eq: tiling}, is that we now have to deal with the sizes of Paper Boats $c_a(Z)$.
For example in \Cref{fig:introA} (resp. \Cref{fig:introB}), the blue tile at the bottom, for instance, has size five (resp. three).
The second difference is that, after a reinterpretation of this formula as a \emph{weighted lattice point count} (see \eqref{eq:weighted_lattice_count}), the number of tiles, of the same type, is equal to the number of lattice points in slices of permutohedra; whereas in \cite{castillo2023size} the relevant polytopes where always permutohedra.

\hfill

\noindent\textbf{Sizes of Paper Boats:}
This is an intriguing problem even in type $\widetilde{A}_n$.
In \Cref{prop:BarquitoGenerico} we obtain a closed formula for the size of Paper Boats $c_a(Z)$, when $Z$ is sufficiently deep inside the dominant chamber; their size is equal to $(n+1)!^2$, which does not depend on $a$.
This is particularly striking given that the shapes of Paper Boats can vary dramatically. 
See the green regions highlighted in both \Cref{fig:introA,fig:introB} for type $\widetilde{A}_2$ and  \Cref{fig:PPBB} for $\widetilde{A}_3$. 
More generally, we conjecture that, for any type, the size of Paper Boats associated with coweights sufficiently deep inside the dominant chamber equals $|W_\mathrm{f}|^2$.

\hfill

\noindent\textbf{Counting lattice points:}
In \Cref{sec: latticepoints}, we see each $\mathfrak{I}_\lambda\cap Z$ as the lattice points in convex polytopes.
These lattice points are depicted in \Cref{fig:introC}, where each color represents a different zone $Z\in\mathcal{Z}$.
One technical difficulty is that we have to consider polytopes whose vertices are not lattice points (\Cref{prop:vertices}) but rather rational points.
Using general results about counting lattice points in polyhedra \cite{barvinok2008integer}, we prove that the number of lattice points in a rational polytope agrees with a quasi-polynomial, leading to the following theorem.

\begin{lettertheorem}\label{thmD}
Let $\lambda\in(\Lambda^\vee)^+$ and $(\lambda_1,\ldots,\lambda_n)$ its coordinates in the basis of fundamental coweights.
For a fixed $ a \in \mathcal{F} $ the function $ q_{a}: \mathbb{N}^{n} \to \mathbb{N} $ defined as $( \lambda_{1}, \dots, \lambda_{n} )\mapsto | \mathcal{I}_{a}(\lambda)|$,
agrees with a \textbf{quasi-polynomial} of degree $ n $ on the subset $ \mathbb{N}_{\geq 2}^{n}$ where each coordinate is at least two.
Moreover, the top homogeneous component is a polynomial and it is independent of $ a $.
\end{lettertheorem}

A quasi-polynomial on $ \mathbb{Z}^{n} $ of degree $ d $ is a function that agrees with a polynomial function of degree $ d $ on cosets of a subgroup $ H $ of $ \mathbb{Z}^{n} $ of finite index, see \Cref{def:quasi_polynomial}.
The most basic example is the floor function $ \lfloor n/2 \rfloor $.

The quasi-polynomial $q_a$ is obtained as a linear combination of quasi-polynomials, and the coefficients in this combination are the sizes of different Paper Boats.
But \Cref{thmD} is not the end of the story.
After extensive computations, we have consistently found that $q_a$ is actually a polynomial.
This is quite unexpected as the quasi-polynomials composing $q_a$ are definitely not polynomials in general.
It suggests the presence of algebraic relations between the sizes $c_a(Z)$ to be discovered.

Furthermore, in every case our computations show that the polynomial is also geometric.
To be more precise, the polynomials are rational combinations of volume polynomials from faces of permutohedra.

\begin{conjecture}\label{superconj}
For any affine Weyl group  the map  $ q_{a} $ is a geometric polynomial. 
\end{conjecture}

We remark that the polinomiality and the geometricity parts of this conjecture are both difficult. 
Heuristically the polinomiality part suggests that $q_a$ is as simple as one could reasonably expect, while the geometricity part indicates that, among the full set of polynomials, these are exceptionally well-behaved and hint to a deeper connection with convex geometry.
In \Cref{thm: pols A3}, we verify this conjecture for type $\widetilde{A}_3$, and we also show it in type $\widetilde{A}_n$ for the top homogeneous part of $ q_{a} $ in \Cref{thm:lattice_count}.

We further conjecture an extension of \Cref{superconj} to lower intervals coming from elements in $\mathfrak{C}$ (rather than $\mathfrak{B}$).
See \Cref{conj: pols geom} for more details.
In this general form, it holds for types $\widetilde{A}_1$ (trivially), $\widetilde{A}_2$ (see \cite{LP}) and $\widetilde{B}_2$ (see \cite{batistelli2023kazhdan}).
We have not found a counterexample in types $\widetilde{A}_3$ and $\widetilde{B}_3$.

\hfill

\noindent\textbf{Structure of the paper.} \Cref{sec:preliminaries} contains a recollection of important facts about affine Weyl groups.
We also introduce a parametrization of the lowest two-sided Kazhdan-Lusztig cell and define the Paper Boat, our main object of study.
In \Cref{sec: main results} we present the proof of Theorems \ref{thmA}, \ref{thmB} and \ref{thmC} for type $\widetilde{A}_n$.
\Cref{sec:SizesPB} is devoted to the study the sizes of Paper Boats (also for type $\widetilde{A}_n$).
We compute the size of the most ``common'' Paper Boat, we give a recursive formula for the size of a Paper Boat in terms of lower Bruhat intervals and we give a criterion to test wether a given function matches the size of a lower interval.
We then use this criterion to establish \Cref{superconj} for type $\widetilde{A}_3$.
In \Cref{sec: latticepoints} we reinterpret \Cref{thmC} as a weighted lattice point count and invoke results from Ehrhart theory to obtain \Cref{thmD}.
Finally, \Cref{sec: criterio} contains the proof of \Cref{prop:order}, which is heavily used from \Cref{sec: main results} onward.

\hfill

\noindent\textbf{Acknowledgments.} 
We would like to thank Gaston Burrull and Felix Schremer for their helpful comments.
FC was partially supported by FONDECYT-ANID grant 1221133.
NL was partially supported by FONDECYT-ANID grant 1230247.
DP was partially supported by FONDECYT-ANID grant 1240199.
We thank the support of the MATH-AmSud program (Grant 24-MATH-01).

\section{The Paper Boat}\label{sec:preliminaries}
In this  section, we introduce the Paper Boat associated with any element in the lowest two-sided Kazhdan-Lusztig cell, which is the central object of study in this paper. To achieve this, we begin by recalling key facts about affine Weyl groups. Additionally, we present a parametrization for the elements within the lowest two-sided Kazhdan-Lusztig cell and some of its properties.

\subsection{Affine Weyl groups}\label{sec: affWeyl}

Let $\Phi$ be an irreducible root system of rank $n$, and let $E$ be the Euclidean space spanned by $\Phi$, with inner product $\langle -,-\rangle : E \times E \to \mathbb{R} $.
We denote by $\mathbf{0}$ the origin of $E$.
Let ${\Phi^+}$ be a choice of positive roots and $\Delta\subset\Phi^+$ the simple root basis. 
We denote by $\rho$ the half-sum of the positive roots.
Let ${\alpha^\vee=2\alpha/ \langle \alpha,\alpha \rangle }$ be the coroot corresponding to a root ${\alpha\in\Phi}$.
The set $\Phi^\vee =\{ \alpha^\vee \mid \alpha \in \Phi\}$ is also a root system.
The fundamental weights $\varpi_\alpha$ and fundamental coweights $\varpi_\alpha^\vee$ are defined by the equations ${\langle \varpi_\alpha,\beta^\vee \rangle =\langle \varpi_\alpha^\vee,\beta \rangle} =\delta_{\alpha\beta}$, for $\alpha,\beta\in\Delta$. 
Both sets form a basis of $E$.
We denote by $ \displaystyle \Lambda^\vee=\bigoplus_{\alpha\in\Delta}\bbZ\varpi_\alpha^\vee$ the set of coweights. 
We define
\begin{equation*}
C^+= \{ x \in E \mid\langle  x,\alpha \rangle \geq 0, \, \mbox{ for all }  \alpha\in\Delta \}.
\end{equation*}
We refer to ${(\Lambda^\vee)^+}\coloneqq \Lambda^\vee \cap C^+$ as the set of dominant coweights.

We denote by $\leq$ the dominance order on $\Lambda^\vee$, that is, ${\mu\leq\lambda}$ if the coordinates of ${\lambda-\mu}$ in the simple coroot basis $\{\alpha^\vee\}_{\alpha\in\Delta}$ are non-negative.

\begin{remark}
In this paper we use the symbol $\leq$ for different orders: for the standard ordering on real numbers, the Bruhat order on elements of the affine Weyl group, and for the dominance order on elements of the (co)weight lattice.
Since they are all on different objects, from context it should be clear which one we are using.
\end{remark}

For $\alpha \in \Phi$ and $k\in \mathbb{Z}$ we define the hyperplane
$H_{\alpha,k}=\{\lambda\in E\mid \langle \lambda,\alpha \rangle  =k\}.$
We denote by $s_{\alpha, k}$ the affine reflection along the hyperplane $H_{\alpha , k}$.
We write $H_\alpha=H_{\alpha,0}$ and $s_\alpha=s_{\alpha, 0}$.
The affine Weyl group $W_\mathrm{aff}$ of $\Phi$ is the group of affine transformations of $E$ generated by $s_{\alpha, k }$, for $\alpha \in \Phi$ and $k \in \mathbb{Z}$.

An alcove is a connected component of the space $E-\bigcup_{\alpha\in \Phi, k \in \mathbb{Z}  } H_{\alpha,k}. $
Let $\mathcal A$ be the set of all alcoves.
We define the fundamental alcove $\mathcal{A_{\text{id}}}\in \mathcal{A}$ by
\begin{equation*}
\mathcal{A}_\mathrm{id}=\{x\in E\mid -1<\langle x,\alpha \rangle <0, \ \forall\alpha\in\Phi^+\}. 
\end{equation*}
The walls of $\mathcal{A}_\mathrm{id}$ are the hyperplanes $\{  H_\alpha\mid \alpha\in\Delta \}$, together with $H_{\tilde{\alpha},-1}$, where $\tilde{\alpha}\in\Phi$ is the highest root.
We set ${s_0=s_{\tilde{\alpha},-1}}$ and ${S=\{s_0,s_\alpha\mid\alpha\in\Delta\}}$. 
Then the pair ${(W_\mathrm{aff},S)}$ is a Coxeter system, with length function $\ell$ and 
Bruhat order $\leq$.

The action of $W_\mathrm{aff}$ on $E$ gives a simply transitive action on $\mathcal{A}$.
Thus, we have a bijection $W_\mathrm{aff}\rightarrow  \mathcal{A}$ given by $w\mapsto\mathcal{A}_w=w(\mathcal{A}_\mathrm{id})$.
We define the vertices of an alcove $\mathcal{A}_w$ to be the vertices of its closure $\overline{\mathcal{A}_w}$, and denote this set by $V(w)$.
For $X\subset W_\mathrm{aff}$, we write
\begin{equation}
\mathcal{A}(X)=\bigcup_{x\in X}\overline{\mathcal{A}_x}.
\end{equation}
If $Y\subset E$ is such that $Y=\mathcal{A}(X)$, we write $\mathcal{A}^{-1}(Y)=X$.

We have a geometric interpretation of the length function.
Concretely, 
\begin{equation}\label{eq: geometric length}
\ell(w) = \# \{H=H_{\alpha,k} \mid H \mbox{ separates } \mathcal{A}_\mathrm{id} \mbox{ and } \mathcal{A}_w \}.
\end{equation}

The (finite) Weyl group $W_\mathrm{f}$ is the parabolic subgroup of $W_\mathrm{aff}$ generated by the set $S_\mathrm{f}=\{s_\alpha \mid \alpha \in \Delta \}$.
We denote by $w_0$ the longest element of $W_\mathrm{f}$.
The affine Weyl group can also be realized as a semi-direct product: ${W_\mathrm{aff}\cong \mathbb{Z}\Phi^\vee \rtimes W_\mathrm{f} }$, where the elements of $\mathbb{Z}\Phi^\vee$ act on $E$ as translations.
We define $W_\mathrm{aff}^+=\mathcal{A}^{-1}(C^+)$ and $W_\mathrm{aff}^-=\mathcal{A}^{-1}(C^-)$, where $C^-=-C^+$.
We have
\begin{equation}\label{eq: partition of affine}
W_\mathrm{aff}=\bigsqcup_{w\in W_\mathrm{f}}wW_\mathrm{aff}^+.
\end{equation}

Consider the extended affine Weyl group $W_\mathrm{e}=\Lambda^\vee \rtimes W_\mathrm{f}$ and the subgroup 
\begin{equation}
\Omega= \{  \sigma \in W_\mathrm{e} \mid \sigma(\mathcal{A}_\mathrm{id}) =\mathcal{A}_\mathrm{id}   \}.
\end{equation}

For $\sigma\in\Omega$   and $w\in W_\mathrm{aff}$ we write  $w_\sigma\coloneqq \sigma w \sigma^{-1}$.  We have $\sigma(\mathcal{A}_w)=\mathcal{A}_{w_\sigma}$.  For $X\subset W_\mathrm{aff}$ we  write $X_\sigma=\sigma X\sigma^{-1}\subset W_\mathrm{aff}$. 
Each $\sigma\in\Omega$  permutes the walls of $\mathcal{A}_\mathrm{id}$ and thus conjugation by $\sigma$ permutes the simple reflections $S$ of $W_\mathrm{aff}$,  \ie , $S_\sigma = S$.
Note that $(W_\mathrm{f})_\sigma\cong W_\mathrm{f}$ is the maximal parabolic subgroup of $W_\mathrm{aff}$ corresponding to  $(S_\mathrm{f})_\sigma=S\setminus\{(s_0)_\sigma\}$, whose longest element is $(w_0)_\sigma$.
In fact, all maximal parabolic subgroups of $W_\mathrm{aff}$ arise in this way.

It is well-known \cite[Prop~VI.2.3.6]{Bour46} that the set $\{-\sigma(\mathbf{0})\mid \sigma\in\Omega\}$ forms a complete system of representatives of $\Lambda^\vee/\bbZ\Phi^\vee$.
Hence, for $\lambda\in\Lambda^\vee$, there is a unique $\sigma_\lambda\in\Omega$ such that
\begin{equation}\label{eq: sigmaofcoweight}
\lambda \equiv -\sigma_\lambda(\mathbf{0}) \hspace{-0.2cm}\pmod{\bbZ\Phi^\vee}.
\end{equation}
Note that if $\lambda\in\bbZ\Phi^\vee$, then $\sigma_\lambda=\mathrm{id}$.

\begin{remark}\label{rmk: forget sigma}
The group $\Omega$ can be seen as a subgroup of the automorphism group of the Dynkin diagram of $W_\mathrm{aff}$.
Indeed, if $\sigma\in\Omega$,
by definition $\sigma$ is an isometry satisfying $\sigma(\mathcal{A}_\mathrm{id})=\mathcal{A}_\mathrm{id}$.
Therefore, it permutes the generators $S$ of the Coxeter system $(W_\mathrm{aff}, S)$ while respecting its defining relations.
\end{remark}

\begin{definition}\label{def: double coset reps}
Given $I,J\subset S$, we define
\begin{align}
\begin{split}
{}_I\mathcal{D}_J &= \{w\in W_\mathrm{aff} \mid w<sw  \mbox{ and }  w <wt \ \mbox{ for all } s\in I \mbox{ and } t\in J\}.
\end{split}
\end{align}
If the parabolics subgroups $W_I$ and $W_J$ are both finite, we define
\begin{equation}
{}^I\mathcal{D}^J = \{w\in W_\mathrm{aff} \mid w>sw  \mbox{ and }  w > wt \ \mbox{ for all  }s\in I \mbox{ and } t\in J\}.
\end{equation}
\end{definition}
It is immediate that $({}_I\mathcal{D}_J)^{-1}={}_J\mathcal{D}_I$ and $({}^I\mathcal{D}^J)^{-1}={}^J\mathcal{D}^I$.
When either of the two sets is empty (resp. $S_\mathrm{f}$) we will just leave a blank space (resp. put $\mathrm{f}$) instead.   
For instance,  $\mathcal{D}_\mathrm{f}={}_\emptyset\mathcal{D}_{S_\mathrm{f}}$ (and similarly for the other cases).
It is a fact \cite[Proposition 2.1]{SingularSoergelBimodules} that ${}_I\mathcal{D}_J$ (resp. ${}^I\mathcal{D}^J$) forms a complete system of representatives of $W_I \backslash W_\mathrm{aff} / W_J$, of minimal (resp. maximal) length.
For this reason, its elements are called the minimal (resp. maximal) double coset representatives.  

Using this notation we can rewrite  $W_\mathrm{aff}^-$ and $W_\mathrm{aff}^+$.\footnote{The rewriting \eqref{eq: dom and antidom coset reps} is due to our choice of fundamental alcove.
In the usual conventions (e.g.\cite{humphreys1992reflection, Bour46}) one would have ${}_\mathrm{f}\mathcal{D}=W_\mathrm{aff}^+$ and ${}^\mathrm{f}\mathcal{D}=W_\mathrm{aff}^-$.}
In formulas, we have
\begin{equation}\label{eq: dom and antidom coset reps}
{}_\mathrm{f}\mathcal{D}=W_\mathrm{aff}^- \qquad \mbox{and} \qquad {}^\mathrm{f}\mathcal{D}=W_\mathrm{aff}^+.
\end{equation}
For $\sigma\in\Omega$, we write $\mathcal{D}^\sigma=\mathcal{D}^{(S_\mathrm{f})_\sigma}$.
It is easy to check that $\mathcal{D}^\sigma=(\mathcal{D}^\mathrm{f})_\sigma=\sigma\mathcal{D}^\mathrm{f}\sigma^{-1}$.
The analogous equality holds for $\mathcal{D}_\sigma, {}_\sigma\mathcal{D}, {}^\sigma\mathcal{D}$ (defined similarly).
With this notation we have that $\mathcal{D}^\mathrm{f}=\mathcal{D}^\mathrm{id}$, however we will always prefer $\mathcal{D}^\mathrm{f}$ to avoid confusion.

\subsection{The lowest two-sided cell} \label{subsection LTSC}
In this section we collect some facts about the lowest two-sided cell $\mathfrak{C}$ of $W_\mathrm{aff}$.
The results preceding \Cref{def: amigos} can be found in \cite{shi1987two,shi1988two}, while the subsequent content is original.

For $x, y \in W_\mathrm{aff}$, if  $\ell (xy) = \ell (x) + \ell (y)$, we write $x\cdot y $ to denote both the element $xy\in W_\mathrm{aff}$ and the statement that its length is the sum of the lengths of $x$ and $y$.

We have
\begin{equation}
\mathfrak{C} = \{  x\cdot (w_0)_\sigma \cdot y\in W_\mathrm{aff} \mid x,y\in W_\mathrm{aff} \text{ and } \sigma\in \Omega  \}.
\end{equation}
Shi gave an alternative alcovic description of $\mathfrak{C}$ in terms of a certain hyperplane arrangement known as the Shi arrangement.
He showed that
\begin{equation} \label{LTSC in Shi}
\mathcal{A}(\mathfrak{C}) = \{ \lambda \in E \mid   0 \leq  \langle  \lambda , \alpha \rangle  \mbox{ or }  \langle  \lambda , \alpha \rangle \leq -1  \mbox{ for all } \alpha \in \Phi^{+} \}.
\end{equation}
Let $\Pi^+$ be the parallelepiped spanned by the fundamental coweights.
More precisely, 
\begin{equation}\label{def: paralele}
\Pi^+= \{ \lambda \in E \mid   0\leq  \langle  \lambda , \alpha \rangle \leq  1, \mbox{ for all } \alpha \in \Delta \}.
\end{equation}
We also define $\Pi^-=w_0(\Pi^+)$, which is the parallelepiped spanned by the negative fundamental coweights.
Both $\Pi^+$ and $\Pi^-$ are alcoved polytopes (as defined in \cite{lam2007alcoved, lam2018alcoved}), meaning they are unions of several alcoves.
Specifically, each consists of $|W_\mathrm{f}|/|\Omega|$ alcoves, as shown in \cite[\S 4]{lam2018alcoved}.

We define $\mathcal{F}=\mathcal{A}^{-1}(\Pi^-)$.
Note that for all $a\in\mathcal{F}$ we have $\ell(w_0a)=\ell(w_0)+\ell(a)$,
since $w_0\in W_\mathrm{f}$ and $a\in {}_\mathrm{f}\mathcal{D}$ (see \cite[Equation (2.12)]{BB}). 
Therefore, we have $w_0\cdot a$ and $a^{-1}\cdot w_0$. 
A direct computation yields
\begin{equation}\label{eq: parallelw0}
\mathcal{A}^{-1}(\Pi^+)=\{w_0\cdot a\mid a\in\mathcal{F}\}.
\end{equation}

For $a\in\mathcal{F}$, consider the set
\begin{equation}
\mathfrak{C}(a)= \{a^{-1}\cdot w_0\cdot z \mid z\in W_\mathrm{aff} \}.
\end{equation}
Note that $\mathfrak{C}(\mathrm{id})=W_\mathrm{aff}^+$.
A third description of the lowest two-sided cell is in \cite[Corollary 1.2]{shi1988two}:
\begin{equation}\label{eq: LTSC cc}
\mathfrak{C}=\bigsqcup_{\substack{a\in\mathcal{F}\\ \sigma \in \Omega  }} \mathfrak{C}(a)_\sigma.
\end{equation}
It is not hard to prove\footnote{Here, we slightly abuse notation, as the unions in \eqref{eq: param c plus} and \eqref{eq: LTSC alcovic param} are not strictly disjoint. They become disjoint only when considering the interiors of the corresponding sets.} that one can tesselate $C^+$ using the parallelepiped $\Pi^+$:
\begin{equation}\label{eq: param c plus}
C^+=\mathcal{A}(\mathfrak{C}(\mathrm{id}))=\bigsqcup_{\lambda\in(\Lambda^\vee)^+}\Pi^++\lambda. 
\end{equation}
Using \eqref{eq: LTSC cc}, we get an alcovic parametrization of $\mathfrak{C}$:
\begin{equation}\label{eq: LTSC alcovic param}
\mathcal{A}(\mathfrak{C})=\bigsqcup_{\substack{a\in\mathcal{F}\\ \sigma \in \Omega \\ \lambda\in(\Lambda^\vee)^+}}\,\sigma a^{-1}(\Pi^++\lambda).
\end{equation}
We will use this formula to derive an algebraic parametrization of $\mathfrak{C}$. Recall that for $\lambda$ a dominant coweight, the translation of any alcove by $\lambda$ is still an alcove.

\begin{definition}\label{def: amigos}
For $\lambda\in(\Lambda^\vee)^+$ and $a\in\mathcal{F}$, we define the element $\theta_a(\lambda)\in W_\mathrm{aff}^+$ as the unique element in $W_\mathrm{aff}$ such that
\begin{equation}
\mathcal{A}_{\theta_a(\lambda)}=\mathcal{A}_{w_0a}+\lambda.
\end{equation}
\end{definition}

We remark that the element $\theta_\mathrm{id}(\lambda)$ coincides with the element $\theta(\lambda)$ introduced in \cite{castillo2023size, PreCanonical}.
The elements $\theta_\mathrm{id}(\lambda)$ posses the following  property, see \cite[Lemma 2.4]{castillo2023size}.
\begin{lemma}\label{lemma: thetas are maximal}
Let $\lambda$ be a dominant coweight and let $\tau=\sigma_\lambda\in\Omega$ as in \eqref{eq: sigmaofcoweight}.
Then $\theta_\mathrm{id}(\lambda)\in {}^\mathrm{f}\mathcal{D}^\tau$.
Furthermore,
\begin{equation}
\{\theta_\mathrm{id}(\lambda) \mid \lambda\in(\Lambda^\vee)^+\} = \bigsqcup_{\sigma\in\Omega} {}^\mathrm{f}\mathcal{D}^\sigma.
\end{equation}
\end{lemma}

The following lemma gives an equivalent algebraic definition of the elements $\theta_a(\lambda)$.

\begin{lemma}\label{lem: amigos algebraic def}
For $\lambda\in(\Lambda^\vee)^+$ and $a\in\mathcal{F}$, let $\sigma_\lambda$ be as in Equation \eqref{eq: sigmaofcoweight}.
Then
\begin{equation}
\theta_a(\lambda)=t_\lambda\, w_0\, a \,\sigma_\lambda^{-1},
\end{equation}
where $t_\lambda\in W_\mathrm{e}$ is the translation by $\lambda$.
\end{lemma}

\begin{proof}

For notational simplicity, we put $\sigma=\sigma_\lambda$.
In \cite[\S 2.1]{PreCanonical} it is shown that $t_\lambda w_0\sigma^{-1}\in W_\mathrm{aff}$ (in that paper this element is called $\theta(\lambda)$).
This implies that $(t_\lambda w_0\sigma^{-1})(\sigma a\sigma^{-1})\in W_\mathrm{aff}$. In other words, $x:=t_\lambda w_0 a\sigma^{-1} \in W_\mathrm{aff}$ and
\begin{equation}
\mathcal{A}_{x}=x(\mathcal{A}_\mathrm{id})=\mathcal{A}_{w_0a}+\lambda,
\end{equation}
since $\sigma^{-1}$ fixes the fundamental alcove.
Therefore, $x=\theta_a(\lambda)$.
\end{proof}

We stress that, for $a\in\mathcal{F}$ and $\lambda,\mu\in\Lambda^\vee$ with $\lambda,\lambda-\mu\in(\Lambda^\vee)^+$,
\begin{equation}\label{vminus}
V(\theta_a(\lambda-\mu))=V(\theta_a(\lambda))-\mu.
\end{equation}

Let $\lambda\in(\Lambda^\vee)^+$.
Note that by Equation \eqref{eq: parallelw0} and \Cref{def: amigos}, we get
\begin{equation}\label{eq: parallel trasladadao}
\mathcal{A}^{-1}(\Pi^++\lambda)=\{\theta_a(\lambda)\mid a\in\mathcal{F}\}.
\end{equation}
Hence, Equation \eqref{eq: LTSC alcovic param} gives an algebraic parametrization of $\mathfrak{C}$:
\begin{equation}
\mathfrak{C} = \bigsqcup_{\substack{a\in\mathcal{F}\\ \sigma \in \Omega \\ \lambda\in(\Lambda^\vee)^+}}\,\mathcal{A}^{-1}(\sigma a^{-1}(\Pi^++\lambda))
= \bigsqcup_{\substack{a\in\mathcal{F}\\ \sigma \in \Omega \\ \lambda\in(\Lambda^\vee)^+}}\,\{(a^{-1}\theta_b(\lambda))_\sigma\mid b\in\mathcal{F}\}.
\end{equation}
That is, there is a bijection
\begin{equation}\label{def: param LTSC}
\Theta:\mathcal{F}\times(\Lambda^\vee)^+\times\mathcal{F}\times\Omega\xlongrightarrow{\sim}\mathfrak{C}
\end{equation}
given by $\Theta(a,\lambda,b,\sigma)=(a^{-1}\theta_b(\lambda))_\sigma=\sigma a^{-1}\theta_b(\lambda)\sigma^{-1}$.

We can say more about this bijection, but first we need to establish the following lemma.
For $X,Y\subset W_\mathrm{aff}$, we write $X\cdot Y$ to denote the set $XY$ and to state that $x\cdot y$ (i.e. $\ell(xy)=\ell(x)+\ell(y)$) for every $x\in X$ and $y\in Y$.
\begin{lemma}\label{lem: coset reps dot multiplication}
For all $\sigma\in\Omega$, we have $\mathcal{D}_\sigma\cdot{}^\sigma\mathcal{D}$ and $\mathcal{D}^\sigma\cdot{}_\sigma\mathcal{D}$.
\end{lemma}

\begin{proof}
Suppose $\mathcal{D}^\mathrm{f}\cdot{}_\mathrm{f}\mathcal{D}$ holds.
Then $(\mathcal{D}^\mathrm{f}\cdot{}_\mathrm{f}\mathcal{D})^{-1} = ({}_\mathrm{f}\mathcal{D})^{-1} \cdot (\mathcal{D}^\mathrm{f})^{-1} = \mathcal{D}_\mathrm{f} \cdot {}^\mathrm{f}\mathcal{D}$, hence simply conjugating by $\sigma$ yields the desired result.

So it suffices to prove $\mathcal{D}^\mathrm{f}\cdot{}_\mathrm{f}\mathcal{D}$.
Let $x\in\mathcal{D}^\mathrm{f}$ and $y\in {}_\mathrm{f}\mathcal{D}$.
We prove $\ell(x)+\ell(y) \leq \ell(xy)$, as the reverse inequality always holds.
For $u,w\in W_\mathrm{aff}$, let $\mathcal{H}(u,w)$ be the set of hyperplanes $H_{\alpha,k}$ that separate $\mathcal{A}_u$ and $\mathcal{A}_w$.
We claim the following:
\begin{enumerate}[(a)]
\item\label{double coset claim1} $x\mathcal{H}(\mathrm{id},y)=\{xH\mid H\in \mathcal{H}(\mathrm{id},y)\} \subset \mathcal{H}(\mathrm{id},xy)$,
\item\label{double coset claim2} $\mathcal{H}(\mathrm{id},x) \subset \mathcal{H}(\mathrm{id},xy)$.
\end{enumerate}

Assume the claims and let $H\in\mathcal{H}(\mathrm{id},x)\cap\mathcal{H}(x,xy)$.
Then $\mathcal{A}_\mathrm{id}$ and $\mathcal{A}_{xy}$ both lie on the same side of $H$, but this contradicts \Cref{double coset claim2}. 
Hence the intersection $\mathcal{H}(\mathrm{id},x)\cap\mathcal{H}(x,xy)$ is empty.
On the other hand,  it is clear that $x\mathcal{H}(\mathrm{id},y) \subset \mathcal{H}(x,xy)$.
Therefore, using \Cref{double coset claim1}, we have
\begin{equation}
\left(\mathcal{H}(\mathrm{id},x)\sqcup x\mathcal{H}(\mathrm{id},y)\right) \subset \mathcal{H}(\mathrm{id},xy),
\end{equation}
which implies $\ell(x)+\ell(y) \leq \ell(xy)$.

So it remains to prove the claims.
The proof is as follows.
\begin{enumerate}[(a)]
\item Let $H\in \mathcal{H}(\mathrm{id},y)$.
Note that $x^{-1}\in {}^\mathrm{f}\mathcal{D}=W_\mathrm{aff}^+$ and $y\in {}_\mathrm{f}\mathcal{D}=W_\mathrm{aff}^-$, so we must have that $H\in\mathcal{H}(x^{-1},y)$.
Therefore $xH\in\mathcal{H}(\mathrm{id},xy)$, as we wanted to show.
\item Let $H\in \mathcal{H}(\mathrm{id},x)$.
Then $x^{-1}H\in \mathcal{H}(x^{-1},\mathrm{id})$.
Once again, since $x^{-1}\in W_\mathrm{aff}^+$ and $y\in W_\mathrm{aff}^-$, we must have $x^{-1}H\in \mathcal{H}(x^{-1},y)$ and hence $H\in\mathcal{H}(\mathrm{id},xy)$.
\end{enumerate}

This concludes the proof of the lemma.
\end{proof}

Now we are in position to describe the length of any element in $\mathfrak{C}$.
\begin{cor}\label{cor: length of parametrization}
Let $\lambda\in(\Lambda^\vee)^+$ and $a,b\in\mathcal{F}$.
Let $\tau=\sigma_\lambda\in\Omega$ as in \eqref{eq: sigmaofcoweight}.
Then $a^{-1}\cdot\theta_\mathrm{id}(\lambda)\cdot \tau b \tau^{-1}$.
Consequently, for any $\sigma\in\Omega$ we have
\begin{equation}\label{eq: length of parametrization}
\ell(\Theta(a,\lambda,b,\sigma))=\ell(a)+\ell(b)+\ell(w_0)+2 \langle\lambda,\rho\rangle.
\end{equation}
\end{cor}

\begin{proof}
On the one hand, since $b\in\mathcal{F}\subset W_\mathrm{aff}^-={}_\mathrm{f}\mathcal{D}$, then $\tau b \tau^{-1} \in ({}_\mathrm{f}\mathcal{D})_\tau={}_\tau\mathcal{D}$.
Also, we know that $\theta_\mathrm{id}(\lambda)\in{}^\mathrm{f}\mathcal{D}^\tau\subset\mathcal{D}^\tau$, by \Cref{lemma: thetas are maximal}.
Hence, \Cref{lem: coset reps dot multiplication} implies that $\theta_\mathrm{id}(\lambda)\cdot \tau b \tau^{-1}$.

On the other hand, $a\in {}_\mathrm{f}\mathcal{D}$ implies that $a^{-1}\in \mathcal{D}_\mathrm{f}$ and we know that $\theta_b(\lambda)\in W_\mathrm{aff}^+ = {}^\mathrm{f}\mathcal{D}$.
Applying \Cref{lem: coset reps dot multiplication} for ($\sigma=\mathrm{id}$) and using \Cref{lem: amigos algebraic def}, we deduce that
\begin{equation}\label{eq: a theta b}
a^{-1}\cdot\theta_b(\lambda) = a^{-1}\cdot\theta_\mathrm{id}(\lambda)\cdot \tau b \tau^{-1}.
\end{equation}
Finally, \eqref{eq: length of parametrization} follows from \Cref{rmk: forget sigma} and the equality $\ell(\theta_\mathrm{id}(\lambda)) = \ell(w_0)+2\langle\lambda,\rho\rangle$ (see \cite[\S 2.1]{PreCanonical}).
\end{proof}

A remarkable property of the parametrization \eqref{def: param LTSC} is that it intertwines the dominance order with the Bruhat order, as the following lemma shows.

\begin{lemma}\label{lem: intertwining orders}
Fix $a,b\in\mathcal{F}$ and $\sigma\in\Omega$.
Let $\lambda,\mu\in(\Lambda^\vee)^+$ such that $\lambda\equiv\mu\mod\bbZ\Phi^\vee$.
Then
\begin{equation*}
\mu\leq\lambda \qquad\mbox{if and only if}\qquad \Theta(a,\mu,b,\sigma)\leq \Theta(a,\lambda,b,\sigma).
\end{equation*}
\end{lemma}
\begin{proof}

It suffices to consider $\sigma=\mathrm{id}$ because if $x,y\in W_{\mathrm{aff}}$ and $\sigma \in \Omega$ one can prove that $x\leq y\iff x_{\sigma}\leq y_{\sigma}$ (see \Cref{rmk: forget sigma}).
Let $\tau=\sigma_\lambda=\sigma_\mu\in\Omega$ as in \eqref{eq: sigmaofcoweight}.
In the following, we use the fact that for all $x,y\in W_\mathrm{aff},$ if $z\in W_\mathrm{aff}$ satisfies that $z\cdot x$ and $ z\cdot y$ then 
\begin{equation}\label{eq: lem intertwining step 1}
x\leq y \iff  z\cdot x\leq z\cdot y 
\end{equation}
There is a similar equivalence if one multiplies by $z$ on the right. Both equivalences are proved by induction on $\ell(z)$ using the Lifting Property for Coxeter Systems. 
We have 
\begin{align*}
\mu\leq\lambda &\iff \theta_{\mathrm{id}}(\mu)\leq\theta_{\mathrm{id}}(\lambda), &\mbox{(\cite[\S 2.1]{PreCanonical})} \\
&\iff a^{-1}\cdot\theta_{\mathrm{id}}(\mu)\cdot\tau b \tau^{-1} \leq a^{-1}\cdot\theta_{\mathrm{id}}(\lambda)\cdot\tau b \tau^{-1},
&\mbox{(\Cref{cor: length of parametrization} \mbox{ and } \eqref{eq: lem intertwining step 1})} \\
&\iff a^{-1}\cdot\theta_b(\mu) \leq a^{-1}\cdot\theta_b(\lambda),
&\mbox{(\Cref{lem: amigos algebraic def})} \\
&\iff \Theta(a,\mu,b,\mathrm{id})\leq \Theta(a,\lambda,b,\mathrm{id}),
\end{align*}
finishing the proof.
\end{proof}

For $(a,\lambda,b,\sigma)\in \mathcal{F}\times(\Lambda^\vee)^+\times\mathcal{F}\times\Omega$, we denote by $\mathcal{I}(a,\lambda,b,\sigma)$ the lower Bruhat interval of $\Theta(a,\lambda,b,\sigma)\in\mathfrak{C}$.
In formulas,
\begin{equation}\label{def: lower int gral}
\mathcal{I}(a,\lambda,b,\sigma) = \{ w \in W_\mathrm{aff} \mid w \leq \Theta(a,\lambda,b,\sigma) \}.
\end{equation}

\subsection{The Paper Boat} \label{subsection PB}
Consider the subposet $((\Lambda^\vee)^+, \leq)$ of the poset of coweights with the dominance order.
For $\lambda , \mu \in (\Lambda^\vee)^+ $ we write $\mu \lessdot \lambda$ if $\lambda $ covers $\mu $ in  $((\Lambda^\vee)^+, \leq)$.
In \cite{stembridge1998partial}, the author gives a complete description of this covering relation.
In particular, he shows that $\lambda-\mu\in(\Phi^\vee)^+$ whenever $\mu \lessdot \lambda$.
For $\lambda\in(\Lambda^\vee)^+$, we define
\begin{equation}
(\Phi^\vee)^+(\lambda) = \{ \lambda-\mu\mid \mu\in (\Lambda^\vee)^+  \mbox{ and } \mu \lessdot \lambda\}  = \{ \alpha \in  (\Phi^\vee)^+ \mid \lambda - \alpha \in (\Lambda^\vee)^+  \mbox{ and } \lambda - \alpha \lessdot \lambda \}.\label{phi+lambda}
\end{equation}

\begin{remark}
We emphasize that the symbol $\lessdot$ is reserved for coverings in the subposet of dominant coweights.
For example consider type $A_2$, which is self-dual so there is no distinction between roots and coroots and between weights and coweights.
We have $\mathbf{0} \lessdot \varpi_1 + \varpi_2$.
This relationship does not hold when considering the full set of weights.
Indeed, note that $\mathbf{0} < -\varpi_1 + 2\varpi_2 < \varpi_1 + \varpi_2$, which means that $\varpi_1 + \varpi_2$ does not cover $\mathbf{0}$ in the weight poset $\Lambda$. 
\end{remark}

\begin{definition}
We define the Paper Boat associated to $(a,\lambda,b,\sigma)\in \mathcal{F}\times(\Lambda^\vee)^+\times\mathcal{F}\times\Omega$, denoted by $PB(a,\lambda,b,\sigma)$, as
\begin{align}
PB(a,\lambda,b,\sigma) &= \mathcal{I}(a,\lambda,b,\sigma) \setminus \bigcup_{\alpha \in (\Phi^\vee)^+(\lambda)}   \mathcal{I}(a,\lambda -\alpha,b,\sigma) \\
&= \mathcal{I}(a,\lambda,b,\sigma) \setminus \bigcup_{\mu\lessdot\lambda}   \mathcal{I}(a,\mu,b,\sigma).
\end{align}
\end{definition}

For $\lambda\in(\Lambda^\vee)^+$, we denote by $\mathfrak{I}_\lambda$ the ideal generated by $\lambda$ in the subposet of dominant coweights.
This ideal can be equivalently described as
\begin{equation}\label{eq: ideal as latticepoints}
\mathfrak{I}_\lambda= \mathsf{P}(\lambda)\cap C^+\cap(\lambda+\bbZ\Phi^\vee),
\end{equation}
where 
\begin{equation}
\mathsf{P}(\lambda) = \mathrm{Conv}(W_\mathrm{f}\cdot\lambda) \label{eq:P_definition},
\end{equation}
see \cite[Proof of Theorem 3.2]{castillo2023size}.

We have two central conjectures.
The first one asserts that each lower Bruhat interval generated by an element in the lowest two-sided cell, is a disjoint union of Paper Boats.

\begin{conjecture}\label{conj: lowerint as pbs}
Let $(a,\lambda,b,\sigma)\in \mathcal{F}\times(\Lambda^\vee)^+\times\mathcal{F}\times\Omega$.
Then
\begin{equation}\label{eq: lowerint as pbs}
\mathcal{I}(a,\lambda,b,\sigma) = \bigsqcup_{\mu\in\mathfrak{I}_\lambda} PB(a,\mu,b,\sigma).
\end{equation}
\end{conjecture}

The non-trivial part of \Cref{conj: lowerint as pbs} is that the union is actually disjoint.
It would imply that if one manages to compute the cardinalities of Paper Boats, then one would be able to compute the sizes of all lower intervals in the lowest two-sided cell.
However, a priori, there are infinitely many Paper Boats to compute.
Our second conjecture would solve this.

\begin{conjecture}\label{conj: paperboat zones}
For all affine Weyl groups there is a partition $\mathcal{Z}$ of $(\Lambda^\vee)^+$ with at most $3^n$ parts satisfying the following property.
Fix $a,b\in\mathcal{F}$, $\sigma\in\Omega$ and $Z\in\mathcal{Z}$.
Then,
\begin{equation}\label{eq:surprising}
| PB(a,\lambda,b,\sigma) |=| PB(a,\mu,b,\sigma) | \qquad\mbox{ for all } \lambda,\mu\in Z.
\end{equation}
\end{conjecture}

Suppose $\mathcal{Z}$ is such a partition.
As mentioned above, \Cref{conj: lowerint as pbs} and \Cref{conj: paperboat zones} would imply that
\begin{equation}\label{eq: lowerint implied by conjs}
|\mathcal{I}(a,\lambda,b,\sigma)| = \sum_{Z\in\mathcal{Z}} |\mathfrak{I}_\lambda\cap Z|\, | PB(a,Z,b,\sigma) |,
\end{equation}
where $| PB(a,Z,b,\sigma) |$ is the cardinality of $| PB(a,\mu,b,\sigma) |$ for any $\mu\in Z$.

\section{The Paper Boat in the dominant cone of type \texorpdfstring{$\widetilde{A}$}{A}}\label{sec: main results}
In this section, we prove \Cref{conj: lowerint as pbs} and \Cref{conj: paperboat zones} for affine type $A$ and for all elements in $W_\mathrm{aff}^+=\mathfrak{C}(\mathrm{id})$.

Since type $A$ is self-dual, we make no distinction between roots and coroots and between weights and coweights.
We use the following realization of the root system.
Consider the Euclidean space $\bbR^{n+1}$ with canonical basis $\varepsilon_1,\ldots,\varepsilon_{n+1}$ and dot product $\lr{-}{-}$.
Let $E$ be the hyperplane of vectors whose coordinate sum is zero.
The root system of type $A_n$ is
$
\Phi= \{  \varepsilon_i-\varepsilon_j \in E \mid 1\leq i,j \leq n+1, \, i\neq j \}.
$
The simple roots are $\Delta =\{ \alpha_i = \varepsilon_i-\varepsilon_{i+1} \mid 1\leq i \leq n  \}$. 
We stress that in type $A$ the vertices of the fundamental alcove are $V(\mathrm{id})=\{\mathbf{0},-\varpi_1,\ldots,-\varpi_n\}.$
For convenience, we write $\varpi_0=\mathbf{0}$.

The main tool we use to establish our results in this section is the following criterion, which allows us to compare elements in the Bruhat order within $W_\mathrm{aff}^+$. As mentioned in the introduction, this criterion states that we can compare two elements in the Bruhat order by comparing the vertices of their corresponding alcoves in the dominance order. For the sake of readability, we defer its proof to 
\Cref{sec: criterio}.

\begin{prop}\label{prop: main tool} Let $x,y\in W_\mathrm{aff}^+$.
Then $x\leq y$ (in the Bruhat order), if and only if $x(-\varpi_i)\leq y(-\varpi_i)$ for all $i\in \{0,1,\ldots ,n\}$ (in the dominance order).
Equivalently, $x\leq y$ in the Bruhat order if and only if $\nu\leq\mu$ in the dominance order for all $(\nu,\mu)\in V(x)\times V(y)$ satisfying $\nu\equiv\mu\mod\bbZ\Phi$.
\end{prop}

Recall the parametrization in \eqref{def: param LTSC}.
By definition, we have $\Theta(\mathrm{id},\lambda,a,\mathrm{id})=\theta_a(\lambda)$.
The restriction of \eqref{def: param LTSC} to $W_\mathrm{aff}^+$, gives
\begin{equation}
W_\mathrm{aff}^+=\mathfrak{C}(\mathrm{id})=\{ \theta_a(\lambda)\mid a\in\mathcal{F}, \lambda\in\Lambda^+\}.
\end{equation}

Thus, for simplicity, for $\lambda\in\Lambda^+$ and $a\in\mathcal{F}$ we introduce the notation
\begin{align}
\mathcal{I}_a(\lambda)&=\mathcal{I}(\mathrm{id},\lambda,a,\mathrm{id})=\{w\in W_\mathrm{aff} \mid w\leq \theta_a(\lambda)\}, \\
PB_a(\lambda)&=PB(\mathrm{id},\lambda,a,\mathrm{id})=\li{\lambda} \setminus \bigcup_{\alpha \in \Phi^+(\lambda)}  \li{\lambda - \alpha},
\end{align}
where the set $\Phi^+(\lambda)=\{ \alpha \in  \Phi^+ \mid \lambda - \alpha \in \Lambda^+  \mbox{ and } \lambda - \alpha \lessdot \lambda \}$ was introduced in \eqref{phi+lambda}. That set was described for arbitrary root systems in \cite{stembridge1998partial}.
In type $A$,  the following explicit description is given.
If $\lambda \in \Lambda^+$,
then $\alpha\in \Phi^+(\lambda)$ if and only if one of the following conditions hold.
\begin{enumerate}
\item $\alpha \in \Delta$ and $ \lr{\lambda}{\alpha} \geq 2 $.
\item $\alpha=\alpha_{i,j} =\alpha_i +\alpha_{i+1}+ \ldots + \alpha_j$, $\lr{\lambda}{\alpha_i} =\lr{\lambda}{\alpha_j}=1$ and $\lr{\lambda}{\alpha_k}=0$ for all $i<k<j$. 
\end{enumerate}

\begin{example}
Consider  type $A_{12}$. If the coordinates of $\lambda$ in the fundamental weight basis are given by $(3,4,1,0,1,1,2,1,1,5,0,1)$, then 
$$\Phi^+(\lambda) = \{ \alpha_1,\alpha_2, \alpha_{3,5}, \alpha_{5,6}, \alpha_{7}, \alpha_{8,9} , \alpha_{10}  \}.$$
\end{example}

\begin{remark}
Given this description, it is easy to see that $\Phi^+(\lambda) = \emptyset$ if and only if $\lambda = \varpi_i$ for $0\leq i \leq n$.
It follows that $PB_a(\lambda) = \mathcal{I}_a(\lambda)$ if and only if $\lambda = \varpi_i$ for some $0\leq i \leq n$.
\end{remark}

\begin{definition}
For $\lambda\in\Lambda^+$ and $a\in \mathcal{F}$, we define $PB_a^+(\lambda)=PB_a(\lambda)\cap W_\mathrm{aff}^+$.
We set
\begin{equation}
c_a(\lambda) = |PB_a(\lambda)|  \qquad \mbox{and}  \qquad  c_a^+(\lambda) = |PB_a^+(\lambda)|.
\end{equation}
\end{definition}

We will show that $c_a(\lambda) = |W_\mathrm{f}| \cdot c_a^+(\lambda) $. 
Before proving this, we need to establish a preliminary result.

\begin{lemma} \label{lem:inv under s_i}
Let $x \in W_\mathrm{aff}^+$. Then, its lower interval  $\{y\in W_\mathrm{aff}\mid y\leq x\}$ is invariant under the action of $W_\mathrm{f}$ on the left.
\end{lemma}

\begin{proof}
Suppose $y \leq x$, and let $s \in S_\mathrm{f}$.
It suffices to show that $sy \leq x$.
We first note that as $x\in W_\mathrm{aff}^+$, then $sx < x$ by \eqref{eq: dom and antidom coset reps}.
We consider two cases:
\begin{itemize}
\item If $s y < y$, then clearly $s y \leq x$.
\item If $y < s y$, the  Lifting Property for Coxeter systems together with the fact that $sx < x$, implies that $s y \leq x$.
\end{itemize}
\end{proof}

\begin{cor}\label{cor: reduction to C+ Paper Boat}
Let $\lambda \in \Lambda^+$ and $a \in \mathcal{F}$.
Then, $PB_a(\lambda)$ is invariant under the action of $W_\mathrm{f}$ on the left.
Consequently, 
\begin{equation} \label{eq: cardinalities}
c_a(\lambda) = |W_\mathrm{f}| \cdot c_a^+(\lambda).
\end{equation}
\end{cor}

\begin{proof}
By the definition of $\Phi^+(\lambda)$, it follows that $\theta_a(\lambda)$ and $\theta_a(\lambda - \alpha)$ lie in $W_\mathrm{aff}^+$ for all $\alpha \in \Phi^+(\lambda)$.
The invariance of $PB_a(\lambda)$ follows by a direct application of \Cref{lem:inv under s_i}.
Hence, by \eqref{eq: partition of affine}, we get
\begin{equation*}
PB_a(\lambda)=\bigsqcup_{w \in W_\mathrm{f}} wPB_a^+(\lambda).
\end{equation*}
Taking cardinalities on both sides gives \eqref{eq: cardinalities}.
\end{proof}

In order to prove \Cref{conj: lowerint as pbs}, we first need some previous results.
Let $\Lambda_0,\ldots,\Lambda_n$ be the cosets of $\Lambda$ modulo $\bbZ\Phi$, and put $\Lambda_i^+=\Lambda_i\cap\Lambda^+$.
It is well known that the component $(\Lambda_i,\leq)$ of the poset $(\Lambda,\leq)$ is a lattice.
In \cite{stembridge1998partial} Stembridge studies the subposet $(\Lambda^+,\leq)$ and shows that its components $(\Lambda_i^+,\leq)$ are also lattices.
Let $\wedge, \vee$ denote the meet and join operations on $\Lambda_i$, and $\wedge^+, \vee^+$ the corresponding operations on $\Lambda_i^+$. Stembridge shows that:
\begin{align}
\left( \sum a_\alpha \alpha \right) \wedge \left( \sum b_\alpha \alpha \right) &= \sum \operatorname{min}(a_\alpha,b_\alpha)\alpha, \label{eq: meet}\\
\left( \sum a_\alpha \alpha \right) \vee \left( \sum b_\alpha \alpha \right) &= \sum \operatorname{max}(a_\alpha,b_\alpha)\alpha, \label{eq: join}
\end{align}
and that $\wedge^+$ is simply the restriction of $\wedge$ to $\Lambda_i^+$. However, the restriction of $\vee$ does not coincide with $\vee^+$.
For example, in type $A_3$, consider $\lambda, \mu \in \Lambda^+$ given by $\lambda = (3,2,1)$ and $\mu = (1,2,3)$ in the simple root basis.
In this case, $\lambda\vee\mu = (3,2,3)\notin\Lambda^+$.
In what follows, we will use $\wedge$ and $\wedge^+$ interchangeably.

\begin{lemma}\label{lem: meet pesos}
Let $\lambda,\mu\in\Lambda_i^+$ and $a\in\mathcal{F}$.
Then, $\li{\lambda}\cap\li{\mu} = \li{ \lambda\wedge \mu}$.
\end{lemma}

\begin{proof}
By \Cref{lem:inv under s_i} and \eqref{eq: partition of affine}, it is enough to show that 
\begin{equation}\label{eq: intersection sums1}
\lip{\lambda}\cap\lip{\mu} = \lip{ \lambda\wedge \mu},
\end{equation}
where $\lip{\lambda}\coloneqq \li{\lambda}\cap W_\mathrm{aff}^+$.
Let $L$ and $R$ be the left- and right-hand sides of \eqref{eq: intersection sums1}, respectively.
The containment $L\supseteq R$ is a consequence of \Cref{lem: intertwining orders}.
Let $w\in L$ and $0\leq j\leq n$.
Put $\sigma=\sigma_\lambda=\sigma_\mu$ as in \eqref{eq: sigmaofcoweight}.
Let us define the vector $v:=w_0a\sigma^{-1}(-\varpi_j)$.
By \Cref{lem: amigos algebraic def} and \Cref{prop: main tool}, we have
\begin{align*}
w(-\varpi_j)&\leq \theta_a(\lambda)(-\varpi_j) = v +\lambda, \\
w(-\varpi_j)&\leq \theta_a(\mu)(-\varpi_j) = v +\mu.
\end{align*}
Since $w(-\varpi_j)-v$ is smaller than  $\lambda$ and $\mu$, it must be smaller than $\lambda\wedge\mu$.
That is, $$w(-\varpi_j)\leq v + (\lambda\wedge\mu) = \theta_a(\lambda\wedge\mu)(-\varpi_j).$$
Since $j$ was arbitrary, \Cref{prop: main tool} gives $w\in R$.
\end{proof}

We now establish \Cref{conj: lowerint as pbs} for affine type $A$ and $W_\mathrm{aff}^+$.

\begin{theorem}\label{thm: loweint as pbs}
Let $\lambda \in \Lambda^+$ and $a \in \mathcal{F}$.  
Then,  
\begin{equation}  \label{eq: loweint as pbs}
\li{\lambda} = \bigsqcup_{\mu \in \mathfrak{I}_\lambda} PB_a(\mu).  
\end{equation}  
\end{theorem}  

\begin{proof}  
By \Cref{lem: intertwining orders}, we know that $\theta_a(\mu) \leq \theta_a(\lambda)$ for all $\mu \in \mathfrak{I}_\lambda$, thus the inclusion $\supseteq$ follows.  
Conversely, let $w \leq \theta_a(\lambda)$, and let $\mu \in \mathfrak{I}_\lambda$ be minimal such that $w \leq \theta_a(\mu)$.  
By minimality, $w \in PB_a(\mu)$, and hence the inclusion $\subseteq$ holds.  

It remains to prove that the union is disjoint.  
Let $\mu_1, \mu_2 \in \mathfrak{I}_\lambda$ with $\mu_1 \neq \mu_2$, and suppose there exists $w \in PB_a(\mu_1) \cap PB_a(\mu_2)$.  
By \Cref{lem: meet pesos}, we have $w \leq \theta_a(\mu_1 \wedge \mu_2)$.  

Let $i$ be such that $\lambda \in \Lambda_i^+$. 
We analyze two cases:  
\begin{enumerate}[(a)]  
\item Let us suppose that $\mu_1 \leq \mu_2$. As $\mu_1 \neq \mu_2$,
there exists a weight ${\nu}$ such that $\mu_1 \leq {\nu} \lessdot \mu_2$.  
By \Cref{lem: intertwining orders}, we have $w \leq \theta_a({\nu})$, contradicting the assumption that $w \in PB_a(\mu_2)$.  

\item Let us suppose that $\mu_1 \nleq \mu_2$. Then $\mu_1 \wedge \mu_2\neq \mu_1$ and there exists a weight ${\nu}$ such that $\mu_1 \wedge \mu_2 \leq {\nu} \lessdot \mu_1$.  
Since $w \leq \theta_a(\mu_1 \wedge \mu_2)$, \Cref{lem: intertwining orders} implies $w \leq \theta_a({\nu})$.
This contradicts the assumption that $w \in PB_a(\mu_1)$. 
\end{enumerate}  

Thus, we conclude that $PB_a(\mu_1) \cap PB_a(\mu_2) = \emptyset$ for $\mu_1 \neq \mu_2$, as required.  
\end{proof}

\begin{remark}
We note that the only missing ingredient to prove \Cref{conj: lowerint as pbs} in full generality, is a general version of \Cref{lem: meet pesos}. 
The proof of this lemma currently depends on  \Cref{prop: main tool}.
\end{remark}

We now focus on \Cref{conj: paperboat zones}.
The following lemma is a key technical step towards the proof.
It claims that $\mathcal{A}(PB_a^+(\lambda))$ is contained in the cone $C^+ +\lambda - \rho$.

\begin{lemma} \label{lemma Trasl is Domin}
Let $w \in PB_a^+(\lambda)$.
If $\mu\in V(w)$, then $-1 \leq \langle \mu -\lambda  , \alpha  \rangle $ for all $\alpha \in \Delta$.
\end{lemma}

\begin{proof}
We argue by contradiction.
Let us assume that there exists a vertex $\mu \in V(w)$ and some $ \alpha \in \Delta $ such that $ \langle \mu - \lambda, \alpha \rangle \leq -2 $.

On the one hand, let $\nu \in V(w)$.
As the inner product of a weight and $\alpha$ is an integer, the definition of an alcove implies that $\langle \mu, \alpha \rangle - 1 \leq \langle \nu, \alpha \rangle \leq \langle \mu, \alpha \rangle + 1$.
Then $-1 \leq \langle \nu - \mu, \alpha \rangle \leq 1$, which in turn gives $\langle \nu - \lambda, \alpha \rangle \leq -1$.

On the other hand, \eqref{def: paralele} and \eqref{eq: parallel trasladadao} show that $\langle \eta - \lambda, \alpha \rangle \geq 0$ for all $\eta \in V(\theta_a(\lambda))$.
Combining these two results, we conclude that
\begin{equation} \label{eq: DOMI1}
\langle \eta - \nu, \alpha \rangle \geq 1
\end{equation}
for all pairs $(\eta, \nu)$ where $\eta \in V(\theta_a(\lambda))$ and $\nu \in V(w)$.

Now, let $ (\eta, \nu) $ be such a pair and suppose that $\eta \equiv \nu \mod\bbZ\Phi$.
Since $ w \in PB_a^+(\lambda)$, we can apply \Cref{prop: main tool} to conclude that $\nu\leq\eta$ in the dominance order.
That is,
\begin{equation} \label{eq: DOMI2}
\langle \eta - \nu, \varpi_\beta \rangle \geq 0 \quad \text{ for all } \beta\in \Delta .
\end{equation}

As $\mu\in V(w)$ and $ \mathcal{A}_w \subseteq C^+ $, we have $0\leq \langle \mu, \alpha \rangle$. Then, as by assumption $ \langle \mu - \lambda, \alpha \rangle \leq -2$, we must have $ \langle \lambda, \alpha \rangle \geq 2 $. This implies that $\alpha\in\Phi^+(\lambda)$. In other words $\theta_a(\lambda - \alpha)$ is dominant so we can apply 
\Cref{prop: main tool} to compare it with $w.$ To prove $ w \leq \theta_a(\lambda - \alpha)$, by Equation \eqref{vminus}, we just need to prove that for $(\eta,\nu)$ as above, 
$$\langle \eta-\alpha-\nu, \varpi_{\beta}\rangle \geq 0\quad \text{ for all } \beta\in \Delta .$$
For $\beta \neq \alpha$ this follows directly from \eqref{eq: DOMI2}. For the case $\beta =\alpha$ we need to make the following calculation:
in \eqref{eq: DOMI1} replace $\alpha$ by $\sum_{\beta\in \Delta}\langle \alpha, \beta\rangle \varpi_{\beta}$.
Since $\lr{\alpha}{\beta} \leq 0$ for all $\beta \neq \alpha$, by using  \eqref{eq: DOMI2} we conclude $ \langle \eta - \nu, \varpi_\alpha \rangle \geq 1 $.
So we have proved that $ w \leq \theta_a(\lambda - \alpha)$, but this contradicts the fact that $ w \in PB_a(\lambda) $.
Thus, our initial assumption is false, and the proof is complete. 
\end{proof}

We now define a partition of the set $\Lambda^+$.
The motivation behind this definition is that for $a\in \mathcal{F}$ fixed,  the numbers $c_{a}(\lambda)$ remain invariant when $\lambda$ runs in the parts of this partition. 

\begin{definition}\label{def: zones}
For  $\boldsymbol{i}=(i_1, \ldots , i_n)\in \{0,1,2\}^{n} $, we define
\begin{equation}
Z(\boldsymbol{i}) =  \{  \lambda \in \Lambda^+ \mid \langle \lambda , \alpha_j \rangle = i_j \mbox{ if } i_j=0, 1, \mbox{ and } \langle \lambda , \alpha_j \rangle \geq 2 \mbox{ if }\ i_j=2 \}.
\end{equation}
We call $Z(\boldsymbol{i})$ the zone associated to $\boldsymbol{i}$. 
\end{definition}
It is clear that the set $\mathcal{Z} = \{ Z(\boldsymbol{i} )  \mid   \boldsymbol{i}\in \{0,1,2\}^{n}\}$ is a partition of $\Lambda^+$.
\begin{remark}\label{deltaphi}
Let $Z\in\mathcal{Z}$.
By the description of $\Phi^+(\lambda)$ for type $A$, we have that $\Phi^+(\lambda)$ remains constant for all $\lambda\in Z$.
We note that there is an equality of sets $\Phi^+(\lambda)=\Delta$ if and only if $\lambda\in Z(2,2,\ldots,2)$.
\end{remark}

We now prove \Cref{conj: paperboat zones} in affine type $A$ and for $W_\mathrm{aff}^+$.

\begin{theorem}\label{thm: paper boat} Let $a\in\mathcal{F}$ and $Z\in\mathcal{Z}$.
Then $c_a(\lambda)$ is constant for all $\lambda\in Z$.    
\end{theorem}

\begin{proof}
Let $\lambda , \mu \in Z$.
We must show that $c_a(\lambda) = c_a(\mu)$.
By \Cref{cor: reduction to C+ Paper Boat} it suffices to prove that $c_a^+(\lambda) = c_a^+(\mu)$.
We prove this last equality by constructing an injective map 
\begin{equation}
f:PB_a^+(\lambda) \rightarrow PB_a^+(\mu).
\end{equation}
This will imply that $c_a^+(\lambda) \leq c_a^+(\mu)$.
By interchanging the roles of $\lambda$ and $\mu$ we will obtain the opposite inequality.
Therefore, we will have $c_a^+(\lambda) = c_a^+(\mu)$ and we will conclude. 

Given $w\in PB_a^+(\lambda)$ we set $f(w) = y$ for the unique $y\in W_\mathrm{aff}$ such that $\mathcal{A}_{y}=t_{\mu - \lambda}(\mathcal{A}_w)$.
If we  show that $f$ is well defined, \ie, $f(w)\in PB_a^+(\mu)$, the injectivity comes for free, as $f$ is a translation. Proving this is equivalent to proving the following three claims:
\begin{enumerate}[(a)]
\item\label{item: fw dom} ${ f(w)\in W_\mathrm{aff}^+}$.
\item $f(w) \leq \theta_{a}(\mu)$.
\item $f(w) \not \leq \theta_{a}(\mu -\alpha)$ for all $\alpha\in \Phi^+ (\mu)$.
\end{enumerate}

The proof of these statements is as follows.

\begin{enumerate}[(a)]
\item 
It is enough to prove that for any $v\in V(f(w))$ we have $v\in C^{+}$. 
By definition of $f$ we have that  $v+\lambda -\mu\in V(w)$.
By \Cref{lemma Trasl is Domin} we have $-1 \leq \langle v -\mu , \alpha \rangle $ for all $\alpha \in \Delta$.
Therefore
\begin{equation}
-1 +\langle \mu,\alpha \rangle \leq \langle v,\alpha \rangle   .
\end{equation}
If $\langle \mu,\alpha \rangle \geq 1$ for some $\alpha\in \Delta$ then $\langle v,\alpha \rangle \geq 0 $. On the other hand, if  $\langle \mu,\alpha \rangle =0$ for some $\alpha \in \Delta$ then $\langle \lambda,\alpha \rangle =0$ since $\lambda$  and $\mu$ belong to the same zone. In this case we have
\begin{equation}
0 \leq \langle v+\lambda -\mu , \alpha  \rangle  = \langle v, \alpha  \rangle, 
\end{equation}
the inequality is because $V(w)\subset C^+.$
All in all, we conclude that  $0 \leq \langle v, \alpha \rangle$ for all $\alpha\in \Delta$, that is, $v\in C^+$.

\item  Let $v\in V(f(w))$ and $v'\in V(\theta_a(\mu))$ be such that $v'  \equiv v \mod \mathbb{Z}\Phi$.
Then, $\lambda - \mu + v'  \equiv  \lambda - \mu +v \mod \mathbb{Z}\Phi$.
We have that $ \lambda - \mu + v \in V(w)$ and $\lambda - \mu +v' \in V(\theta_a(\lambda))$. 
Since $w\in PB_{a}^+(\lambda)$, \Cref{prop: main tool} implies that 
\begin{equation}
\lambda - \mu +v  \leq \lambda - \mu + v'  
\end{equation}
in the dominance order, so that $v\leq v'$.
As $v$ was an arbitrary vertex of $f(w),$ and $f(w)\in W_\mathrm{aff}^+$ by \ref{item: fw dom}, \Cref{prop: main tool} implies $f(w) \leq \theta_{a}(\mu)$ in the Bruhat order. 

\item  We proceed by contradiction.  
Suppose there exists some $\beta \in \Phi^+(\mu)$ such that $f(w) \leq \theta_a(\mu - \beta)$.
Since $\lambda,\mu\in Z$, we have $\Phi^+(\mu) = \Phi^+(\lambda)$ by \Cref{deltaphi}.  
Consequently, as $w\in PB_a^+(\lambda)$, we have that  $w \not\leq \theta_a(\lambda - \beta)$ (in the Bruhat order).  
By \Cref{prop: main tool}, there exists a pair $(v, v') \in V(w) \times V(\theta_a(\lambda - \beta))$ with $v \equiv v'\mod\mathbb{Z}\Phi$ such that $v \not\leq v'$ (in the dominance order).  
This implies $\mu - \lambda + v \not\leq \mu - \lambda + v'$.  
However, $(\mu - \lambda + v, \mu - \lambda + v') \in V(f(w)) \times V(\theta_a(\mu - \beta))$, which contradicts our assumption by \Cref{prop: main tool}.
\end{enumerate}
\end{proof}

We finish this section by recalling \Cref{eq: lowerint implied by conjs}, i.e., combining   \Cref{thm: loweint as pbs,thm: paper boat}   we obtain a weighted counting formula for the sizes of all lower Bruhat intervals associated to elements in $W_\mathrm{aff}^+$.

\begin{cor}\label{cor:weigthed_sum}
For all $\lambda\in\Lambda^+$ and $a\in\mathcal{F}$, we have
\begin{equation}\label{eq:weighted_sum}
|\li{\lambda}| = \sum_{Z\in\mathcal{Z}} c_a(Z) |\mathfrak{I}_\lambda \cap Z |,
\end{equation}
where $c_a(Z)$ is $c_a(\mu)$ for any $\mu\in Z$.
\end{cor}

\Cref{tab:PBin A3} contains the values of the numbers $c_a(Z)$ (divided by $24$) for type $\widetilde{A}_3$.

\begin{proof}
Taking cardinalities on both sides of \eqref{eq: loweint as pbs} yields
\begin{equation}
		|\li{\lambda}| = \sum_{\mu \in \mathfrak{I}_\lambda} c_a(\mu) = \sum_{Z\in\mathcal{Z}} \sum_{\mu \in \mathfrak{I}_\lambda \cap Z} c_a(\mu) =  \sum_{Z\in\mathcal{Z}}c_a(Z) \sum_{\mu \in \mathfrak{I}_\lambda \cap Z} 1 = \sum_{Z\in\mathcal{Z}} c_a(Z) |\mathfrak{I}_\lambda \cap Z |,
\end{equation}
where the third equality follows from \Cref{thm: paper boat}.
\end{proof}

\section{On the sizes of the Paper Boats} \label{sec:SizesPB}
In this section, we continue focusing on affine type $A$ and on lower Bruhat intervals associated with elements of $W_\mathrm{aff}^+$, as in the previous section.  
In \Cref{222}, we provide a formula for the size of a Paper Boat within the zone $Z(2, \ldots, 2)$.  
In \Cref{lowerint}, we present a formula for the size of a Paper Boat expressed in terms of the sizes of lower intervals.  
An interesting implication of this formula is \Cref{cor: RecurrenceImpliesLESS}, that produces  a method to test whether a given function from the dominant weights to the integers matches with the cardinality of $\mathcal{I}_a(\lambda)$ for a fixed $a\in \mathcal{F}$.  
We then use this result in \Cref{sec: geoforA3} to compute the sizes of any lower interval coming from an element in $W_\mathrm{aff}^+$, in type $\widetilde{A}_3$, thus establishing \Cref{superconj} for this case.

\subsection{Zone \texorpdfstring{$Z(2,\ldots , 2)$}{Z(2,...,2)}}\label{222}

The goal of this section is to prove that the size of $PB_a^+(\lambda)$ is equal to $(n+1)!$ for all $\lambda\in Z(2,\ldots , 2)$ and all $a\in \mathcal{F}$. 
As we already pointed out in the introduction, this result was unexpected as the shape of (the alcoves of) the Paper Boat associated with different elements of $\mathcal{F}$ can vary but somehow, their cardinality does not. 
This discrepancy is illustrated in \Cref{fig:PPBB}.

\begin{figure}[hbt!]
\captionsetup[subfigure]{labelformat=simple}
\centering

\begin{subfigure}{0.25\textwidth}
\centering
\includegraphics[width=0.8\textwidth]{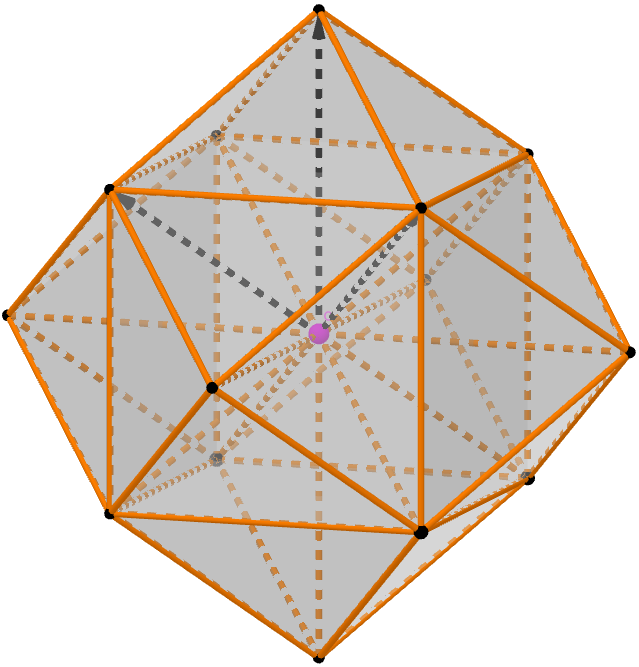}
\caption{The Paper Boat of $\theta_\mathrm{id}(\lambda)$ (whose alcove is the one with black edges).
}
\label{fig:PBAAA}
\end{subfigure}
\hfill
\begin{subfigure}{0.7\textwidth}
\centering
\includegraphics[width=\textwidth]{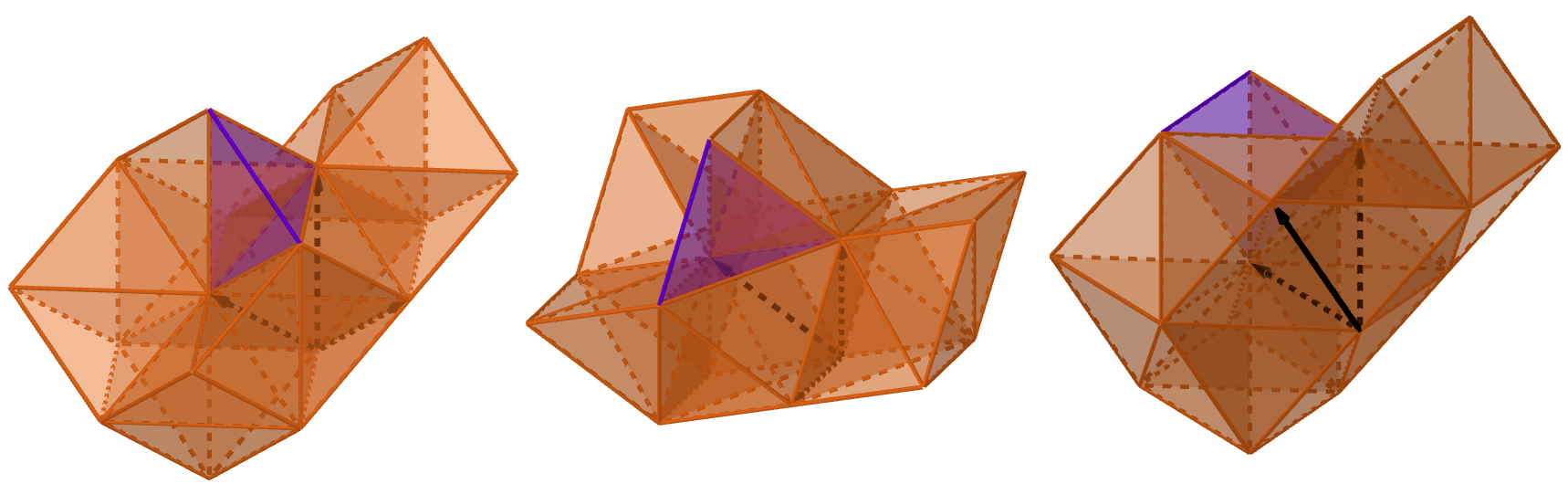}
\caption{Three distinct sights of the same Paper Boat.
The purple alcove represents $\theta_a(\lambda)$, while the alcove with black edges corresponds to $\theta_\mathrm{id}(\lambda)$ (which, for some $a$, does not belong to $PB_a(\lambda)$)}
\label{fig:PBBBB}
\end{subfigure}
\hfill
\caption{Two Paper Boats in $\widetilde{A}_3$ for the zone $Z(2,2,2)$.}
\label{fig:PPBB}
\end{figure}

\begin{prop} \label{prop:BarquitoGenerico}
Let $a\in \mathcal{F}$ and $\lambda \in Z(2, \ldots , 2)$.
Then, $c_a^+(\lambda) = |W_\mathrm{f}|=(n+1)!$.  
\end{prop}

\begin{proof}
We will construct a bijection between $W_\mathrm{f}$ and $PB_a^+(\lambda)$.
Let $\mu \in \mathbb{Z}\Phi$ and $w \in W_\mathrm{f}$ be such that $\theta_a(\lambda) = t_\mu w \in \mathbb{Z}\Phi \rtimes W_\mathrm{f} = W_\mathrm{aff}$. 
Let $v\in W_\mathrm{f}$.
For $0\leq i\leq n$ and $\alpha\in\Delta$, we write
\begin{equation}
v_i= w(-\varpi_i) - v(-\varpi_i)\in\bbZ\Phi \quad \mbox{and} \quad \eta_{\alpha,v} = \min_{0\leq i \leq n} \{ \langle v_i, \varpi_\alpha \rangle \}.
\end{equation}

Recall the meet \eqref{eq: meet} and join \eqref{eq: join} operators.
We define
\begin{equation}
\displaystyle \kappa_v = (-v_0)\vee (-v_1) \vee \cdots \vee (-v_n) 
= - (v_0\wedge v_1 \wedge \cdots \wedge v_n) 
= -\sum_{\alpha \in \Delta} \eta_{\alpha , v} \alpha.
\end{equation}
By definition of $\kappa_v$, we have
\begin{equation} \label{eq:coeffLessEqual0 domi}
0\leq v_i +\kappa_v, \qquad \forall \, 0\leq i\leq n
\end{equation}
in the dominance order.
As $v_0=\mathbf{0},$ Equation \eqref{eq:coeffLessEqual0 domi} implies that $\kappa_v\geq \mathbf{0}$, that is, $\kappa_v\in \bbZ_{\geq 0}\Phi^+$.
This implies that
\begin{equation}\label{eq: kappa with Delta_v}
\kappa_v = -\sum_{\alpha \in \Delta_v} \eta_{\alpha , v} \alpha, \quad \mbox{with} \quad \Delta_v = \{ \alpha \in \Delta \mid \langle v_i, \varpi_\alpha \rangle <0 \text{ for some } 0\leq i \leq n \}.
\end{equation}

Consider $\phi:  W_\mathrm{f}  \longrightarrow  PB_a^+(\lambda)$ given by $\phi(v)=t_{\mu - \kappa_v} v$.
To prove the proposition, it is enough to prove that $\phi$ is a bijection.
We first show that $\phi$ is well defined, \ie , $t_{\mu - \kappa_v} v\in PB_a^+(\lambda)$.
By definition of $PB_a^+(\lambda) $, we have to prove the following three statements:
\begin{enumerate}[(a)]
\item\label{item: PB222a} $t_{\mu - \kappa_v} v \in W_\mathrm{aff}^+$,
\item\label{item: PB222b} $t_{\mu - \kappa_v} v \leq \theta_a(\lambda)=t_\mu w$,
\item\label{item: PB222c} $t_{\mu - \kappa_v} v \not\leq \theta_a(\lambda-\alpha)$, for each simple root $\alpha \in \Delta$ (recall \Cref{deltaphi}).
\end{enumerate}

The proof of the above statements is as follows.
\begin{enumerate}[(a)]
\item
Suppose that $t_{\mu - \kappa_v} v \not\in W_\mathrm{aff}^+$.
Let $A$ be the alcove associated to $t_{\mu - \kappa_v} v$. 

On the one hand,  since $\lambda \in Z(2,\ldots, 2)$, we have that $\lr{t_\mu w(-\varpi_i)}{\beta}\geq 2$ for all $0 \leq i\leq n$ and all $\beta\in \Delta$.

On the other hand, since $A$ lies outside $C^+$,
there exists a vertex of $A$ that is not dominant. 
In other words, there exists   $\alpha\in \Delta$ and $0\leq i \leq n$ such that $\lr{t_{\mu - \kappa_v} v(-\varpi_i)}{\alpha} = k < 0$.
Therefore, every vertex $x$ of  $A$ satisfies $k-1 \leq \langle x, \alpha \rangle \leq k+1$.
We conclude that $\lr{t_{\mu - \kappa_v} v(-\varpi_i)}{\alpha} \leq 0$ for all $0\leq i \leq n$.

By combining these two facts for the simple root $\alpha$ and for $i=0$, we get
\begin{equation}
2 \leq \lr{\mu}{\alpha}  = \lr{\mu-\kappa_v}{\alpha} + \lr{\kappa_v}{\alpha}  \leq \lr{\kappa_v}{\alpha}. 
\end{equation}
By \eqref{eq: kappa with Delta_v}, this implies that $\eta_{\alpha,v}\neq 0$ (recall that $\lr{\alpha}{\beta} \leq 0$ for all $\beta \in \Delta$ such that  $\beta \neq \alpha$), and  this can occur only if $\alpha\in \Delta_v$. 
By definition of $\eta_{\alpha,v}$ there exists some $j\in\{1,\ldots,n\}$  such that $\eta_{\alpha,v}=\langle v_j, \varpi_\alpha \rangle$, so that  
\begin{equation}\label{eq:coeffEqual0}
\lr{v_j+\kappa_v}{\varpi_\alpha} =0.
\end{equation} 

A combination of \eqref{eq:coeffLessEqual0 domi}  and  \eqref{eq:coeffEqual0}  yields $ \lr{v_j+\kappa_v}{\alpha} \leq 0$.
Recall that $v_j=w(-\varpi_j) - v(-\varpi_j)$.
We have $
2   \leq \lr{t_\mu w(-\varpi_j)}{\alpha}  = \lr{v_j+\kappa_v}{\alpha} + \lr{ t_{\mu-\kappa_v}v(-\varpi_j)}{\alpha}  \leq 0$.

This contradiction proves \ref{item: PB222a}. 

\item By part \ref{item: PB222a} we know that $t_{\mu-\kappa_v}v\in W_\mathrm{aff}^+$. 
Applying \Cref{prop: main tool}, we see that \ref{item: PB222b} is equivalent to \eqref{eq:coeffLessEqual0 domi}.

\item Let $\alpha \in \Delta$.
We first note that since $\lambda$ is in the zone $Z(2,\ldots,2)$, then $\lambda-\alpha\in\Lambda^+$.
Therefore the equality $\theta_a(\lambda)=t_{\mu} w$ and \Cref{lem: amigos algebraic def} imply that $\theta_a(\lambda-\alpha)=t_{\mu - \alpha} w$.
We  suppose that  $t_{\mu - \kappa_v} v \leq t_{\mu - \alpha} w$ and will derive a contradiction.
By part \ref{item: PB222a} and \Cref{prop: main tool}, we get
\begin{equation} \label{eq:SIFuera}
\mathbf{0} \leq \kappa_v + v_i - \alpha, \qquad \forall \, i\in \{0,1,\ldots , n\}
\end{equation}
in the dominance order.
In particular, $i=0$ yields  $\alpha \leq \kappa_v$, so $-\eta_{\alpha,v}\geq 1$, and in particular $\eta_{\alpha,v}\neq 0,$
thus $\alpha \in \Delta_v$. As before \eqref{eq:coeffEqual0}  holds for some $j\in \{1,\ldots , n\}$.
By pairing  \eqref{eq:SIFuera} for $i=j$ with $\varpi_\alpha$, we obtain $0\leq -1$.
This contradiction proves \ref{item: PB222c}. 
\end{enumerate}

We have shown that $\phi$ is well defined. 

The injectivity of the map $\phi$ follows directly from the definition of semidirect product.
Let us now prove that $\phi$ is surjective.
Let $t_{\mu- \gamma}v \in PB_a^+(\lambda)$ for some $\gamma \in \mathbb{Z}\Phi$ and $v\in W_\mathrm{f}$.
We must show that $\gamma =\kappa_v$.

We begin by noticing that, since $t_{\mu- \gamma}v\leq t_{\mu}w$, \Cref{prop: main tool} implies that $-v_i\leq \gamma$ for all $0\leq i \leq n$ in the dominance order.
Therefore, $(-v_0) \vee \cdots \vee(-v_n) = \kappa_v \leq \gamma$.
To prove the other inequality it is enough to show the following.
\begin{enumerate}[(I)]
\item \label{itemI} $\Delta_v  = \operatorname{Supp}(\gamma)$, where $\operatorname{Supp}(\gamma) = \{\alpha \in \Delta \mid \lr{\gamma}{\varpi_\alpha} \neq 0  \}$.   
\item \label{itemII} $\lr{\gamma}{\varpi_\alpha}\leq \lr{\kappa_{v}}{\varpi_\alpha} = -\eta_{\alpha ,v}$ for all $\alpha\in \Delta_v$. 
\end{enumerate}

We first prove \Cref{itemI}.
Since $0\leq \kappa_v \leq \gamma$ we have $\Delta_v  \subset \operatorname{Supp}(\gamma)$.
On the other hand, let us suppose that $\alpha \in \operatorname{Supp}(\gamma) \setminus \Delta_v$.
Then, for all $i\in \{0,1,,\ldots ,n\}$  we have
\begin{equation}\label{eq:condition_not}
\langle \gamma - \alpha + v_i, \varpi_\alpha \rangle  
= \langle \gamma, \varpi_\alpha \rangle - 1 + \langle v_i, \varpi_\alpha \rangle 
\geq 1 - 1 + \langle v_i, \varpi_\alpha \rangle 
\geq 0,
\end{equation} 
where the first inequality follows from the fact that $\gamma \geq 0$.
We notice that \eqref{eq:coeffLessEqual0 domi} implies that  $0 \leq  \gamma +v_i $ for all $i\in \{0,1,\ldots ,n\}$. 
Let $\beta \in \Delta$ with $\beta \neq \alpha$. Then we have
\begin{equation}
\langle \gamma - \alpha + v_i, \varpi_\beta \rangle   = \langle \gamma  + v_i, \varpi_\beta \rangle \geq 0
\end{equation}
It follows that  $\mathbf{0} \leq  \gamma-\alpha +v_i$ for all $i\in \{0,1,,\ldots ,n\}$. 
Then \Cref{prop: main tool} implies that $t_{\mu-\gamma}v \leq t_{\mu-\alpha}w$.
This contradicts the fact that $t_{\mu- \gamma}v \in PB_a^+(\lambda)$.   
Thus, we conclude that $\operatorname{Supp}(\gamma) = \Delta_v$.

Finally, we prove \Cref{itemII}. 
Suppose that  $\langle \gamma, \varpi_\alpha \rangle > -\eta_{\alpha,v}$ for some $\alpha \in \Delta_v$.  
Let $0\leq i\leq n$ and note that
\begin{equation}
\langle \gamma + v_i, \varpi_\alpha \rangle
> -\eta_{\alpha , v}  + \lr{v_i}{\varpi_\alpha} \geq  -\eta_{\alpha , v}  + \eta_{\alpha , v} =0.
\end{equation}
It follows that $\langle \gamma - \alpha + v_i, \varpi_\alpha \rangle \geq 0$.
This is the same inequality as the one obtained in \eqref{eq:condition_not}.
Therefore, we reach the same contradiction and this proves \Cref{itemII}. 

We have proved that $\gamma = \kappa_v$.
Therefore, $\phi$ is bijective as we wanted to show.
\end{proof}

\begin{example}
In \Cref{fig:DibFede}, we graphically illustrate the bijection $\phi$ occurring in the proof of \Cref{prop:BarquitoGenerico} for type $\tilde{A}_2$.
Let $a=s_0$.
In \Cref{fig:Fede} (resp. \Cref{fig:FedeAA}), the green triangles represent the set $PB_a^+(2\varpi_1+2\varpi_2)$ (resp. $PB_a^+(3\varpi_1+2\varpi_2)$) and the corresponding element $\theta_a(2\varpi_1 + 2\varpi_2)$ (resp. $\theta_a(3\varpi_1 + 2\varpi_2)$) is represented by the green triangle marked with $\circled{0}$ (resp. $\circled{3}$).
\begin{figure}[ht]
\captionsetup[subfigure]{labelformat=simple}
\centering

\begin{subfigure}{0.4\textwidth}
\centering
\begin{tikzpicture}[x=0.75cm, y=0.75cm,>=Latex]
% el origen esta en  (0,0)++(60:9)++(4,0) 

\clip (3,6) rectangle + (11,9);

% \filldraw[lightgray] (0,0)++(60:9)++(4,0) ++ (60:3)++(120:3)--++
% (240:1)--++(-1,0)--++
% (240:1)--++(-1,0)--++
% (240:1)--++(-1,0)--++
% (-60:1)--++(240:1)--++
% (-60:1)--++(240:1)--++
% (-60:1)--++(240:1)--++
% (1,0)--++(-60:1)--++
% (1,0)--++(-60:1)--++
% (1,0)--++(-60:1)--++
% (60:1)--++(1,0)--++
% (60:1)--++(1,0)--++
% (60:1)--++(1,0)--++
% (120:1)--++(60:1)--++
% (120:1)--++(60:1)--++
% (120:1)--++(60:1)--++
% (-1,0)--++(120:1)--++
% (-1,0)--++(120:1)--++
% (-1,0)--++(120:1)
% -- cycle;

\filldraw[green] (0,0)++(60:9)++(4,0) ++ (60:3)++(120:3)--++
(240:1)--++(-1,0)--++(-60:1)--++(2,0)--++(60:1)--++(-1,0)
-- cycle;

\foreach \i in {0,...,8} {
     \foreach \j in {-3,...,12}{
  \draw[gray, thick](0:0)++(1+3*\i,0)++(60:\j)++(120:\j)--++(60:1)--++(120:1)--++(-1,0)--++(240:1)--++(300:1)--++(1,0);
  }}
  \foreach \i in {0,...,5} {
     \foreach \j in {-3,...,12} {
  \draw[gray, thick](0:0)++(60:1)++(1,0)++(1+3*\i,0)++(60:\j)++(120:\j)--++(60:1)--++(120:1)--++(-1,0)--++(240:1)--++(300:1)--++(1,0);
  }}
\foreach \i in {0,...,6} {
     \foreach \j in {-3,...,12}{
  \draw[gray, thick](0:0)++(60:1)++(1+3*\i,0)++(60:\j)++(120:\j)--++(60:1)--++(120:1)--++(-1,0)--++(240:1)--++(300:1)--++(1,0);
  }}  
  \foreach \i in {0,...,5} {
     \foreach \j in {-3,...,12} {
  \draw[gray, thick](0:0)++(2,0)++(1+3*\i,0)++(60:\j)++(120:\j)--++(60:1)--++(120:1)--++(-1,0)--++(240:1)--++(300:1)--++(1,0);
  }}
\foreach \i in {0,...,6} {
     \foreach \j in {-3,...,12}{
  \draw[gray, thick](0:0)++(-1,0)++(60:1)++(1+3*\i,0)++(60:\j)++(120:\j)--++(60:1)--++(120:1)--++(-1,0)--++(240:1)--++(300:1)--++(1,0);
  }}  
  \foreach \i in {0,...,5} {
     \foreach \j in {-3,...,12}{
  \draw[gray, thick](0:0)++(1,0)++(1+3*\i,0)++(60:\j)++(120:\j)--++(60:1)--++(120:1)--++(-1,0)--++(240:1)--++(300:1)--++(1,0);
  }}

\draw[black, ultra thick]  (0,0)++(60:9)++(4,0)++(-9,0)--++(17,0);
\draw[black, ultra thick]  (0,0)++(60:9)++(4,0)++(120:9)--++(300:18);
\draw[black, ultra thick]  (0,0)++(60:9)++(4,0)++(60:9)--++(240:18);

\draw[black, ultra thick]  (0,0)++(60:9)++(4,0)++(60:4)++(120:3)--++(-1,0)--++(240:1)--++(-60:1)--++(1,0)--++(60:1)--++(120:1);

\draw[black, ultra thick]  (0,0)++(60:9)++(4,0)++(60:1)--++(-1,0)--++(240:1)--++(-60:1)--++(1,0)--++(60:1)--++(120:1);

%  PUNTOS
\filldraw [white] (0,0)++(60:9)++(4,0) ++ (60:3)++(120:3)  circle (3pt);
\draw[black ,very thick]  (0,0)++(60:9)++(4,0) ++ (60:3)++(120:3)  ++ (-60:0)  circle(3pt) ;

\filldraw [white] (0,0)++(60:9)++(4,0) ++ (60:2)++(120:2)  circle (3pt);
\draw[black ,very thick]  (0,0)++(60:9)++(4,0) ++ (60:2)++(120:2)  ++ (-60:0)  circle(3pt) ;

 \draw (0,0)++(60:9)++(4,0) ++ (60:3)++(120:3)++(0,0.55) node[draw, circle, inner sep=0pt,minimum size=1pt] {$5$};

\draw (0,0)++(60:9)++(4,0) ++ (60:3)++(120:3)++(0,0.55)++(210:0.5) ++(-0.08,0) node[draw, circle, inner sep=0pt,minimum size=1pt] {$4$};

 \draw (0,0)++(60:9)++(4,0) ++ (60:3)++(120:3)++(0,0.55)++(240:1) node[draw, circle, inner sep=0pt,minimum size=1pt] {$2$};

 \draw (0,0)++(60:9)++(4,0) ++ (60:3)++(120:3)++(0,0.55)++(210:0.5) ++(-0.08,0)++(1,0) node[draw, circle, inner sep=0pt,minimum size=1pt] {$3$};
  \draw (0,0)++(60:9)++(4,0) ++ (60:3)++(120:3)++(0,0.55)++(240:1)++(1,0) node[draw, circle, inner sep=0pt,minimum size=1pt] {$1$};

 \draw (0,0)++(60:9)++(4,0) ++ (60:3)++(120:3)++(0,0.55)++(240:1)++(0.5,0)++(0,-0.25) node[draw, circle, inner sep=0pt,minimum size=1pt] {$0$};

 \draw (0,0)++(60:9)++(4,0) ++ (60:3)++(120:3)++(0,0.55)++(240:1)++(0.5,0)++(0,-0.25)++(0,-0.6)++(-1,0) node[draw, circle, inner sep=0pt,minimum size=1pt] {$1$};

 \draw (0,0)++(60:9)++(4,0) ++ (60:3)++(120:3)++(0,0.55)++(240:1)++(0.5,0)++(0,-0.25)++(0,-0.6)++(0,0) node[draw, circle, inner sep=0pt,minimum size=1pt] {$5$};

 \draw (0,0)++(60:9)++(4,0) ++ (60:3)++(120:3)++(0,0.55)++(240:1)++(0.5,0)++(0,-0.25)++(0,-0.6)++(1,0) node[draw, circle, inner sep=0pt,minimum size=1pt] {$2$};

 \draw (0,0)++(60:9)++(4,0) ++ (60:3)++(120:3)++(0,0.55)++(240:1)++(0.5,0)++(0,-0.25)++(0,-0.6)++(-0.5,0)++(0,-0.2) node[draw, circle, inner sep=0pt,minimum size=1pt] {$4$};
 \draw (0,0)++(60:9)++(4,0) ++ (60:3)++(120:3)++(0,0.55)++(240:1)++(0.5,0)++(0,-0.25)++(0,-0.6)++(-0.5,0)++(0,-0.2)++(1,0) node[draw, circle, inner sep=0pt,minimum size=1pt] {$3$};

  \draw (0,0)++(60:9)++(4,0) ++(0,0.55) node[draw, circle, inner sep=0pt,minimum size=1pt] {$5$};

\draw (0,0)++(60:9)++(4,0)++(0,0.55)++(210:0.5) ++(-0.08,0) node[draw, circle, inner sep=0pt,minimum size=1pt] {$4$};

 \draw (0,0)++(60:9)++(4,0)++(0,0.55)++(240:1) node[draw, circle, inner sep=0pt,minimum size=1pt] {$2$};

 \draw (0,0)++(60:9)++(4,0) ++(0,0.55)++(210:0.5) ++(-0.08,0)++(1,0) node[draw, circle, inner sep=0pt,minimum size=1pt] {$3$};
  \draw (0,0)++(60:9)++(4,0) ++(0,0.55)++(240:1)++(1,0) node[draw, circle, inner sep=0pt,minimum size=1pt] {$1$};

 \draw (0,0)++(60:9)++(4,0) ++(0,0.55)++(240:1)++(0.5,0)++(0,-0.25) node[draw, circle, inner sep=0pt,minimum size=1pt] {$0$};
 
\end{tikzpicture}
\caption{$\lambda = 2\varpi_1+2\varpi_2$.}
\label{fig:Fede}
\end{subfigure}
\hfill
\begin{subfigure}{0.4\textwidth}
\centering
\begin{tikzpicture}[x=0.75cm, y=0.75cm,>=Latex]
% el origen esta en  (0,0)++(60:9)++(4,0) 

\clip (3,6) rectangle + (11,9);

% \filldraw[lightgray] (0,0)++(60:9)++(4,0) ++ (60:3)++(120:3)--++
% (240:1)--++(-1,0)--++
% (240:1)--++(-1,0)--++
% (240:1)--++(-1,0)--++
% (-60:1)--++(240:1)--++
% (-60:1)--++(240:1)--++
% (-60:1)--++(240:1)--++
% (1,0)--++(-60:1)--++
% (1,0)--++(-60:1)--++
% (1,0)--++(-60:1)--++
% (60:1)--++(1,0)--++
% (60:1)--++(1,0)--++
% (60:1)--++(1,0)--++
% (120:1)--++(60:1)--++
% (120:1)--++(60:1)--++
% (120:1)--++(60:1)--++
% (-1,0)--++(120:1)--++
% (-1,0)--++(120:1)--++
% (-1,0)--++(120:1)
% -- cycle;

\filldraw[green] (0,0)++(60:9)++(4,0) ++ (60:4)++(120:3)--++
(240:1)--++(-1,0)--++(-60:1)--++(2,0)--++(60:1)--++(-1,0)
-- cycle;

\foreach \i in {0,...,8} {
     \foreach \j in {-3,...,12}{
  \draw[gray, thick](0:0)++(1+3*\i,0)++(60:\j)++(120:\j)--++(60:1)--++(120:1)--++(-1,0)--++(240:1)--++(300:1)--++(1,0);
  }}
  \foreach \i in {0,...,5} {
     \foreach \j in {-3,...,12} {
  \draw[gray, thick](0:0)++(60:1)++(1,0)++(1+3*\i,0)++(60:\j)++(120:\j)--++(60:1)--++(120:1)--++(-1,0)--++(240:1)--++(300:1)--++(1,0);
  }}
\foreach \i in {0,...,6} {
     \foreach \j in {-3,...,12}{
  \draw[gray, thick](0:0)++(60:1)++(1+3*\i,0)++(60:\j)++(120:\j)--++(60:1)--++(120:1)--++(-1,0)--++(240:1)--++(300:1)--++(1,0);
  }}  
  \foreach \i in {0,...,5} {
     \foreach \j in {-3,...,12} {
  \draw[gray, thick](0:0)++(2,0)++(1+3*\i,0)++(60:\j)++(120:\j)--++(60:1)--++(120:1)--++(-1,0)--++(240:1)--++(300:1)--++(1,0);
  }}
\foreach \i in {0,...,6} {
     \foreach \j in {-3,...,12}{
  \draw[gray, thick](0:0)++(-1,0)++(60:1)++(1+3*\i,0)++(60:\j)++(120:\j)--++(60:1)--++(120:1)--++(-1,0)--++(240:1)--++(300:1)--++(1,0);
  }}  
  \foreach \i in {0,...,5} {
     \foreach \j in {-3,...,12}{
  \draw[gray, thick](0:0)++(1,0)++(1+3*\i,0)++(60:\j)++(120:\j)--++(60:1)--++(120:1)--++(-1,0)--++(240:1)--++(300:1)--++(1,0);
  }}

\draw[black, ultra thick]  (0,0)++(60:9)++(4,0)++(-9,0)--++(17,0);
\draw[black, ultra thick]  (0,0)++(60:9)++(4,0)++(120:9)--++(300:18);
\draw[black, ultra thick]  (0,0)++(60:9)++(4,0)++(60:9)--++(240:18);

\draw[black, ultra thick]  (0,0)++(60:9)++(4,0)++(60:4)++(120:3)--++(-1,0)--++(240:1)--++(-60:1)--++(1,0)--++(60:1)--++(120:1);

\draw[black, ultra thick]  (0,0)++(60:9)++(4,0)++(60:1)--++(-1,0)--++(240:1)--++(-60:1)--++(1,0)--++(60:1)--++(120:1);

%  PUNTOS
\filldraw [white] (0,0)++(60:9)++(4,0) ++ (60:3)++(120:3)  circle (3pt);
\draw[black ,very thick]  (0,0)++(60:9)++(4,0) ++ (60:3)++(120:3)  ++ (-60:0)  circle(3pt) ;

\filldraw [white] (0,0)++(60:9)++(4,0) ++ (60:3)++(120:2)  circle (3pt);
\draw[black ,very thick]  (0,0)++(60:9)++(4,0) ++ (60:3)++(120:2)  ++ (-60:0)  circle(3pt) ;

 \draw (0,0)++(60:9)++(4,0) ++ (60:3)++(120:3)++(0,0.55) node[draw, circle, inner sep=0pt,minimum size=1pt] {$5$};

\draw (0,0)++(60:9)++(4,0) ++ (60:3)++(120:3)++(0,0.55)++(210:0.5) ++(-0.08,0) node[draw, circle, inner sep=0pt,minimum size=1pt] {$4$};

 \draw (0,0)++(60:9)++(4,0) ++ (60:3)++(120:3)++(0,0.55)++(240:1) node[draw, circle, inner sep=0pt,minimum size=1pt] {$2$};

 \draw (0,0)++(60:9)++(4,0) ++ (60:3)++(120:3)++(0,0.55)++(210:0.5) ++(-0.08,0)++(1,0) node[draw, circle, inner sep=0pt,minimum size=1pt] {$3$};
  \draw (0,0)++(60:9)++(4,0) ++ (60:3)++(120:3)++(0,0.55)++(240:1)++(1,0) node[draw, circle, inner sep=0pt,minimum size=1pt] {$1$};

 \draw (0,0)++(60:9)++(4,0) ++ (60:3)++(120:3)++(0,0.55)++(240:1)++(0.5,0)++(0,-0.25) node[draw, circle, inner sep=0pt,minimum size=1pt] {$0$};

 % \draw (0,0)++(60:9)++(4,0) ++ (60:3)++(120:3)++(0,0.55)++(240:1)++(0.5,0)++(0,-0.25)++(0,-0.6)++(-1,0) node[draw, circle, inner sep=0pt,minimum size=1pt] {$1$};

 \draw (0,0)++(60:9)++(4,0) ++ (60:3)++(120:3)++(0,0.55)++(240:1)++(0.5,0)++(0,-0.25)++(0,-0.6)++(0,0)++(30:1.7)++(0.05,-0.05) node[draw, circle, inner sep=0pt,minimum size=1pt] {$5$};

 % \draw (0,0)++(60:9)++(4,0) ++ (60:3)++(120:3)++(0,0.55)++(240:1)++(0.5,0)++(0,-0.25)++(0,-0.6)++(1,0) node[draw, circle, inner sep=0pt,minimum size=1pt] {$2$};

 \draw (0,0)++(60:9)++(4,0) ++ (60:3)++(120:3)++(0,0.55)++(240:1)++(0.5,0)++(0,-0.25)++(0,-0.6)++(-0.5,0)++(0,-0.2)++(30:1.7)++(0.05,-0.05) node[draw, circle, inner sep=0pt,minimum size=1pt] {$4$};
 % \draw (0,0)++(60:9)++(4,0) ++ (60:3)++(120:3)++(0,0.55)++(240:1)++(0.5,0)++(0,-0.25)++(0,-0.6)++(-0.5,0)++(0,-0.2)++(1,0) node[draw, circle, inner sep=0pt,minimum size=1pt] {$3$};

  \draw (0,0)++(60:9)++(4,0) ++(0,0.55) node[draw, circle, inner sep=0pt,minimum size=1pt] {$5$};

\draw (0,0)++(60:9)++(4,0)++(0,0.55)++(210:0.5) ++(-0.08,0) node[draw, circle, inner sep=0pt,minimum size=1pt] {$4$};

 \draw (0,0)++(60:9)++(4,0)++(0,0.55)++(240:1) node[draw, circle, inner sep=0pt,minimum size=1pt] {$2$};

 \draw (0,0)++(60:9)++(4,0) ++(0,0.55)++(210:0.5) ++(-0.08,0)++(1,0) node[draw, circle, inner sep=0pt,minimum size=1pt] {$3$};
  \draw (0,0)++(60:9)++(4,0) ++(0,0.55)++(240:1)++(1,0) node[draw, circle, inner sep=0pt,minimum size=1pt] {$1$};

 \draw (0,0)++(60:9)++(4,0) ++(0,0.55)++(240:1)++(0.5,0)++(0,-0.25) node[draw, circle, inner sep=0pt,minimum size=1pt] {$0$};
 
\end{tikzpicture}
\caption{$\lambda = 3\varpi_1+2\varpi_2$.}
\label{fig:FedeAA}
\end{subfigure}
\caption{The bijection $\phi: W_\mathrm{f} \longrightarrow PB_a^+(\lambda)$, for $\lambda\in Z(2,\ldots,2)$.
The set $PB_a^+(\lambda)$ is illustrated in green.}
\label{fig:DibFede}
\end{figure}

The elements of $W_{\mathrm{f}}$ are labeled as follows:
\begin{equation*}
\circled{0}=e, \qquad \circled{1}=s_1   , \qquad \circled{2}=s_2   , \qquad \circled{3}= s_1s_2  , \qquad \circled{4}= s_2s_1  , \qquad \circled{5}= s_1s_2s_1  .
\end{equation*}

The bijection in the proof of \Cref{prop:BarquitoGenerico} goes as follows. 
For $v\in W_{\mathrm{f} }$ we first translate the corresponding alcove by $\mu=3\varpi_1+3\varpi_2$.
Then we translate by $-\kappa_v$ to obtain the green triangle in $PB_a^+(\lambda)$. 
\end{example}

\subsection{Paper boats in terms of lower intervals}\label{lowerint}
In this section we use the principle of inclusion-exclusion to obtain a formula for the size of a Paper Boat in terms of lower Bruhat intervals. 
Recall the weight poset $\Lambda$ and its join \eqref{eq: join} and meet \eqref{eq: meet} operators.
For $J\subset \Phi$, we define $\displaystyle \vee_J = \bigvee_{\alpha \in J} \alpha$.

\begin{prop}\label{thm: recurrence}
Let $a\in \mathcal{F}$ and $\lambda \in \Lambda^+$.
Then 
\begin{equation}\label{eq: recurrencia}
c_a(\lambda) = \sum_{J\subset \Phi^+(\lambda)} (-1)^{|J|} \, \left|\,\li{\lambda -\vee_J} \,\right|.
\end{equation}
\end{prop}

In the proof of \Cref{thm: recurrence}, we will make use of the following lemma.

\begin{lemma}\label{lem: intersection is meet}
Let $\lambda\in\Lambda^+$, $a\in\mathcal{F}$ and $\emptyset\neq J\subset\Phi^+(\lambda)$.
Then,
\begin{equation}\label{eq: intersection is meet}
\bigcap_{\alpha \in J} \li{\lambda -\alpha}= \mathcal{I}_a\left({\bigwedge_{\alpha\in J}}(\lambda -\alpha)\right) = \li{\lambda -\vee_J}.
\end{equation}
\end{lemma}

\begin{proof}
We note that by Equations \eqref{eq: join}, \eqref{eq: meet} and the fact that $\wedge^+$ is the restriction of $\wedge$ to the dominant weights, we have
\begin{equation}\label{eq: meetplus}
{\bigwedge_{\alpha\in J}}^+ \left(\lambda-\alpha\right) =  \bigwedge_{\alpha\in J} \left(\lambda-\alpha\right) = \lambda - \vee_J.
\end{equation}
Therefore, $\bigwedge_{\alpha\in J} \left(\lambda-\alpha\right) = \lambda - \vee_J$ is a dominant weight, so that \eqref{eq: intersection is meet} makes sense. 
The first equality in \Cref{eq: intersection is meet} is a direct application of \Cref{lem: meet pesos}, whereas the second equality follows from \eqref{eq: meetplus}.
\end{proof}

\begin{proof}[Proof of \Cref{thm: recurrence}]
Note that
\begin{equation}\label{eq: recu1}
\displaystyle  - \left|  \bigcup_{\alpha \in \Phi^+(\lambda)} \li{\lambda-\alpha}  \right|     \displaystyle = \sum_{\emptyset \neq J\subset \Phi^+(\lambda)}  (-1)^{|J|} \left| \bigcap_{\alpha \in J} \li{\lambda-\alpha }  \right|  
\displaystyle =  \sum_{\emptyset \neq J\subset \Phi^+(\lambda)}  (-1)^{|J|} \left| \li{\lambda - \vee_J}  \right|,
\end{equation}
where the first equality follows from the inclusion-exclusion principle and the second one is implied by \Cref{lem: intersection is meet}.
By the definition of $c_a(\lambda)$, \eqref{eq: recu1} implies \eqref{eq: recurrencia}. 
\end{proof}

\begin{example}
Using \Cref{thm: recurrence}, we compute $ c_a(\lambda) $ for all $ \lambda \in \Lambda^+ $ in the case of the affine Weyl group of type $ A_3 $.
Since the set $ \mathcal{Z} $ (\Cref{def: zones}) is a partition of  $ \Lambda^+ $ and the values $ c_a(\lambda) $ remain constant within each zone $ Z \in \mathcal{Z} $ (\Cref{thm: paper boat}), it suffices to compute $ c_a(\lambda) $ for a single representative $ \lambda_Z $ per zone $ Z $.
By definition of $ c_a(\lambda) $, we must also compute the set $ \Phi^+(\lambda) $.
As before, it is sufficient to determine $ \Phi^+(\lambda_Z) $ for each $ Z \in \mathcal{Z} $, since the sets $ \Phi^+(\lambda) $ are constant within zones (\Cref{deltaphi}).
For the zone $Z=Z(i_1, i_2, i_3)$, we set $\lambda_Z=(i_1, i_2, i_3)= i_1 \varpi_1 + i_2 \varpi_2 + i_3 \varpi_3$.
The sets $ \Phi^+(\lambda) $ are presented in the columns of \Cref{tab:Phiplusin A3}, while the values $ c_a^+(\lambda) $ are displayed in \Cref{tab:PBin A3}.
Finally, we emphasize that the values of $ c_a(\lambda) $ are obtained by multiplying $ c_a^+(\lambda) $ by $ |W_{\mathrm{f}}| = 24 $ (\Cref{cor: reduction to C+ Paper Boat}).

\begin{table}[h]
\centering
\resizebox{15cm}{!}{
\begin{tabular}{|c||c|c|c|c|c|c|c|c|c|c|c|c|c|c|c|c|c|c|c|c|c|c|c|c|c|c|c|} \hline
$i_1$    & 0 &0 & 0&0 & 0 & 0 & 0& 0& 0& 1 & 1 &1 & 1& 1& 1 & 1 &1 & 1& 2& 2 & 2 & 2& 2& 2& 2& 2 &2 \\ \hline
$i_2$ & 0 &0 &0 &1 & 1 & 1 & 2& 2& 2& 0 & 0 &0 &1 & 1 & 1 & 2 & 2& 2& 0& 0 &0  &1 &1 &1 &2 & 2 &2 \\ \hline
$i_3$    & 0 &1 &2 &0 & 1 & 2 & 0 &1 &2 & 0 & 1 &2 &0 & 1& 2 & 0 &1 &2 & 0& 1 & 2 &0 & 1& 2& 0& 1 &2 \\ \hline \hline
$\alpha_1$     &  & & & &  &  & & & &  &  & & & &  &  & & & $\bullet$ & $\bullet$ & $\bullet$ & $\bullet$ &$\bullet$ &$\bullet$ &$\bullet$ &$\bullet$  &$\bullet$  \\ \hline
$\alpha_2$     &  & & & &  &  & $\bullet$ & $\bullet$& $\bullet$&  &  & & & &  & $\bullet$ &$\bullet$ &$\bullet$ & &  &  & & & &$\bullet$ & $\bullet$ &$\bullet$ \\ \hline
$\alpha_3$     &  & & $\bullet$ & &  & $\bullet$ & & & $\bullet$ &  &  & $\bullet$ & & & $\bullet$ &  & &$\bullet$ & &  & $\bullet$ & & &$\bullet$ & &  & $\bullet$ \\ \hline
$\alpha_{12}$  &  & & & &  &  & & & &  &  & & $\bullet$ &$\bullet$  & $\bullet$  &  & & & &  &  & & & & &  & \\ \hline 
$\alpha_{23}$  &  & & & &$\bullet$  &  & & & &  &  & & & $\bullet$&  &  & & & &  &  & &$\bullet$ & & &  & \\ \hline
$\alpha_{13}$ &  & & & &  &  & & & &  &$\bullet$   & & & &  &  & & & &  &  & & & & &  & \\ \hline   
\end{tabular}
}
\caption{The sets $\Phi^+(i_1,i_2,i_3)$.}
\label{tab:Phiplusin A3}
\end{table}

\begin{table}[h]
\centering
\resizebox{15cm}{!}{
\begin{tabular}{|c||c|c|c|c|c|c|c|c|c|c|c|c|c|c|c|c|c|c|c|c|c|c|c|c|c|c|c|}  \hline
$i_1$    & 0 &0 & 0&0 & 0 & 0 & 0& 0& 0& 1 & 1 &1 & 1& 1& 1 & 1 &1 & 1& 2& 2 & 2 & 2& 2& 2& 2& 2 &2 \\ \hline
$i_2$ & 0 &0 &0 &1 & 1 & 1 & 2& 2& 2& 0 & 0 &0 &1 & 1 & 1 & 2 & 2& 2& 0& 0 &0  &1 &1 &1 &2 & 2 &2 \\ \hline
$i_3$    & 0 &1 &2 &0 & 1 & 2 & 0 &1 &2 & 0 & 1 &2 &0 & 1& 2 & 0 &1 &2 & 0& 1 & 2 &0 & 1& 2& 0& 1 &2 \\ \hline \hline
$a_1$   &  1 & 4 & 4 & 6& 12 & 12 & 6 & 12 & 12& 4 & 12 & 12 & 12 & 24 & 24 & 12 & 24 & 24 & 4& 12& 12 & 12 & 24& 24 & 12 & 24 & 24  \\ \hline
$a_2$  & 2 &7&6&10&17&16&8&15&14&7&18&16&17&30&28&15&28&26&6&16&14&16&28&26&14&26&24\\ \hline
$a_3$  & 3&10&9&13&22&22&9&18&18&9&22&21&18&33&33&15&30&30&6&16&15&15&27&27&12&24&24\\ \hline
$a_4$  &3&9&6&13&18&15&9&15&12&10&22&16&22&33&27&18&30&24&9&21&15&22&33&27&18&30&24\\ \hline
$a_5$ &5&14&10&18&24&22&9&16&14&14&29&24&24&37&34&16&29&26&10&24&20&22&34&32&14&26&24\\ \hline
$a_6$ & 8 &20&12 &26&31&26&14&23& 18 &20& 36& 26&31& 44& 36& 23&38&30& 12&  26&18& 26& 36& 30&18& 30& 24 \\ \hline
\end{tabular}
}
\caption{Values of $c_a^+(i_1,i_2,i_3)$.}
\label{tab:PBin A3}
\end{table}
\end{example}

We obtain the following corollary. 
\begin{cor} \label{cor: RecurrenceImpliesLESS}
Let $ a \in \mathcal{F} $ and $ F_a: \Lambda^+ \rightarrow \mathbb{R} $ be a function. Suppose that for all $ \lambda \in \Lambda^+ $,
\begin{equation} \label{eq: coincide recurrence}
c_a(\lambda) = \sum_{J \subset \Phi^+(\lambda)} (-1)^{|J|} F_a(\lambda - \vee_J).
\end{equation}
Then, $ F_a(\lambda) = |\mathcal{I}_a(\lambda)| $ for all $ \lambda \in \Lambda^+ $.
\end{cor}

\begin{proof}  
We proceed by induction on the dominance order on $ \Lambda^+ $. 
The minimal elements in this order are $ \varpi_0, \varpi_1, \dots, \varpi_n $. 
For these elements, we have $ \Phi^+(\varpi_i) = \emptyset $. 
Therefore, by the definition of the Paper Boat, we obtain $ PB_a(\varpi_i) = \mathcal{I}_a(\varpi_i) $. 
This implies that $c_a(\varpi_i) = |\mathcal{I}_a(\varpi_i)|$. 
On the other hand, by \eqref{eq: coincide recurrence}, we also have $ c_a(\varpi_i) = F_a(\varpi_i) $. 
Thus, we conclude that  $F_a(\varpi_i) = |\mathcal{I}_a(\varpi_i)|$,
which proves the result for the minimal elements in the dominance order.

Now, let $ \lambda \in \Lambda^+ $  and assume that $ F_a(\mu) = |\mathcal{I}_a(\mu)| $ for all $ \mu \in \Lambda^+ $ such that $ \mu < \lambda $. By combining \Cref{thm: recurrence}, our inductive hypothesis and  \eqref{eq: coincide recurrence}, we get
\begin{equation}
\begin{array}{rl}
|\mathcal{I}_a(\lambda)| & \displaystyle = c_a(\lambda) - \sum_{\emptyset \neq J \subset \Phi^+(\lambda)} (-1)^{|J|} |\mathcal{I}_a(\lambda - \vee_J)| \\
&\displaystyle  = c_a(\lambda) - \sum_{ \emptyset \neq J \subset \Phi^+(\lambda)} (-1)^{|J|} F_a(\lambda - \vee_J) \\
&= F_a(\lambda).
\end{array}
\end{equation}
This completes the proof by induction.
\end{proof}

\subsection{Geometric formula in $\widetilde{A_3}$}\label{sec: geoforA3}
We end this section by applying \Cref{cor: RecurrenceImpliesLESS} to prove that, for a fixed $a\in \mathcal{F}$, the function $|\mathcal{I}_a(\lambda)|$ is a polynomial in $\lambda$ for the affine Weyl group of type $A_3$.
These polynomials are explicitly given in \Cref{thm: pols A3} and are in fact geometric polynomials.
For the rest of this section we work in type $A_3$.

We begin by noticing that
\begin{equation}
\mathcal{F} = \{ \mathrm{id}, s_0, s_0s_1, s_0s_3, s_0s_1s_3,s_0s_1s_3s_2  \}. 
\end{equation}

For $a=\mathrm{id}$, we proved in \cite[Theorem B]{castillo2023size} that $|\mathcal{I}_a(\lambda)|$ is a linear combination of the volumes of the faces of the orbit polytope $\mathsf{P}(\lambda)$ \eqref{eq:P_definition}.
In this paper we use the usual Euclidean volume, for more details we refer to \cite[\S 2.3]{castillo2023size}.

\begin{definition}\label{def: V_J}
 Recall that $(W_\mathrm{f},S_\mathrm{f})$ is a Coxeter system.
For $J\subset S_\mathrm{f}$, let $W_J$ be the parabolic subgroup generated by $J$.
The faces of $P(\lambda)$ are given by $F_J(\lambda)=\text{Conv} (W_J \cdot \lambda)$.
We write $$V_J(\lambda):=\text{Vol}(F_J(\lambda)).$$
\end{definition}

With this definition \cite[Theorem B]{castillo2023size} establishes that there exist unique $\mu_J \in \mathbb{R}$ such that
\begin{equation}\label{GF}
|\mathcal{I}_\mathrm{id}(\lambda)|  = \sum_{J\subset S_\mathrm{f}} \mu_J V_J(\lambda),
\end{equation}
for all $\lambda\in\Lambda^+$.
We stress that the numbers $\mu_J$ \textbf{do not} depend on $\lambda$.
This is called the Geometric Formula in \cite{castillo2023size} and it holds for arbitrary affine Weyl groups. 
It is a fact \cite[Remark 4.6]{castillo2023size} that there is a polynomial $p_J\in\bbR[x,y,z]$ such that
\begin{equation*}
V_J(x\varpi_1+y\varpi_2+z\varpi_3)=p_J(x,y,z), \quad \mbox{for all } x,y,z\in\bbZ_{\geq0}.
\end{equation*}
Let us write $(x,y,z):=x\varpi_1+y\varpi_2+z\varpi_3$.
With this notation, \eqref{GF} becomes
\begin{equation}\label{GF1}
|\mathcal{I}_\mathrm{id}(x,y,z)|  = \sum_{J\subset S_\mathrm{f}} \mu_J V_J(x,y,z),
\end{equation}
and we see that $|\mathcal{I}_\mathrm{id}(x,y,z)|$ is a polynomial in $\bbR[x,y,z]$.
The following theorem, a generalization of \eqref{GF1} to all elements of $\mathcal{F}$, establishes the polynomiality for the cardinality of all lower Bruhat intervals for dominant elements.
It precisely answers the question of counting the sets $\{u\in W_\mathrm{aff}\mid u\leq w\}$, for all $w\in W_\mathrm{aff}^+$.

We use the identification $S_\mathrm{f}=\{s_1,s_2,s_3\}$ with the set $\{1,2,3\}$, and we enumerate $\mathcal{F}$ as
\begin{equation}
a_1=\mathrm{id}, \quad a_2=s_0, \quad a_3= s_0s_1, \quad a_4=s_0s_3 , \quad a_5= s_0s_1s_3, \quad a_6= s_0s_1s_3s_2.
\end{equation}

\begin{theorem}\label{thm: pols A3}
Fix any $a\in \mathcal{F}$.
Then for all $J\subset S_\mathrm{f}$, there exist unique coefficients $\mu_{J,a}\in \mathbb{R}$ such that
\begin{equation}
|\mathcal{I}_a(x,y,z)|  = \sum_{J\subset S_\mathrm{f}} \mu_{J,a} V_J(x,y,z).
\end{equation}
Explicitly, the coefficients $\mu_{J,a}$ are given in \Cref{tab: coeff G}, and the polynomials $V_J(x,y,z)$ are given in \Cref{eq: volumenes A3}:
\begin{table}[h]
\centering
\begin{tabular}{|c||c|c|c|c|c|c|c|c|}  \hline
& $ \mu_{\emptyset,a} $  & $ \mu_{\{1\},a} $ & $ \mu_{\{2\},a} $ &$ \mu_{\{3\},a} $ & $ \mu_{\{1,2\},a} $ & $ \mu_{\{1,3\},a} $ &$ \mu_{\{2,3\},a} $ &  $\mu_{\{1,2,3\},a}$  \\[0.5ex]\hline
\rule{0pt}{2.3ex} $a_1$  & $24$ & $22 \, \sqrt{2}$ & $28 \, \sqrt{2}$ & $22 \, \sqrt{2}$ & $16 \, \sqrt{3}$ & $36$ & $16 \, \sqrt{3}$ & $12$   \\\hline
\rule{0pt}{2.3ex} $a_2$    & $48$ & $40 \, \sqrt{2}$ & $52 \, \sqrt{2}$ & $40 \, \sqrt{2}$ & $24 \, \sqrt{3}$ & $60$ & $24 \, \sqrt{3}$ & $12$   \\  \hline
\rule{0pt}{2.3ex} $a_3$   & $72$ & $52 \, \sqrt{2}$ & $70 \, \sqrt{2}$ & $58 \, \sqrt{2}$ & $24 \, \sqrt{3}$ & $72$ & $32 \, \sqrt{3}$ & $12$   \\ \hline
\rule{0pt}{2.3ex} $a_4$ & $72$ & $58 \, \sqrt{2}$ & $70 \, \sqrt{2}$ & $52 \, \sqrt{2}$ & $32 \, \sqrt{3}$ & $72$ & $24 \, \sqrt{3}$ & $12$   \\ \hline
\rule{0pt}{2.3ex} $a_5$ & $120$ & $82 \, \sqrt{2}$ & $100 \, \sqrt{2}$ & $82 \, \sqrt{2}$ & $32 \, \sqrt{3}$ & $96$ & $32 \, \sqrt{3}$ & $12$  \\ \hline
\rule{0pt}{2.3ex} $a_6$ & $192$ & $112 \, \sqrt{2}$ & $148 \, \sqrt{2}$ & $112 \, \sqrt{2}$ & $40 \, \sqrt{3}$ & $108$ & $40 \, \sqrt{3}$ & $12$  \\  \hline       
\end{tabular}
\caption{}
\label{tab: coeff G}
\end{table}

\begin{equation}\label{eq: volumenes A3}
\begin{array}{l}
V_{\emptyset}(x,y,z)=1, \hspace{6em}
V_{\{1\}}(x,y,z)=\sqrt{2} x, \hspace{6em}
V_{\{2\}}(x,y,z)=\sqrt{2} y,    \\[1em] 
V_{\{3\}}(x,y,z)=\sqrt{2} z, \hspace{10em}
V_{\{1,2\}}(x,y,z)=\tfrac{\sqrt{3}}{2}x^2 + 2\sqrt{3}xy + \tfrac{\sqrt{3}}{2}y^2,\\[1em] 
V_{\{1,3\}}(x,y,z)=2xz, \hspace{9.7em}
V_{\{2,3\}}(x,y,z)=\tfrac{\sqrt{3}}{2}y^2 + 2\sqrt{3}yz + \tfrac{\sqrt{3}}{2}z^2, \\[1em] 
V_{\{1,2,3\}}(x,y,z)=\tfrac{1}{3}x^3 + 2x^2y + 4xy^2 + \tfrac{4}{3}y^3 + 3x^2z + 12xyz + 4y^2z + 3xz^2 + 2yz^2 + \tfrac{1}{3}z^3.
\end{array}
\end{equation}
\end{theorem}

\begin{proof}
For a fixed element $a\in \mathcal{F}$, we define the function $F_a: \Lambda^+\rightarrow \mathbb{R} $ by  
\begin{equation}
F_a(\lambda)= \sum_{J\subset \Phi^+(\lambda)} \mu_{J,a} V_J(x,y,z).
\end{equation}
Our goal is to verify that $F_a$ satisfies \eqref{eq: coincide recurrence} for all $\lambda\in \Lambda^+$.

Since the set of zones $\mathcal{Z}$ (\Cref{def: zones}) partitions $\Lambda^+$, it suffices to verify \eqref{eq: coincide recurrence} within each individual zone $Z\in\mathcal{Z}$. By \Cref{thm: paper boat} and \Cref{deltaphi}, the coefficients $c_a(\lambda)$ and the sets $\Phi^+(\lambda)$ remain constant for all $\mu$ in a given zone $Z$. Consequently, for each zone $Z(i_1,i_2,i_3)\in\mathcal{Z}$, we must confirm the validity of the condition
\begin{equation}\label{eq: conditions zones}
c_a(i_1,i_2,i_3) = \sum_{J \subset \Phi^+(i_1,i_2,i_3)} (-1)^{|J|} F_a((x,y,z) - \vee_J),
\end{equation}
for all $\lambda \in Z(i_1,i_2,i_3)$, where $\lambda = (x,y,z) = x\varpi_1+y\varpi_2+z\varpi_3$.

We perform these verifications using SageMath together with Tables \ref{tab:Phiplusin A3} and \ref{tab:PBin A3}.
Below, we provide illustrative computations for three zones:

\begin{enumerate}[(a)]

\item \textbf{Zone $\mathbf{Z(2,2,2)}$:}
\begin{itemize}
\item Using \Cref{tab:Phiplusin A3}, we find that $\Phi^+(2,2,2)=\{ \alpha_1, \alpha_2,\alpha_3\}$.
\item From \Cref{tab:PBin A3}, we obtain $c_a(2,2,2)=24\cdot c_a^+(2,2,2)=24^2=576$, which does not depend on the choice of $a\in\mathcal{F}$ (as predicted by \Cref{prop:BarquitoGenerico}).
\item Substituting into \eqref{eq: conditions zones}, we get
\end{itemize}
\begin{small}
\begin{equation}\label{eq: condition 222}
\begin{array}{ll}
576 = & F(x,y,z) \\
	& -F_a(x-2,y+1,z)-F_a(x+1,y-2,z+1)-F_a(x,y+1,z-2) \\
	&  + F_a(x-1,y-1,z+1)+F_a(x-2,y+2,z-2)+F_a(x+1,y-1,z-1)\\
	&  -F_a(x-1,y,z-1).
\end{array}
\end{equation}
\end{small}

Each term in the right-hand side of \eqref{eq: condition 222} is obtained computing the join of a subset of $\Phi^+(2,2,2)$. 
For instance, if  $J=\{ \alpha_1, \alpha_2\}$ we have $\vee_J=\alpha_1 + \alpha_2=(1,1,-1)$, leading to
\begin{equation}\label{eq: example of term}
(-1)^{|J|} F_a((x,y,z) - \vee_J) = F_a(x-1,y-1,z+1).
\end{equation}
The other terms are computed in a similar fashion.
We have that \eqref{eq: condition 222} simplifies to
\begin{equation*}
576 = 48\mu_{\{1,2,3\},a},
\end{equation*}
and consulting \Cref{tab: coeff G}, the equality holds for all $a\in \mathcal{F}$.

\item \textbf{Zone $\mathbf{Z(1,1,2)}$:}
\begin{itemize}
\item Here, $\Phi^+(1,1,2)= \{\alpha_3, \alpha_{12}\}$.
\item For $a=a_5=s_0s_1s_3$, we compute $c_{a_5}(1,1,2) = 24 \cdot c_{a_5}^+(1,1,2) = 24 \cdot 34 = 816$.
\item Substituting into \eqref{eq: conditions zones}, we obtain the relation
\end{itemize}
\begin{equation}
816 = F_{a_5}(1,1,z)-F_{a_5}(1,2,z-2)- F_{a_5}(0,0,z+1) +F_{a_5}(0,1,z-1),
\end{equation}
which simplifies to
\begin{equation}
816 = -3\sqrt{3}\mu_{\{1,2\},a_5} + 4\mu_{\{1,3\},a_5} + 3\sqrt{3}\mu_{\{2,3\},a_5} + 36\mu_{\{1,2,3\},a_5}.
\end{equation}
Using \Cref{tab: coeff G}, we confirm this identity.

\item \textbf{Zone $\mathbf{Z(1,1,1)}$:}
\begin{itemize}
\item We find $\Phi^+(1,1,1)= \{\alpha_{12}, \alpha_{23}\}$.
We note that $J=\Phi^+(1,1,1)$ gives $\vee_J=\alpha_{13}=(1,0,1)$, which is different from $\alpha_{12}+\alpha_{23}$.
\item With $a=a_2=s_0$, the condition to verify is
\end{itemize}
\begin{equation} \label{eq: condition 111}
720 = F_{a_2}(1,1,1)-F_{a_2}(0,0,2)- F_{a_2}(2,0,0) +F_{a_2}(0,1,0).
\end{equation}
After simplifying, we get the equivalent condition
\begin{align}
\begin{split}
720 =& - \sqrt{2}\mu_{\{1\},a_2} + \tfrac{3}{2}\sqrt{3}\mu_{\{1,2\},a_2} + 2\sqrt{2}\mu_{\{2\},a_2} + \tfrac{3}{2}\sqrt{3}\mu_{\{2,3\},a_2}\\
&- \sqrt{2}\mu_{\{3\},a_2} + 28\mu_{\{1,2,3\},a_2} + 2\mu_{\{1,3\},a_5}.  
\end{split}
\end{align}
Substituting values from \Cref{tab: coeff G}, the equality is verified.
\end{enumerate}

By repeating similar calculations for each zone $Z\in\mathcal{Z}$, we confirm that $F_a$ satisfies \eqref{eq: conditions zones} for all $\lambda \in \Lambda^+$. Thus, by \Cref{cor: RecurrenceImpliesLESS}, we conclude that $F_a(\lambda) = |\mathcal{I}_a(\lambda)|$ for all $\lambda \in \Lambda^+$.
\end{proof}

\begin{remark}\label{rmk: permutohedron approx}
We stress that, for $a\in\mathcal{F}$, \Cref{thm: pols A3} implies that the cardinality of $\mathcal{I}_a(x,y,z)$ is a polynomial in $\bbR[x,y,z]$ of degree $3$.
Furthermore, \Cref{tab: coeff G} shows that that the homogeneous part of degree $3$ of this polynomial, is
\begin{equation*}
\mu_{\{1,2,3\},a}\cdot V_{\{1,2,3\}}(x,y,z)=12\cdot V_{\{1,2,3\}}(x,y,z),
\end{equation*}
which does not depend on $a$.
Heuristically, we view this phenomenon as the permutohedron $\mathsf{P}(\lambda)$ serving as an approximation of the lower intervals $\mathcal{I}_a(\lambda)$.
In \Cref{thm:lattice_count} we will prove this for type $\widetilde{A}_n$ in general.
\end{remark}

After extensive computations in types $\widetilde{A_3}$ and $\widetilde{B_3}$, we conjecture a generalization of  \Cref{thm: pols A3} to all affine Weyl groups and to the whole lowest two-sided cell $\mathfrak{C}$.

\begin{conjecture}\label{conj: pols geom}
Let $\Phi$ be any irreducible root system, and recall the parametrization
\begin{equation}
\Theta:\mathcal{F}\times(\Lambda^\vee)^+\times\mathcal{F}\times\Omega\xlongrightarrow{\sim}\mathfrak{C}.
\end{equation}
Fix $a,b\in\mathcal{F}$ and $\sigma\in\Omega$.
Then there exist unique coefficients
$\mu_{J,a,b,\sigma}\in \mathbb{R}$ such that
\begin{equation}\label{GFAmigos}
|\mathcal{I}(a,\lambda,b,\sigma)|  = \sum_{J\subset S_\mathrm{f}} \mu_{J,a,b,\sigma} V_J(\lambda),
\end{equation}
for all $\lambda\in(\Lambda^\vee)^+$.
\end{conjecture}
By \Cref{rmk: forget sigma} the choice of $\sigma$ is superfluous, i.e., if the conjecture holds for $\sigma=\mathrm{id}$ then it holds for all $\sigma$.
In this case we would have $\mu_{J,a,b,\sigma}=\mu_{J,a,b,\mathrm{id}}$.

\section{Lattice points in polyhedra}\label{sec: latticepoints}
In this section we work entirely in type $\widetilde{A}_n$.
As before we use the standard label of the simple roots $\Delta =\{ \alpha_1, \dots, \alpha_n \}$ with fundamental weigths
$\{ \varpi_1, \dots, \varpi_n\} $.
For notation we denote the set $\{1,\dots,k\} = [k]$ for any natural number $k$.

We begin by rewriting the formula from \Cref{cor:weigthed_sum} in terms of discrete geometry.
Recall the definition of $\mathsf{P}(\lambda)$ in \eqref{eq:P_definition}.
We had previously defined certain discrete zones in \Cref{def: zones}, and now we define a continuous version. For $\boldsymbol{i}\in \{0,1,2\}^n$ we define

\begin{equation}\label{eq: continuous zones}
Z_\mathbb{R}(\boldsymbol{i}) =  \{  x \in C^+ \mid \langle x , \alpha_j \rangle = i_j \mbox{ if } i_j=0, 1, \mbox{ and } \langle x, \alpha_j \rangle \geq 2 \mbox{ if }\ i_j=2 \}.
\end{equation}

\begin{definition}
Let $ \lambda \in C^{+} $ be an element in the dominant chamber.
We define the following polyhedra related to $  \mathsf{P}(\lambda)  $.
\begin{align}
\mathsf{P}^+(\lambda) &= \mathsf{P}(\lambda)\cap C^+, \label{eq:P+} \\
\mathsf{P}_{\boldsymbol{i}}^+(\lambda) &= \mathsf{P}^{+}(\lambda) \cap  Z_\mathbb{R}(\boldsymbol{i}), \, \forall \boldsymbol{i} \in \left\{ 0,1,2 \right\}^{n}. \label{eq:polyhedral_zone}
\end{align}
\end{definition}

With these definitions we can combine \eqref{eq: ideal as latticepoints} and \eqref{eq:weighted_sum} to obtain
\begin{equation}\label{eq:weighted_lattice_count}
|\li{\lambda}| = \sum_{\boldsymbol{i} \in \left\{ 0,1,2 \right\}^{n}} c_a(\boldsymbol{i}) |\mathsf{P}_{\boldsymbol{i}}^+(\lambda) \cap \left( \mathbb{Z}\Phi + \lambda \right) |,
\end{equation}
for $a\in \mathcal{F}$ and $\lambda \in \Lambda^+$ a dominant weight.

\Cref{eq:weighted_lattice_count} is the key connection between the Bruhat order and polyhedral combinatorics.
Counting lattice points in polyhedra is a classic problem and it is well known that such problem exhibits a polynomial behavior \cite{barvinok2008integer}.
There are however some technical details that we must address to properly use the available results.

\subsection{Vertices of $\mathsf{P}^{+}(\lambda)$}

We begin by describing how the vertices of $ \mathsf{P}^{+}(\lambda)  $ depends on $ \lambda $.

\begin{prop}\label{prop:vertices_lambda}

Let $ \lambda \in C^{+}$ be any element in the dominant region.
The point $ \mathbf{0} $ is a vertex of $ \mathsf{P}^{+}(\lambda)  $.
For every other vertex $ \sum_{i=1}^n a_i \varpi_{i} $ we have that each $ a_i $ is a $\mathbb{Z} \left[ \frac{1}{(n+1)!} \right]$-linear combination of $\lambda_1,\dots, \lambda_n$, where $ \lambda = \sum_{i=1}^{n} \lambda_{i} \varpi_{i} $.
\end{prop}

\begin{proof}
The inequality description of $\mathsf{P}^+(\lambda)$ has the following form
\begin{equation}\label{eq: P plus ineq}
\left\{ x \in E \, \middle \vert 
\begin{array}{rll}
\langle \varpi_i , x \rangle &\leq \langle \varpi_i, \lambda \rangle, & i \in [n],  \\
\langle \alpha_{i} , x \rangle &\geq 0, & i \in [n],\\
\end{array}
\right\}.
\end{equation}
Every vertex is the unique solution to a linear subsystem of
\begin{equation}\label{eq:subsystem}
\begin{array}{rll}
\langle \varpi_i , x \rangle &= \langle \varpi_i, \lambda \rangle, & i \in [n],  \\
\langle \alpha_{i} , x \rangle &= 0, & i \in [n],\\
\end{array}
\end{equation}
consisting of of $n$ equations such that the set of normal\footnote{For a linear equality of the form $ \langle \mathbf{a}, x \rangle = b $, we call the vector $ \mathbf{a} $ the normal vector.} vectors are linearly independent, see \cite[Lemma 4.9]{barvinok2008integer}.
For example, the point $\mathbf{0}$ is the unique solution to the subsystem consisting of the simple roots.
Since $\mathbf{0} \in \mathsf{P}^+(\lambda)$, then it is a vertex. 
In general, writing $ x = \sum_i x_i \varpi_i $ we transform the system \eqref{eq:subsystem} into a matrix equation of the form $ \mathbf{C}x = b(\lambda) $, where $b(\lambda)$ is a column vector such that each entry is a $ \mathbb{Z}[\frac{1}{n+1}] $-linear in $ \lambda_{1}, \dots, \lambda_{n} $, since $(n+1)\langle \varpi_i, \varpi_j \rangle \in \mathbb{Z}$ for every pair $i,j$.
The matrix $ \mathbf{C} $ has to be invertible and by Cramer's rule we have that the entries of the unique solution are a $ \mathbb{Z}[\frac1D] $-linear combination of the entries of $ b(\lambda) $, where $ D := \det(\mathbf{C}) $ .
By \Cref{lem:cool} the factor $\frac{1}{n+1}$ cancels and we get that each coordinate is indeed a $\zfrac$-linear combination of $\lambda_1, \dots, \lambda_n$.

\end{proof}

\begin{lemma}\label{lem:cool}
Consider the $ 2n \times n $ matrix $ \mathbf{M}_{n} $ obtained by concatenating vertically the inverse of the Cartan matrix and the 
and the identity matrix $ \mathbf{I}_{n} $.
Then every maximal minor of $ \mathbf{M}_{n} $ is either $0$ or $k/(n+1)$ where $k$ divides $ (n+1)! $.
\end{lemma}
\begin{proof}
The set of maximal minors of $\mathbf{M}_{n}$ coincides with the set of all minors of the inverse of the Cartan matrix $\mathbf{A}_n^{-1}$.
The Jacobi's complementary minor formula \cite[Lemma A.1]{caracciolo2013algebraic} states that
\begin{equation}\label{eq:jacobi}
 \det \mathbf{N}[I,J] = |(\det(\mathbf{N}))(\det \mathbf{N}^{-1}[I^c, J^c])|,
\end{equation}
where $\mathbf{N}[I,J]$ is the submatrix of $\mathbf{N}$ obtained by restricting to rows indexed by $I$ and columns indexed by $J$.
Applying \eqref{eq:jacobi} to the inverse Cartan matrix we obtain that the set of all (absolute values of) minors are the same as those of the Cartan matrix multiplied by $1/(n+1)$, since it is well-known that $\det(\mathbf{A}_n)=n+1$.
The conclusion follows from \Cref{prop:minors_cartan}.
\end{proof}

\begin{prop}\label{prop:minors_cartan}
Let $\mathbf{A}_n$ be the Cartan  matrix of type $A_n$ and
let $\mathbf{B}$ be a submatrix of $\mathbf{A}_n$ obtained by eliminating some rows and the same number of columns. 
Then, either $\det(\mathbf{B})=0$ or there exists a partition $\mu=(\mu_1,\ldots ,\mu_\ell)$ of $0\leq m\leq n$ with $\ell \leq n+1-m$ parts such that 
\begin{equation} \label{eq:det minor}
	 |\det (\mathbf{B})| =  (\mu_1+1)(\mu_2+1) \cdots (\mu_\ell+1).
\end{equation}
In particular, it divides $(n+1)!$.
\end{prop}

\begin{proof}
Let $k$ be the size of the square matrix $\mathbf{B}$. 
We proceed by induction on $k$. 
If $k=1$ then $|\det (\mathbf{B})|=0,1,$ or $2$.  
If $|\det (\mathbf{B})|=1$ we take $\mu $  as the empty partition of $0$. 
Similarly, if $|\det (\mathbf{B})|=2$ we take $\mu =(1)$. 
This proves the base case. 

We now fix $2\leq k\leq n $ and suppose that \eqref{eq:det minor} holds for matrices of size less than $k$. 
We first suppose that the set of rows and columns eliminated from $\mathbf{A}_n$ to obtain $\mathbf{B}$ is the same. 
It follows that $\mathbf{B}$ is a diagonal block matrix, where each block is a Cartan matrix of type $\mathbf{A}_j$ for some  $j\leq k$. 
Let $\mu = (\mu_1, \ldots , \mu_\ell) $ be the sizes of each block ordered in decreasing order (here we admit repetitions). 
Since the determinant of a Cartan matrix of type $\mathbf{A}_j$ is $j+1$ we have that $\mu$ satisfies \eqref{eq:det minor}.
This proves the proposition in this case. 

Now let $i_1< i_2 <\ldots < i_k$ and $j_1< j_2<\ldots <j_k$ be the index of the rows and columns that were not eliminated in obtaining $\mathbf{B}$ from $\mathbf{A}_n$. 
By the previous paragraph we can assume that the tuples $(i_1,\dots,i_k)$ and $(j_1,\dots, j_k)$ are not identical.
It follows that we can define
\begin{equation}
u= \min \{  1\leq b \leq k \mid   i_b\neq j_b  \}. 
\end{equation}
Let $\mathbf{B}=(b_{ij})$.
We split $\mathbf{B}$ in blocks as follows 
\begin{equation}\label{eq:B_Blocks}
\mathbf{B}= \left[
\begin{array}{c|c|c}
X_{11}  &  X_{12} & X_{13}  \\ \hline
X_{21} &  b_{uu}  &  X_{23}  \\ \hline
X_{31} &  X_{32}  &  X_{33}  \\ 
\end{array}\right],
\end{equation}
where  $X_{11}$ and $X_{33}$ are square matrices of size $u-1$ and $k-u$, respectively. 
We notice that $b_{uu}=-1$ or $b_{uu}=0$. 

We first treat the case when $b_{uu}=0$. 
Let us assume that this zero was located above the diagonal in $\mathbf{A}_n$. 
In this case, we have that $X_{12}$, $X_{13}$ and $X_{23}$ are the zero matrix.
Therefore, $\det (B)=0$.
The case when the zero is located under the diagonal in $\mathbf{A}_n$ is analogous.

We now suppose that $b_{uu}=-1$. 
Let us assume that this $-1$ was located above the diagonal in $\mathbf{A}_n$. 
The case when the $-1 $ is located below the diagonal is entirely analogous and is left to the reader. 
As before, we have  that $X_{12}$, $X_{13}$ and $X_{23}$ are the zero matrix.
It follows that 
\begin{equation}
	\label{eq:det_of_B}
|\det (\mathbf{B})| = |\det (X_{11})| \cdot  |\det (X_{33})|.
\end{equation}

Note that splitting $\mathbf{A}_n$ as in \eqref{eq:B_Blocks}, we see that $X_{11}$ (resp. $X_{33}$) is a submatrix of $\mathbf{A}_{u-1}$ (resp. $\mathbf{A}_{n-u}$).
We can apply our inductive hypothesis to $X_{11} $ and $X_{33}$ to obtain partitions $\nu =  (\nu_1, \ldots , \nu_r)$ and $ \eta = (\eta_1, \ldots , \eta_s) $ such that:
\begin{enumerate}
\item  $\nu$ is a partition of $ 0\leq c \leq  u-1$;
\item  $\eta  $ is a partition of $0\leq d \leq n-u$;
\item $ r \leq u-c$   and $s\leq n-u+1-d$;
\item $  |\det (X_{11})| =  (\nu_1+1)(\nu_2+1) \cdots (\nu_r+1)  $ and $  |\det (X_{33})| =  (\eta_1+1)(\eta_2+1) \cdots (\eta_s+1)  $.
\end{enumerate}
Let $\mu$ be the partition of $m\coloneqq c+d$ that is obtained by reordering in decreasing order the composition $(\nu_1, \ldots , \nu_r,\eta_1, \ldots , \eta_s)$. 
We notice that $m\leq n-1 <n $.
On the other hand,  $\mu$ has $\ell\coloneqq r+s $ parts and $\ell \leq n+1-m$.
Thus by \eqref{eq:det_of_B}, $\mu$ satisfies \eqref{eq:det minor}. 
This proves the first claim in the proposition.

To see the second claim, we first notice that 
\begin{equation*}
(\mu_1+1)+(\mu_2+1)+\cdots+(\mu_\ell+1) = m+\ell \leq n+1.
\end{equation*}
By the integrality of the multinomial coefficients we have that $(\mu_1+1)!(\mu_2+1)! \cdots (\mu_\ell+1)!$ divides $(m+\ell)!$, which in turns divides $(n+1)!$.
Since clearly $(\mu_1+1)(\mu_2+1) \cdots (\mu_\ell+1)$ divides $(\mu_1+1)!(\mu_2+1)! \cdots (\mu_\ell+1)!$, the proposition follows.
\end{proof}

\begin{cor}\label{cor:vertices_lattice}
Let $ \lambda \in \Lambda^{+}$ be a dominant weight. 
Vertices of $ \mathsf{P}^{+}(\lambda) $ lie on $ \zfrac \Phi $.
\end{cor}

\begin{proof}
By \Cref{prop:vertices_lambda}, the vertices lie in $ \zfrac\Lambda $ if $\lambda$ is a weight.
Since $ \Lambda \subset \mathbb{Z}[\frac{1}{n+1}] \Phi $, we obtain the desired conclusion.
\end{proof}

% We can describe the polytope $  \mathsf{P}^{+}(\lambda)  $ from Equation \eqref{eq:P+} as an intersection of two simplicial cones $\mathsf{P}^{+}(\lambda) =  \left( \lambda - \mathbb{R}_{\geq 0}\Delta \right) \cap C^{+}.$

The following Lemma is a crucial ingredient for proving quasi-polynomiality.

\begin{lemma}\label{prop:Minkowski_sum}
Let $z, y \in C^+$ and $r\in \mathbb{R}_{\geq 0}$.
We have $\mathsf{P}^+(z)+\mathsf{P}^+(y)=\mathsf{P}^+(z + y)$, where the left-hand side denotes the Minkowski sum, and $\mathsf{P}^+(rz)=r\mathsf{P}^+(z).$
\end{lemma}

\begin{proof}
The first equality is shown in \cite[Proposition 2.13]{burrull2024strongly}, while the second one is clear.
\end{proof}

The polytopes $ \mathsf{P}_{\boldsymbol{i}}^+(\lambda) $ appearing in Equation \eqref{eq:polyhedral_zone} are defined as slices of $ \mathsf{P}^+(\lambda) $.
It will be more convenient to consider $\mathsf{P}_{\boldsymbol{i}}^+$ as faces of certain translations of $\mathsf{P}^+$.
For $\boldsymbol{i}= ( i_{1}, \dots, i_{n} )\in \{0,1,2\}^n$, we define
\begin{equation*}
\varpi(\boldsymbol{i})\coloneqq i_1\varpi_1+\cdots+i_n\varpi_n, \quad \text{and} \quad V_{\boldsymbol{i}} \coloneqq \left\{ x \in E \mid \langle x, \alpha_{j} \rangle = 0, \forall\,  j \in [n] :i_j\neq 2 \right\}.
\end{equation*}

\begin{lemma}\label{prop:vertices}
Let $\lambda\in C^+$ and let $\boldsymbol{i} \in \{0,1,2\}^n$.
Then
\begin{enumerate}[(i)]
    \item\label{item: vertices 1} $\mathsf{P}^{+}(\lambda)\cap V_{\boldsymbol{i}} $ is a face of $\mathsf{P}^{+}(\lambda)$.

    \item\label{item: vertices 2} Suppose $\lambda\in Z_\bbR(\boldsymbol{i})$.
    Then
    \begin{equation}\label{eq: translation trick}
    \mathsf{P}_{\boldsymbol{i}}^{+}(\lambda) - \varpi(\boldsymbol{i}) = \mathsf{P}^{+}(\lambda - \varpi(\boldsymbol{i})) \cap V_{\boldsymbol{i}}.
    \end{equation}
\end{enumerate}
As a consequence, for $\lambda\in Z_\bbR(\boldsymbol{i})$, the polytope $   \mathsf{P}_{\boldsymbol{i}}^{+}(\lambda)  $ has vertices in the lattice $\mathbb{Z}[\frac{1}{(n+1)!}]\Phi$.
\end{lemma}

\begin{proof}\hfill
\begin{enumerate}[(i)]
\item This follows from the inequality description in Equation \eqref{eq: P plus ineq}.

\item As $\lambda\in Z_\bbR(\boldsymbol{i})$, we have that $\lambda - \varpi(\boldsymbol{i})\in C^+$.
We start by noticing the following:
\begin{equation*}
C^+ \cap V_{\boldsymbol{i}} = Z_\bbR(\boldsymbol{i}) - \varpi(\boldsymbol{i})
\quad \text{and} \quad
\big(\mathsf{P}^+(\lambda)-\varpi(\boldsymbol{i})\big)\cap C^+ = \mathsf{P}^+(\lambda-\varpi(\boldsymbol{i})).
\end{equation*}
Both equalities follow easily by definition (for the second one, use the inequality description \eqref{eq: P plus ineq} of $\mathsf{P}^+$).
We get
\begin{align*}
\mathsf{P}_{\boldsymbol{i}}^{+}(\lambda) - \varpi(\boldsymbol{i}) &=
\big(\mathsf{P}^{+}(\lambda)\cap  Z_\bbR(\boldsymbol{i})\big) - \varpi(\boldsymbol{i})\\
&= \big(\mathsf{P}^{+}(\lambda) -\varpi(\boldsymbol{i})\big) \cap C^+\cap V_{\boldsymbol{i}}\\
&= \mathsf{P}^{+}(\lambda - \varpi(\boldsymbol{i})) \cap V_{\boldsymbol{i}},
\end{align*}
as desired.
% We prove \cref{eq: translation trick} by double inclusion, using the inequality description \eqref{eq: P plus ineq} of $\mathsf{P}^+$.
% Fix any $j\in[n]$.
% \begin{itemize}
%     \item[$\subseteq$:] Let $x\in \mathsf{P}_{\boldsymbol{i}}^{+}(\lambda)$, and write $y\coloneqq x - \varpi(\boldsymbol{i})$.
%     Since $x\in \mathsf{P}^{+}(\lambda)$, it is clear that $\langle y, \varpi_j\rangle \leq \langle \lambda-\varpi(\boldsymbol{i}), \varpi_j\rangle$.
%     On the other hand since $x\in Z_\bbR(\boldsymbol{i})$, we have that $\langle x, \alpha_j\rangle \geq i_j = \langle \varpi(\boldsymbol{i}), \alpha_j\rangle$, from which $\langle y, \alpha_j\rangle\geq 0$.
%     Therefore $y\in \mathsf{P}^{+}(\lambda - \varpi(\boldsymbol{i})).$
%     Now if $i_j\neq 2$ we have the equality $\langle x, \alpha_j\rangle = i_j$, so that $\langle y, \alpha_j\rangle=0$.
%     Thus $y\in V_{\boldsymbol{i}}$ as desired.

%     \item[$\supseteq$:] Let $y\in \mathsf{P}^{+}(\lambda - \varpi(\boldsymbol{i})) \cap V_{\boldsymbol{i}}$ and write $x\coloneqq y+\varpi(\boldsymbol{i})$.
%     By the above paragraph, it is clear that $x\in \mathsf{P}^{+}(\lambda)$ and that $\langle x,\alpha_j\rangle \geq i_j$.
%     If $i_j\neq 2$ then $\langle x,\alpha_j\rangle = 0$, so that $\langle x,\alpha_j\rangle = 0$
% \end{itemize}
\end{enumerate}

Now let $\lambda\in Z_\bbR(\boldsymbol{i})$.
Combining items \ref{item: vertices 1} and \ref{item: vertices 2}, we see that $\mathsf{P}_{\boldsymbol{i}}^{+}(\lambda) - \varpi(\boldsymbol{i})$ is a face of $\mathsf{P}^{+}(\lambda - \varpi(\boldsymbol{i}))$.
It must have vertices in $\mathbb{Z}[\frac{1}{(n+1)!}]\Phi$ by \Cref{cor:vertices_lattice}.
The last statement follows after translating by $\varpi(\boldsymbol{i}) \in \Lambda \subset \bbZ[\frac{1}{n+1}]\Phi.$
\end{proof}

\begin{example}\label{ex:explanation}
Let us take a closer look into this translation trick.
Let $ n=2 $.
Consider the element $ \lambda = 6\varpi_1 + 4\varpi_2$ in $ C^+ $ (depicted as a red dot in \Cref{fig:rational}).
The polyhedral region $\mathsf{P}^{+}_{(2,1)}(6\varpi_1 + 4\varpi_2) $ is equal to the face $ \mathsf{F} $ of the polytope $\mathsf{P}^{+}(4\varpi_1+3\varpi_2)$ defined by $ \langle \alpha_{2}, x \rangle = 0 $, translated by $\varpi(2,1)=2\varpi_1 + \varpi_2$.
See \Cref{fig:rational}.
% This explains why we defined $ \mathsf{P}^{+}(z) $ for elements $ z $ that are not necessarily in $ C^+ $.
\end{example}

\begin{figure}[h]
	\centering
\begin{tikzpicture}[scale=0.5,x={(60:1)}, y={(120:2)}]

\begin{scope}[xshift=-6cm]
	\foreach \i in {0,1,...,11} {
  \foreach \j in {0,1,...,11} {
    \ifnum\numexpr\i+\j<11
      \fill (\i,\j) circle (2pt);
    \fi
  }
}

\fill[red] (6,4) circle (2.5pt);
\fill[black, fill opacity = 0.1] (0,7) -- (6,4) -- (8,0) -- (0,0) -- cycle;
\draw[thick, blue] (2,1) -- (7.5,1) ;

\draw[ultra thick, red, ->, >=stealth] (0,0) -- (1,0);
\draw[ultra thick, red, ->, >=stealth] (0,0) -- (0,1);
\node[right] at (1,0) {$\varpi_1$};
\node[left] at (0,1) {$\varpi_2$};
\end{scope}

\begin{scope}[xshift=6cm]

	\foreach \i in {0,1,...,11} {
  \foreach \j in {0,1,...,11} {
    \ifnum\numexpr\i+\j<11
      \fill (\i,\j) circle (2pt);
    \fi
  }
}

\fill[red] (4,3) circle (2.5pt);
\fill[black, fill opacity = 0.1] (0,5) -- (4,3) -- (5.5,0) -- (0,0) -- cycle;
\draw[thick, blue] (0,0) -- (5.5,0) ;
\draw[ultra thick, red, ->, >=stealth] (0,0) -- (1,0);
\draw[ultra thick, red, ->, >=stealth] (0,0) -- (0,1);
\node[right] at (1,0) {$\varpi_1$};
\node[left] at (0,1) {$\varpi_2$};
\end{scope}
\end{tikzpicture}
\caption{The translation trick.
The polytopal zone $\mathsf{P}^+_{(2,1)}(6\varpi_1 + 4\varpi_2)$ is a translation of a face of the polytope $\mathsf{P}^+(4\varpi_1+3\varpi_2)$.}
\label{fig:rational}
\end{figure}

\subsection{Quasipolynomiality}

\begin{definition}\label{def:quasi_polynomial}

A function $ q: \mathbb{N}^n \to \mathbb{Z} $ is a \textbf{quasi-polynomial} of degree $d$ if there exists a lattice (a discrete subgroup) $ \Gamma $ of rank $ n $ such that $\Gamma \subset \mathbb{Z}^{n} $ and
cosets $ g_{1} + \mathbb{Z}^{n}, \dots, g_{M} + \mathbb{Z}^{n} $ together with polynomial functions $ p_{1}, \dots, p_{M} $ all of degree $d$, such that $ q(\gamma) = p_k(\gamma)$ whenever $ \gamma \in g_k  + \mathbb{Z}^{n}$.
\end{definition}

Quasi-polynomials appear in many different contexts, see \cite{woods2014unreasonable} for a survey.
If $ q $ is a quasi-polynomial then its generating function has the following rather simple description
\begin{equation}\label{eq:quasi-genfun}
\sum_{\mathbf{n}\in \mathbb{N}^{n}} q(\mathbf{n})X^{\mathbf{n}} = \frac{p(X)}{(1-X^{\mathbf{b}_{1}})\dots (1-X^{\mathbf{b}_{k}})},
\end{equation}
where $ p(X) $ is a polynomial in $ \mathbb{Q}[X ]$ and $ \mathbf{b}_{i} \in \mathbb{N}^{n}\setminus \left\{ \mathbf{0} \right\} $.

\begin{example}\label{ex:floor}
Let $\mathsf{P} = [0,1/2]$ be a 1-dimensional polytope.
The function $q: \mathbb{N} \to \mathbb{Z}$ defined as $q(n) = |nP \cap \mathbb{Z}|$ is the quasi-polynomial given by $\lfloor n/2 \rfloor + 1$.
This function has two different polynomial expressions in the cosets $2\mathbb{Z}$ and $1+2\mathbb{Z}$.
In this case the generating function is
$
\sum_{n\geq 0} (\lfloor n/2 \rfloor + 1)X^n = \frac{1+X}{(1-X^2)^2}.  
$
\end{example}

The main technical result we use is \cite[Theorem 7]{mcmullen1978lattice}, which we restate here in a slightly different form.

\begin{theorem}[McMullen]\label{thm:technical_tool}
Let $E $ be a real vector space of dimension $ n $, $ \Gamma $ a lattice, and $ \mathsf{P}_{1}, \dots, \mathsf{P}_{k} $ be polytopes in $ E $ with vertices in $ \Gamma \otimes_{\mathbb{Z}} \mathbb{Q} $.
Then we have that the function $ q: \mathbb{N}^{k} \to \mathbb{Z} $ defined as
\begin{equation}
q( \lambda_{1}, \dots, \lambda_{k}) = | \left( \lambda_{1}\mathsf{P}_{1} + \cdots + \lambda_{k}\mathsf{P}_{k} \right) \cap \Gamma |,
\end{equation}
is a quasi-polynomial of degree at most $ n $ over the subgroup $ \bigoplus_{i=1}^k N_{i}\mathbb{Z} \subseteq \mathbb{Z}^{k} $ where
\begin{equation}\label{eq:period}
N_i := \min \left\{ m_{i} \in \mathbb{N} ~:~ m_i \mathsf{P}_{i} \text{ has vertices in } \Gamma \right\}
\end{equation}
Furthermore, let $q^{\mathsf{top}}$ be the top homogeneous component of $q$.
Then $q^{\mathsf{top}}$ is a polynomial, and it satisfies the following relation
\begin{equation}\label{eq:volume_polynomial}
\relvol \left( \lambda_{1}\mathsf{P}_{1} + \cdots +\lambda_{k}\mathsf{P}_{k} \right)  = q^{\mathsf{top}}( \lambda_{1}, \dots, \lambda_{k}).
\end{equation}
\end{theorem}

The left-hand side of \cref{eq:volume_polynomial} refers to the \emph{relative volume} of a polytope with integral vertices but perhaps contained in a linear subspace $L$.
The relative volume is a measure on $L$ normalized so that the fundamental parallelepiped of the lattice $\Gamma \cap L$, in $L$, has relative volume equal to one.
A priori, $q^{\mathsf{top}}$ is merely a quasi-polynomial.
It is a classical result that the relative volume of a Minkowski sum is a polynomial, as in \cref{eq:volume_polynomial} (cf. \cite[Chapter 4]{ewald1996combinatorial}).
The last part on the homogeneous top component can be easily deduced from the integral version \cite[Theorem A.3]{postnikov2009permutohedra}.
This in particular implies that if $ \sum \lambda_{i} \mathsf{P}_{i} $ is not full dimensional, then the quasi-polynomial has degree strictly lower than $ n $.
The statement in \cite{mcmullen1978lattice} uses a different definition of $ N_{i} $, and we chose this one (which is a multiple of the original) for simplicity.

We now have all the key ingredients to prove the next proposition.

\begin{prop}\label{prop:quasi_polinomiality}
Let $ \boldsymbol{i} \in \left\{ 0,1,2 \right\}^{n} $ and write $\bbN_{\boldsymbol{i}}\coloneqq \prod_{j=1}^n \bbN_{\geq i_j}$.
The function $ q_{\boldsymbol{i}}: \bbN_{\boldsymbol{i}} \to \mathbb{Z} $ given by
\[
q_{\boldsymbol{i}}( \lambda_{1}, \dots, \lambda_{n}) = | \mathsf{P}^+_{\boldsymbol{i}}(\lambda) \cap  \left( \mathbb{Z}\Phi + \lambda \right) |, \quad \text{where} \quad \lambda = \sum_{i=1}^n \lambda_i \varpi_{i},
\]
is a quasi-polynomial in the coordinates $\lambda_1, \dots, \lambda_n$, that is, when both $\lambda$ and $\lambda-\varpi(\boldsymbol{i})$ are dominant weights.
Furthermore, $q^{\mathsf{top}}_{\boldsymbol{i}}( \lambda_{1}, \dots, \lambda_{n}) = \relvol \left( \mathsf{P}^+_{\boldsymbol{i}}(\lambda)\right)$.
\end{prop}

\begin{proof}
Let $\mu\coloneqq\lambda-\varpi(\boldsymbol{i})$.
We have $\mu=\sum_{i=1}^n\mu_i\varpi_i$, with $\mu_j=\lambda_j-i_j$.
Note that \Cref{prop:vertices} gives
\begin{equation}\label{eq:translation trick prop}
\mathsf{P}^+_{\boldsymbol{i}}(\lambda) - \lambda = 
\big(\mathsf{P}^+(\mu)\cap V_{\boldsymbol{i}}\big) - \mu.
\end{equation}
Therefore, 
\begin{equation}\label{eq: mu dominant}
\big(\mathsf{P}^+(\mu)\cap V_{\boldsymbol{i}}\big) - \mu =
\Big[\Big(\sum_{i=1}^n\mu_i\mathsf{P}^+(\varpi_i)\Big)\cap V_{\boldsymbol{i}}\Big] - \mu =
\sum_{i=1}^n\mu_i\big[\big( \mathsf{P}^+(\varpi_i)\cap V_{\boldsymbol{i}}\big) - \varpi_i\big],
\end{equation}
where the first equality follows from \Cref{prop:Minkowski_sum}, while the second one is easily derived from the definition of $V_{\boldsymbol{i}}$.
On the other hand, by \cref{eq:translation trick prop}, we get
\begin{equation}\label{eq:Li translation trick}
q_{\boldsymbol{i}}( \lambda_{1}, \dots, \lambda_{n}) =
\big| \left(\mathsf{P}^+_{\boldsymbol{i}}(\lambda) - \lambda \right)\cap  \mathbb{Z}\Phi \big|
= \big| \left[\big(\mathsf{P}^+(\mu)\cap V_{\boldsymbol{i}}\big) - \mu \right]\cap  \mathbb{Z}\Phi \big|.
\end{equation}
Since $\mathsf{P}^+(\varpi_i)\cap V_{\boldsymbol{i}}$ is a face of $\mathsf{P}^+(\varpi_i)$ (cf. \Cref{prop:vertices}), the dilation $(n+1)!\mathsf{P}^+(\varpi_i)\cap V_{\boldsymbol{i}}$ has vertices in $\bbZ\Phi$, by \Cref{cor:vertices_lattice}.
Therefore, substituting \eqref{eq: mu dominant} in \eqref{eq:Li translation trick} and applying \Cref{thm:technical_tool}, we get that the function $q_{\boldsymbol{i}}( \lambda_{1}, \dots, \lambda_{n})$ is a quasi-polynomial in the variables $\mu_1,\ldots,\mu_n$, and that its top homogeneous component is
\begin{equation*}
\relvol \left( \mathsf{P}^+_{\boldsymbol{i}}(\lambda) - \lambda\right) = \relvol \left( \mathsf{P}^+_{\boldsymbol{i}}(\lambda)\right).
\end{equation*}
After performing the change of variables from $\mu_i$ to $\lambda_i$, we get that $q_{\boldsymbol{i}}$ is also quasipolynomial in these variables, and its top homogeneous component remains unaltered as $\mu_i$ and $\lambda_i$ differ only by a constant.
\end{proof}

Now we come to the crux of the section.
Recall that we proved \cite[Theorem A]{castillo2023size} that the size of the lower Bruhat interval $\li{\lambda}$ for $a=\text{id}$ is equal to $(n+1)!|\mathsf{P}(\lambda)\cap (\mathbb{Z}\Phi + \lambda)|$, which agrees with a polynomial function on $\lambda$.
The following theorem establishes a weaker result, concluding only quasi-polynomiality. 
However, it offers greater generality, as it holds for general $a\in\mathcal{F}$ and for weights deep enough in the dominant region.
Let $\boldsymbol{2}\coloneqq(2,\ldots , 2) \in \{0,1,2\}^n$.
In the notation of \Cref{prop:quasi_polinomiality}, recall that $\bbN_{\boldsymbol{2}}$ denotes the $n$-tuples of integers $(\lambda_1,\ldots,\lambda_n)$ with all $\lambda_i\geq2$.

\begin{theorem}\label{thm:lattice_count}
Let $a\in \mathcal{F}$.
The function $q_a:\bbN_{\boldsymbol{2}}\to \mathbb{Z}$ given by
\begin{equation*}
q_a(\lambda_1,\ldots,\lambda_n) = \big| \mathcal{I}_a(\lambda)  \big|,
\quad \text{where} \quad
\lambda = \sum_{i=1}^n \lambda_i \varpi_{i},
\end{equation*}
is a quasi-polynomial in $\lambda_1,\ldots,\lambda_n$, that is, when $\lambda\in Z(\boldsymbol{2})$.
Furthermore, it is of degree $n$ and its top homogeneous component is given by $$q^{\mathsf{top}}_a( \lambda_{1}, \dots, \lambda_{n}) = (n+1)!\relvol \left( \mathsf{P}(\lambda)\right).$$
\end{theorem}

\begin{proof}
We recall Equation \eqref{eq:weighted_lattice_count}:
\[
|\li{\lambda}| = \sum_{\boldsymbol{i} \in \left\{ 0,1,2 \right\}^{n}} c_a(\boldsymbol{i}) |\mathsf{P}_{\boldsymbol{i}}^+(\lambda) \cap \left( \mathbb{Z}\Phi + \lambda \right) |.
\]
Since $\bbN_{\boldsymbol{2}} \subset \bbN_{\boldsymbol{i}}$ for all $\boldsymbol{i}$, each $\mathsf{P}_{\boldsymbol{i}}^+(\lambda)$ is non-empty.
\Cref{prop:quasi_polinomiality} implies that the right-hand side agrees with a quasi-polynomial in $\lambda_1,\ldots,\lambda_n$ and so the left-hand side does too.
Note that \( \dim(\mathsf{P}^+_{\boldsymbol{2}}(\lambda)) = n \), whereas in all other cases, the dimension is strictly lower.
As the top homogeneous component is contributed by polytopes of maximal dimension, we get
\begin{equation}
q^{\mathsf{top}}_a(\lambda_1,\ldots,\lambda_n) = c_a(\boldsymbol{2}) |\mathsf{P}_{\boldsymbol{2}}^+(\lambda) \cap \left( \mathbb{Z}\Phi + \lambda \right) |
= (n+1)!^2\relvol(\mathsf{P}_{\boldsymbol{2}}^+(\lambda)),
\end{equation}
where the second equality follows from Propositions \ref{prop:quasi_polinomiality} and \ref{prop:BarquitoGenerico} (see \cref{eq: cardinalities}).

The polytope \( \mathsf{P}(\lambda) \) from \cref{eq:P_definition} can be subdivided into \( (n+1)! 
\) regions, each isometric to \( \mathsf{P}^{+}(\lambda) \). 
Since $\lambda-\varpi(\boldsymbol{2})$ is dominant, we have
\[
(n+1)! \relvol(\mathsf{P}_{\boldsymbol{2}}^+(\lambda)) = (n+1)! \relvol(\mathsf{P}^+(\lambda -\varpi(\boldsymbol{2}))) = \relvol(\mathsf{P}(\lambda - \varpi(\boldsymbol{2}))),
\]  
where the first equality follows from \Cref{prop:vertices}, as $V_{\boldsymbol{2}}=E$.

Finally observe that as polynomials in $\lambda_1,\ldots,\lambda_n$, the top homogeneous component of $ \relvol(\mathsf{P}(\lambda - \varpi(\boldsymbol{2})))$ is precisely $\relvol(\mathsf{P}(\lambda))$, which is a homogeneous polynomial of degree $n$.
\end{proof}

\begin{remark}\label{rem:why_2+}
There are elements $\lambda\in C^+\setminus Z_\bbR(\boldsymbol{2)}$ satisfying the condition that $ \mathsf{P}^{+}_{\boldsymbol{2}}(\lambda) $ is non-empty.
However, our methods do not apply to such $ \lambda $ since $  \mathsf{P}^{+}_{\boldsymbol{2}}(\lambda)  $ may fail to be a translation of a polytope of the form $  \mathsf{P}^{+} $.
Hence we don't have the same Minkowski decomposition as in \Cref{prop:Minkowski_sum} to apply \Cref{thm:technical_tool}.
An example of this case happens with $ n=3 $ and the polytope $ \mathsf{P}^{+}_{\boldsymbol{2}}(3\varpi_1 + \varpi_2 + 3\varpi_3) $.
\end{remark}

\begin{remark}
In \Cref{thm:lattice_count} we use relative volume polynomials, whereas in \Cref{sec: geoforA3} we use volume polynomials.
For type $A_n$ and a full-dimensional polytope in $\Lambda \otimes \mathbb{R}$, they differ by a constant: $\sqrt{n+1}\relvol = \vol$.
This comes from the fact that the fundamental parallelogram of the weight lattice $\Lambda$ has euclidean volume $\sqrt{n+1}$, see \cite[Section 2.3.1]{castillo2023size}.
\end{remark}

The quasi-polynomial obtained in \Cref{thm:lattice_count} involves a large number of cosets.
We believe that this is only a consequence of the methods used and that actually the sizes of the lower Bruhat intervals are polynomials.
Here is a weaker version of \Cref{conj: pols geom}.

\begin{conjecture}\label{conj:miracle}
Let $\lambda = \sum_i \lambda_i \varpi_i \in \Lambda^+$ be a dominant weight and let $a \in \mathcal{F}$.
Then $|\mathcal{I}_a(\lambda)|$ is a \textbf{polynomial} in $\lambda_1, \dots, \lambda_n$ of degree $n$.
\end{conjecture}

\Cref{conj:miracle} has been verified for $n = 1,2,3$.
From the point of view of this section it feels truly miraculous;
\Cref{thm:lattice_count} writes the formula as a combination of genuine quasi-polynomials.
But then, the combination using the paper boat coefficients $c_a(\boldsymbol{i})$ apparently make all the cases collapse into a single one.
This means that there are strong numerical relations between the $c_a(\boldsymbol{i})$ that are yet to be discovered.

We end by mentioning other instances of this phenomemon of rational polytopes having an Ehrhart polynomial (as opposed to quasi-polynomial).
There is the example of the \emph{stretched Littlewood-Richardson poynomials} shown by \cite{rassart2004polynomiality}, where the relevant polynomial is a priori a quasi-polynomial but the author proves that there is a single case.
More examples of this phenomenon can be found in \cite{haase2008quasi}.

\section{A criterion for Bruhat order in affine type A} \label{sec: criterio}
The main goal of this section is to establish the criterion for the Bruhat order presented in \Cref{prop: main tool}. To do so, we first prove a generalized version of the criterion applicable to the entire affine Weyl group. Subsequently, we demonstrate that this generalized criterion simplifies significantly when comparing two elements within the dominant chamber.

In this section,  we use the realization of a root system of type $A$ introduced at the beginning of \Cref{sec: main results}. 
We recall that the vertices of the fundamental alcove are $V(\mathrm{id})=\{\mathbf{0}=\varpi_0,-\varpi_1,\ldots,-\varpi_n\}$.

\subsection{Bruhat order in terms of dominance order}
\subsubsection{The group \texorpdfstring{$S_{n+1} \ltimes \mathbb{Z}\Phi'$}{Sn x ZPhi'}}
\label{ssec:realization}
Let $S_{n+1}$ be the symmetric group on $n+1$ letters. It acts on the left on $\mathbb{Z}^{n+1}$ by permutation of coordinates, \ie  if $w \in S_{n+1}$ and $(\mu_1,\ldots, \mu_{n+1})\in \mathbb{Z}^{n+1},$ we have 
$w\cdot (\mu_1,\ldots, \mu_{n+1})= (\mu_{w^{-1}(1)},\ldots, \mu_{w^{-1}(n+1)})$.
Let us define (\Cref{dictionary} will clarify the reasoning behind our choice of notation)
\begin{align*}
\Lambda'   \displaystyle &:=\{(\mu_1,\ldots, \mu_{n+1})\in \mathbb{Z}^{n+1} \mid  0\leq \sum_{i=1}^{n+1}\mu_i \leq n \},   \\
\mathbb{Z}\Phi'&:=\{(\mu_1,\ldots, \mu_{n+1})\in \mathbb{Z}^{n+1}\ \vert\  \sum_{i=1}^{n+1}\mu_i=0\}\subset \Lambda'.
\end{align*}

For $ i \in \{ 1, 2, \ldots, n \} $, define $ \alpha'_i := (0, \ldots, 0, 1, -1, 0, \ldots, 0) \in \mathbb{Z}\Phi' $, where the $ 1 $ appears in the $ i $-th position.
Similarly, we define $ \varpi'_i := (1, \ldots, 1, 0, \ldots, 0) \in \bbZ^{n+1} $, where the first $ i $ entries are equal to $ 1 $, and the remaining $ n+1 - i $ entries are equal to $ 0 $.
We note that $\varpi'_i\in\Lambda'$ for $0\leq i\leq n$.

Define the partial order $\preceq$ on $\mathbb{Z}^{n+1}$ as follows: $(\mu_1, \ldots, \mu_{n+1}) \preceq (\mu'_1, \ldots, \mu'_{n+1})$ if and only if 
\begin{equation}\label{eq: Zn order}
\displaystyle  \sum_{i=1}^{n+1} \mu_i = \sum_{i=1}^{n+1} \mu'_i 
\quad \text{ and } \quad  \sum_{i=1}^m \mu_i \leq \sum_{i =1}^m \mu'_i,  \qquad   \text{ for all } m = 1, \ldots, n+1.
\end{equation}

We say $\mu = (\mu_1, \ldots, \mu_{n+1})$ is \emph{dominant} if $\mu_1 \geq \cdots \geq \mu_{n+1}$. We denote by $\mu_{\text{dom}}$ the unique dominant element in the $S_{n+1}$-orbit of $\mu$. 
We notice that $\mu_{\text{dom}}$ satisfies    $w\cdot \mu \preceq \mu_{\text{dom}} $ for all $w\in S_{n+1}$.

For $i\in \{1,2,\ldots, n\}$, let $s_i\in S_{n+1}$ be the transposition exchanging $i$ and $i+1$.
It is well known that the map $h:S_{n+1}\to W_\mathrm{f}$  given by $h(s_i)= s_{\alpha_i}$ for all $i\in \{1,2,\ldots, n\}$ is a group isomorphism.
For $w\in S_{n+1}$ and $\lambda\in\Lambda$, we have an action $w\cdot\lambda := h(w)\cdot \lambda = h(w)(\lambda)$.
The proof of the following lemma is straightforward and is therefore omitted.

\begin{lemma}\label{dictionary} Consider the  map  $g:\mathbb{Z}^{n+1} \rightarrow \Lambda$ given by $\displaystyle g(\mu_1,\ldots, \mu_{n+1})= \sum_{i=1}^{n}(\mu_i-\mu_{i+1})\varpi_i$.
Then,
\begin{enumerate}[(i)]
\item\label{itemNicoA} $g$ is a morphism of $S_{n+1}$-modules.
\item $g(\alpha'_i)=\alpha_i$ for all $i\in \{1, \ldots , n\}$ and $g(\varpi'_i)=\varpi_i$ for $i\in \{0,1,\ldots, n\}$.
\item $g$ restricts to a bijection $\Lambda'\rightarrow \Lambda$ and to a bijection $\mathbb{Z}\Phi' \rightarrow \mathbb{Z}\Phi$.
\item In the former bijection, dominant elements in $\Lambda'$ correspond to elements in $\Lambda^+$.
\item\label{itemNicoE} $\mu\preceq\nu \iff g(\mu)\leq g(\nu),$  where $\leq$ here stands for the dominance order.  
\end{enumerate}
\end{lemma}
In particular, we note that $\mathbf{0}=\varpi_0=g(\varpi_0')=g(\varpi_{n+1}')$.

For $\lambda\in\Lambda$, we write $\lambda_\mathrm{dom}$ for the unique element in $\Lambda^+$ in the $W_\mathrm{f}$-orbit of $\lambda$.
We notice that $\lambda_{\text{dom}}$ is the unique maximal element in the $W_\mathrm{f}$-orbit of $\lambda$.

\begin{cor}\label{gdom}
For $\mu\in \mathbb{Z}^{n+1}$ there is an equality $g(\mu_\mathrm{dom})=g(\mu)_\mathrm{dom}$.
\end{cor}

\begin{proof}
Let $\mu_\mathrm{dom}=w\cdot\mu$, for some $w\in S_{n+1}$.
Then for every $u\in S_{n+1}$ we have $u\cdot\mu\preceq w\cdot\mu$.
It follows from  items \ref{itemNicoA} and \ref{itemNicoE} in   \Cref{dictionary} that $u\cdot g(\mu)\leq w\cdot g(\mu)$.
Since $h$ is a bijection we conclude that $w\cdot g(\mu)$ is maximal in the $W_\mathrm{f}$-orbit of $g(\mu)$.
It follows that $g(\mu)_\mathrm{dom} = w\cdot g(\mu) = g(w\cdot \mu) = g(\mu_\mathrm{dom})$. 
\end{proof}

The left action of $S_{n+1}$ on $\mathbb{Z}^{n+1}$ restricts to a left action of $S_{n+1}$ on $ \mathbb{Z}\Phi'$. Any left action induces a corresponding right action by applying the inverse, allowing us to define the left semidirect product $S_{n+1} \ltimes \mathbb{Z}\Phi',$ with a right action of $S_{n+1}$. The elements of this semidirect product are denoted  $y = w \varepsilon^{\mu},$ where $w\in S_{n+1}$ and $\mu\in \mathbb{Z}\Phi'.$  For $x=w \varepsilon^\mu$ and $y=u \varepsilon^\lambda$, the product in the semidirect product is given by
$$xy=wu\varepsilon^{u^{-1}(\mu)+\lambda}.$$ 

\subsubsection{Four isomorphic affine Weyl groups}\label{sec:isomorphisms}
We will use the notation $(\mathbb{Z}\Phi \rtimes W_\mathrm{f} )^\bullet$ for the affine Weyl group, following the conventions in \cite{Bour46},  where the identity alcove is chosen to be $\mathcal{A}_{w_0}$ in our conventions. As before, the left semidirect product $(W_\mathrm{f}\ltimes \mathbb{Z}\Phi)^\bullet$ is defined using the right action induced by the inverse of the standard left action. 

Recall the map $h:S_{n+1}\to W_\mathrm{f}$  from last section. Using \Cref{dictionary}, we derive the following isomorphism of groups.
\begin{align*}
g_1\colon S_{n+1} \ltimes \mathbb{Z}\Phi' &\xrightarrow{\sim} (W_\mathrm{f}\ltimes \mathbb{Z}\Phi)^\bullet \\
w \varepsilon^{\mu} &\mapsto (h(w),g(\mu)).\end{align*}
The following map is the standard group isomorphism between left and right semidirect products.
\begin{align*}
g_2\colon (W_\mathrm{f}\ltimes \mathbb{Z}\Phi)^\bullet &\xrightarrow{\sim}(\mathbb{Z}\Phi \rtimes W_\mathrm{f} )^\bullet  \\
(w,\eta) &\mapsto (w(\eta),w).\end{align*}
We now need to relate this group and $\widetilde{S}_{n+1}(GR)$, which is defined by the generators 
$\{s_0,s_1,\cdots, s_{n}\}$ and the relations determined by the following Coxeter graph:
\begin{center}
\begin{dynkinDiagram}[labels={s_0,s_1,s_2,s_{n-1},s_{n}},edge length=1.5cm,root radius=0.09cm]A[1]{}
\end{dynkinDiagram}    
\end{center}

We define $\tilde{w}=s_{\tilde{\alpha}}\in W_\mathrm{f}$, where $\tilde{\alpha}$ is the highest root.
In terms of generators, we have $$\tilde{w}=s_{\alpha_1}s_{\alpha_2}\cdots s_{\alpha_{n-1}}s_{\alpha_{n}}s_{\alpha_{n-1}}\cdots s_{\alpha_2}s_{\alpha_1},$$
so that $h^{-1}(\tilde{w})\in S_{n+1}$ corresponds to the transposition exchanging $1$ and $n+1$.
The following isomorphism is classical.
\begin{align*}
g_3 \colon (\mathbb{Z}\Phi \rtimes W_\mathrm{f} )^\bullet &\xrightarrow{\sim}  \widetilde{S}_{n+1}(GR)  \\
(0,s_{\alpha_i}) &\mapsto s_i \  \ \mathrm{for\ all\ } 1\leq i\leq n,\\
(\tilde{\alpha},\tilde{w}) &\mapsto s_0.\end{align*}
Lastly we need to introduce the group $\perm$ (in \cite{BB} is called $\widetilde{S}_{n+1}$). It is the  subgroup of the permutation group of $\mathbb{Z}$, with elements being the permutations $\pi: \mathbb{Z}\to \mathbb{Z}$  satisfying: 
\begin{enumerate}[label=(P{{\arabic*}})]
\item $\pi(t+n+1)=\pi(t)+n+1$ for all $t\in \mathbb{Z}$, \label{p1}
\item $\displaystyle  \sum_{i=1}^{n+1} \pi(t)=  \sum_{t=1}^{n+1} t=\frac{(n+1)(n+2)}{2}. $\label{p2}
\end{enumerate}
Such a $\pi \in \perm$ is uniquely determined by its values on $\{1,2,\ldots, n+1\}$,  we write $\pi=[a_1, \ldots, a_{n+1}]$ (this is called the \emph{window} notation) to mean that $\pi(t)=a_t$ for $1\leq t\leq n+1$. The following isomorphism can be found in \cite[Proposition 8.3.3]{BB}.
\begin{align*}
g_4 \colon \widetilde{S}_{n+1}(GR) &\xlongrightarrow{\sim} \perm  \\
s_i &\mapsto [1,2,\ldots, i-1,i+1,i+2,\ldots, n+1]\\
s_0 &\mapsto [0,2,3,\ldots, n,n+2].\end{align*}

\subsubsection{The map \texorpdfstring{$\pi$}{pi}}
To each element $y = w \varepsilon^\mu \in S_{n+1} \ltimes \mathbb{Z}\Phi'$, we associate a  map $\pi_y : \mathbb{Z} \to \mathbb{Z}$ as follows. Since each integer can be uniquely written in the form $q(n+1) + r$ with $q \in \mathbb{Z}$ and $r \in \{1, \ldots, n+1\}$, it suffices to set
\[
\pi_y(q(n+1) + r) = (q + \mu_r)(n+1) + w(r).
\]
Let us explain why the assignment $y\mapsto \pi_y$ defines a map $\pi:S_{n+1} \ltimes \mathbb{Z}\Phi'\to \perm$.
The bijectivity of $\pi_y$  and  \Cref{p1} follow directly from the definitions.
To verify \Cref{p2}, one must use the fact that, by definition, for $\mu\in \mathbb{Z}\Phi'$ the equation $\displaystyle \sum_{t=1}^{n+1}\mu_t=0$ holds.
\begin{lemma}\label{lem:pi}
The following diagram commutes.
\[\begin{tikzcd}
{S_{n+1} \ltimes \mathbb{Z}\Phi'} & {(W_\mathrm{f}\ltimes \mathbb{Z}\Phi)^\bullet} & {(\mathbb{Z}\Phi \rtimes W_\mathrm{f} )^\bullet} & {\widetilde{S}_{n+1}(GR)} & \perm
\arrow["{g_1}", from=1-1, to=1-2]
\arrow["\pi"', bend right=20, from=1-1, to=1-5]
\arrow["{g_2}", from=1-2, to=1-3]
\arrow["{g_3}", from=1-3, to=1-4]
\arrow["{g_4}", from=1-4, to=1-5]
\end{tikzcd}\]
\end{lemma}
\begin{proof}

Let us first prove that $\pi$ is a group morphism. 
Indeed, if $x=w\varepsilon^\mu$ and $y=u\varepsilon^\lambda,$
\begin{equation}
\begin{array}{rl}
\pi_x\circ \pi_y(q(n+1)+r)    & =  \displaystyle  \pi_x((q+\lambda_r)(n+1)+u(r)) \\
&  \\
&  =  \displaystyle
(q+\lambda_r+\mu_{u(r)})(n+1)+wu(r).
\end{array}
\end{equation}
On the other hand, $\pi_{xy}(q(n+1)+r)=(q+(u^{-1}(\mu)+\lambda)_r)(n+1)+wu(r)$, and these numbers are equal as by definition $\mu_{u(r)}=(u^{-1}(\mu))_r.$

This means that to prove the equality, one just needs to check it on the generators $ g_1^{-1} \circ g_2^{-1}\circ g_3^{-1}(s_i)$ with $i\in \{0,1,\ldots, n\}$.
For $i\in \{1,\ldots, n\}$, these elements are $s_i\varepsilon^0$, and $\pi_{s_i\varepsilon^0}(q(n+1)+r)=q(n+1)+s_i(r)$, which is clearly equal to $[1,2,\ldots, i-1,i+1,i+2,\ldots, n+1]$ in the window notation.

The element $ g_1^{-1} \circ g_2^{-1}\circ g_3^{-1}(s_0)$ equals $y=h^{-1}(\tilde{w})\varepsilon^{\tilde{\mu}}$ with $\tilde{\mu}=(-1,0,0,\ldots, 0,1)$.
To verify this, one needs to carefully follow the isomorphisms and use the fact that $\tilde{w}(\tilde{\alpha})=-\tilde{\alpha}=-\varpi_1-\varpi_n$ in $\bbZ\Phi$.
Recall that $h^{-1}(\tilde{w})$ is the transposition $(1 \ \ n+1)\in S_{n+1}$.
It follows that
\begin{equation}
\pi_y(0\cdot (n +1)+ i) = \tilde{\mu}_i(n+1) +(1 \ \ n+1)(i) =
\begin{cases}
0, &\mbox{if } i=1\\
i, &\mbox{if } 2\leq i\leq n\\
n+2, &\mbox{if } i=n+1\\
\end{cases}
\end{equation}
Therefore $\pi_y=[0,2,3,\ldots, n,n+2]$ and the proof is complete.
\end{proof}

\subsubsection{The criterion}
The goal of this section is to prove the criterion in \Cref{criteriofinal}.
The Bruhat order is well defined in any Coxeter system, in particular in $\widetilde{S}_{n+1}(GR)$.
We transfer this structure to the other three  groups in Section \Cref{sec:isomorphisms} via the isomorphisms $g_i$.
In particular we obtain a Bruhat order in $S_{n+1} \ltimes \mathbb{Z}\Phi'$.

\begin{theorem}\label{criterio}
Let $y_1 = w_1 \varepsilon^{\mu_1} $, $y_2 = w_2 \varepsilon^{\mu_2} \in S_{n+1} \ltimes \mathbb{Z}\Phi'$. Then the following are equivalent.
\begin{enumerate}[label=(\alph*)]
\item We have $y_1 \leq y_2$ in the Bruhat order.
\item For all $i, j \in \{1, \ldots, n+1\}$, we have
\begin{equation}\label{eq: ineq criterio}
(\mu_1 + \varpi_i' - w_1^{-1} \varpi_{j-1}')_\mathrm{dom}\preceq (\mu_2 + \varpi_i' - w_2^{-1} \varpi_{j-1}')_\mathrm{dom}.    
\end{equation}
\end{enumerate}
\end{theorem}
Note that the coordinate-sum of either side of \eqref{eq: ineq criterio} equals $i+j-1$, so that the left condition in \eqref{eq: Zn order} is always satisfied.

\begin{proof}
For integers $i, i' \in \mathbb{Z}$ and $y = w \varepsilon^{\mu} \in \widetilde{S}_{n+1}$, we define
\[
y[i, i'] := \#\{a \leq i \mid \pi_y(a) \geq i'\}.
\]

By \Cref{lem:pi},   \cite[Theorem 8.3.7]{BB} translates to the following result.

\begin{prop}\label{prop:Brenti} Let $y_1 = w_1 \varepsilon^ {\mu_1}$, $y_2 = w_2 \varepsilon^{\mu_2} \in S_{n+1} \ltimes \mathbb{Z}\Phi'$. Then the following are equivalent.
\begin{enumerate}[label=(\alph*)]
\item\label{item:Brenti1} $y_1 \leq y_2$ in the Bruhat order.
\item\label{item:Brenti2} For all $i, i' \in \mathbb{Z}$, we have
\[
y_1[i, i'] \leq y_2[i, i'].
\]
\end{enumerate}
\end{prop}
For $y=w \varepsilon^ {\mu}\in \widetilde{S}_{n+1},$ let us evaluate $y[i, i']$ for $i = (n+1)q + r$ and $i' = (n+1)q' + r'$, where $r,r' \in \{1,2,...,n+1\}$. We have the sequence of equalities
\begin{align*}
y[i, i'] &= \#\{a \leq i \mid \pi_y(a) \geq i'\} \\
&= \sum_{m=1}^{n+1} \#\{a \leq i \mid \pi_y(a) \geq i' \text{ and } a \equiv m \pmod{n+1}\} \\
&= \sum_{m=1}^{n+1} \#\{b \in \mathbb{Z} \mid b(n+1) + m \leq i \text{ and } \pi_y(b(n+1) + m) \geq i'\} \\
&= \sum_{m=1}^{n+1} \#\{b \in \mathbb{Z} \mid b(n+1) + m \leq q(n+1) + r \text{ and } (b + \mu_m)(n+1) + w(m) \geq q'(n+1) + r'\}.
\end{align*}

\begin{lemma}\label{lem:easy}
Suppose that $m,r\in \{1,2, \ldots ,n+1\}$. 
The following conditions are equivalent:
\begin{enumerate}
\item\label{item:easy1} $b(n+1)+m \leq q(n+1) +r$
\item\label{item:easy2} $b+(\varpi'_m)_{r+1} \leq q $
\item\label{item:easy3} $b<q+(\varpi'_r)_m$
\end{enumerate}
\end{lemma}
\begin{proof}
The equality $(\varpi'_m)_{r+1}=1-(\varpi'_r)_m$ establishes the equivalence \ref{item:easy2} $\iff$ \ref{item:easy3}.
The equivalence \ref{item:easy1} $\iff$ \ref{item:easy3} is obtained by analyzing the cases $m\leq r$ and $m>r$ separately.
\end{proof}

We can now evaluate $y[i,i']$ in the abovementioned elements. 
\begin{equation}
\begin{array}{rl}
y[i,i']     & =  \displaystyle  \sum_{m=1}^{n+1}  \# \{  b\in \mathbb{Z} \mid b(n+1)+m\leq q(n+1)+r \mbox{ and } q'(n+1)+r'\leq (b+\mu_m)(n+1)+w(m)   \}    \\
&  \\
&  \displaystyle =  \sum_{m=1}^{n+1}  \# \{  b\in \mathbb{Z} \mid b<q+(\varpi'_r)_{m} \mbox{ and }  q'+ ( \varpi'_{r'})_{w(m)+1}\leq b+\mu_m     \}   \\
& \\
&  =  \displaystyle   \sum_{m=1}^{n+1}  \# \{  b\in \mathbb{Z} \mid b<q+(\varpi'_r)_{m} \mbox{ and }  q'+ (w^{-1}\cdot \varpi'_{r'-1})_m  \leq b+\mu_m     \}    \\
& \\
&   =   \displaystyle  \sum_{m=1}^{n+1}  \max ( 0, q+(\varpi'_r)_{m} +\mu_m -  q'-  (w^{-1}\cdot \varpi'_{r'-1} )_m ) . 
\end{array}
\end{equation}
We use the implications \ref{item:easy1} $\implies$ \ref{item:easy2} and \ref{item:easy1} $\implies$ \ref{item:easy3} of \Cref{lem:easy} for the second equality.
The third equality follows from the equation $(\varpi'_{r'})_{w(m)+1}=(\varpi'_{r'-1})_{w(m)}$ and the last equation is obtained applying the formula $\# \{a\leq b<c\}=\max (0,c-a)$ in this context.

We define $Q:=q-q'$ and $y^{(r,r')} := \mu +\varpi'_r - w^{-1}\cdot \varpi'_{r'-1}$.
With this notation we have
\begin{equation}
y[i,i']   = \displaystyle  \sum_{m=1}^{n+1} \max ( 0, Q + (y^{(r,r')})_m )
\end{equation}

So condition \ref{item:Brenti2} in \Cref{prop:Brenti} is equivalent to
\begin{equation}\label{condb1}
\forall\,  r, r' \in \{1, \ldots, n+1\},  \ \forall \ Q \in \mathbb{Z} : \sum_{m=1}^{n+1} \max(0, Q + (y^{(r, r')}_1)_m) \leq \sum_{m=1}^{n+1} \max(0, Q + (y^{(r, r')}_2)_m).
\end{equation}

For $\nu\in \mathbb{Z}^{n+1}$ and  $Q\in \mathbb{Z}$ let us use the notation $$\vert\nu\vert_Q:= \sum_{m=1}^{n+1} \max(0, Q+\nu_m).$$
In this language, condition \eqref{condb1} is equivalent to
\begin{equation}\label{condb2}
\forall\,  r, r' \in \{1, \ldots, n+1\}, \ \forall \ Q  \in \mathbb{Z} : \vert y^{(r, r')}_1 \vert_Q \leq \vert y^{(r, r')}_2 \vert_Q.
\end{equation}

\begin{lemma}\label{luc}
Let $\mu, \lambda \in \mathbb{Z}^{n+1}$ be dominant elements such that  $\mu_1 + \cdots + \mu_{n+1} = \lambda_1 + \cdots + \lambda_{n+1}$.
Then, we have
\begin{equation}\label{equiv} 
\mu \preceq \lambda
\Longleftrightarrow \forall Q \in \mathbb{Z} :\vert\mu\vert_Q \leq \vert\lambda\vert_Q.
\end{equation}
\end{lemma}

\begin{proof}
We first reduce the proof to the case where both $\mu$ and $\lambda$ are partitions. 
Let  $a=\min\{\mu_{n+1},\lambda_{n+1}\}$ and define $\tilde{\lambda}:=(\lambda_1-a,\ldots,\lambda_{n+1}-a)$, $\tilde{\mu}:=(\mu_1-a,\ldots,\mu_{n+1}-a) \in \mathbb{Z}^{n+1}$.
These two elements are partitions satisfying 
\begin{equation}\label{abc}
	\mu\preceq \lambda \iff \tilde{\mu}\preceq \tilde{\lambda}.
\end{equation} 
On the other hand, shifting the variable  $Q$  to $Q+a$  yields the equivalence
\begin{equation}\label{def}
\forall Q \in \mathbb{Z} : \vert\tilde{\mu}\vert_Q \leq \vert\tilde{\lambda}\vert_Q
\Longleftrightarrow 
\forall Q \in \mathbb{Z} : \vert{\mu}\vert_Q \leq \vert{\lambda}\vert_Q.    
\end{equation}
By combining \eqref{abc} and \eqref{def} we conclude that the original statement holds for the pair  $(\mu, \lambda)$ if and only if it holds for the pair  $(\tilde{\mu}, \tilde{\lambda})$.

By the previous paragraph we can assume that  $\mu$ and $\lambda$ are partitions.
For any partition $\nu=(\nu_1,\ldots,\nu_l)$ we use the convention that $\nu_i:=0$ if $i>l,$ so that $\nu_i$ is defined for $i\in \mathbb{Z}_{>0}.$
We denote by $\nu'$  the conjugate of $\nu$ and we put $\vert\nu\vert=\sum_{i=1}^l\nu_i$. For all $Q\in \mathbb{Z}_{<0}$ one has the equality (see \Cref{fig:Young})
\begin{equation}\label{eq:property mu}
 \mu'_1+\cdots +\mu'_{-Q}+\vert\mu\vert_Q=  \vert\mu\vert= \vert\lambda\vert= \lambda'_1+\cdots + \lambda'_{-Q}+\vert\lambda\vert_Q,
\end{equation}
where $\mu'=(\mu_1',\mu_2',\ldots)$ and $\lambda' =(\lambda_1',\lambda_2', \ldots)$.
\begin{figure}
\centering
\begin{tikzpicture}[scale=0.45]
\filldraw[lightgray]  (2,5)--++(1,0)--++(0,2)--++(1,0)--++(0,2)--++(2,0)--++(0,1)--++(-4,0) -- cycle;
\draw[black, ultra thick]  (-2,0)--++(1,0)--++(0,2)--++(2,0)--++(0,3)--++(2,0)--++(0,2)--++(1,0)--++(0,2)--++(2,0)--++(0,1)--++(-8,0)-- cycle;
\draw[blue,ultra thick] (2,-2)--++(0,13);
\node at (-3.5, 6)   { \large  $  \lambda =$};
\node at (3, 8)   { \large $  |\lambda|_Q $};
\node at (0, -1)   { \small $ - Q $};
\draw[black, thick,->] (0.8,-1)--(1.9,-1);
\draw[black, thick,->] (-0.8,-1)--(-2,-1);
\end{tikzpicture}
\caption{Using Young diagrams to illustrate \Cref{eq:property mu}.}
\label{fig:Young}
\end{figure}
So we conclude that for all $Q\in \mathbb{Z}_{<0},$
\begin{equation}\label{keyeq}
	\mu'_1+\cdots +\mu'_{-Q}\geq  \lambda'_1+\cdots +\lambda'_{-Q} \iff \vert\mu\vert_Q\leq \vert\lambda\vert_Q.
\end{equation}
Using the well-known equivalence $\mu \preceq \lambda \iff  \mu'\succeq \lambda',$ as well as \Cref{keyeq} we obtain that 
\begin{equation}\label{keyeq1}
\mu \preceq \lambda \iff \vert\mu\vert_Q\leq \vert\lambda\vert_Q.
\end{equation}
for all  $Q \in \mathbb{Z}_{<0}$.
On the other hand, for all  $Q\in \mathbb{Z}_{\geq 0}$ we have 
\begin{equation}\label{keyeq2}
| \lambda |_Q = Q(n+1)+   |\lambda|  = Q(n+1) + |\mu| = |\mu|_Q.
\end{equation}
By combining \eqref{keyeq1} and \eqref{keyeq2} we obtain \eqref{equiv}.
\end{proof}

It is clear that if $\nu\in \mathbb{Z}^{n+1}$ and  $Q\in \mathbb{Z}$ one has $\vert\nu\vert_Q=\vert\nu_{\mathrm{dom}}\vert_Q$. Thus, by applying \Cref{luc}, we conclude that condition \eqref{condb2} is equivalent to:

\[
\forall r, r' \in \{1, \ldots, n+1\}  :  (y^{(r, r')}_1)_{\mathrm{dom}}  \preceq (y^{(r, r')}_2)_{\mathrm{dom}},
\]
or in other words, 
\begin{equation}\label{quasifinal}
\forall r, r' \in \{1, \ldots, n+1\}  :  (\mu_1 +\varpi'_r - w_1^{-1}(\varpi'_{r'-1}))_{\mathrm{dom}}  \preceq (\mu_2 +\varpi'_r - w_2^{-1}(\varpi'_{r'-1}))_{\mathrm{dom}}.
\end{equation}
This ends the proof of \Cref{criterio}.
\end{proof}
In the next corollary, we will identify  $S_{n+1}$ and $W_\mathrm{f}$ through the map $h$. This simplification is intended to avoid further complicating the notation in the proof.

\begin{cor}\label{criteriofinal}
Let $x_1 = (w_1, \nu_1) $, $x_2 = (w_2,\nu_2) \in (W_\mathrm{f} \ltimes \mathbb{Z}\Phi)^\bullet$,
then the following are equivalent.
\begin{enumerate}[label=(\alph*)]
\item We have $x_1 \leq x_2$ in the Bruhat order.
\item For all $i, j \in \{0,1 \ldots, n\}$, we have
\[
(\nu_1 + \varpi_i - w_1^{-1} \varpi_j)_\mathrm{dom}\leq (\nu_2 + \varpi_i - w_2^{-1} \varpi_j)_\mathrm{dom}.
\]
\end{enumerate}
\end{cor}

\begin{proof}
Recall the bijection $g_{\mathrm{res}}:\mathbb{Z}\Phi'\rightarrow\mathbb{Z}\Phi$, defined as the restriction of the map $g$ in \Cref{dictionary}.
Using the isomorphism $g_1$, we see that $x_1 \leq x_2$ in the Bruhat order if and only if
\begin{equation}\label{guatamala1}
w_1\varepsilon^{g_{\mathrm{res}}^{-1}(\nu_1)} = g_1^{-1}(x_1) \leq w_1\varepsilon^{g_{\mathrm{res}}^{-1}(\nu_1)} = g_1^{-1}(x_2).
\end{equation}

With this notation, and using the equivalence of \Cref{criterio}, \eqref{guatamala1} is equivalent to 
\begin{equation}\label{guatamala2}
\forall i, j \in \{1, \ldots, n+1\}  :  (g_{\mathrm{res}}^{-1}(\nu_1) +\varpi'_i - w_1^{-1}(\varpi'_{j-1}))_{\mathrm{dom}}  \preceq (g_{\mathrm{res}}^{-1}(\nu_2) +\varpi'_i - w_2^{-1}(\varpi'_{j-1}))_{\mathrm{dom}}.
\end{equation}

Applying the map $g$ to both sides of the inequality, using \Cref{dictionary} part \eqref{itemNicoE},  and then using \Cref{gdom}, we obtain that \ref{guatamala2} is equivalent to
\begin{equation}\label{guatamala3}
\forall i, j \in \{1, \ldots, n+1\}  :  g(g_{\mathrm{res}}^{-1}(\nu_1) +\varpi'_i - w_1^{-1}(\varpi'_{j-1}))_{\mathrm{dom}}  \leq g(g_{\mathrm{res}}^{-1}(\nu_2) +\varpi'_i - w_2^{-1}(\varpi'_{j-1}))_{\mathrm{dom}}.\end{equation}
As $g(\varpi'_{n+1})=\varpi_0$, by applying again \Cref{dictionary}, we obtain that \eqref{guatamala3} is equivalent to \[ \forall i, j \in \{0, \ldots, n\}  :  (\nu_1 +\varpi_i - w_1^{-1}(\varpi_{j}))_{\mathrm{dom}}  \leq (\nu_2 +\varpi_i - w_2^{-1}(\varpi_{j}))_{\mathrm{dom}}. \] 
\end{proof}

\subsection{Specializing the criterion for \texorpdfstring{$C^{+}$}{C+} and in our conventions}
We prove \Cref{prop: main tool}, which we restate for the reader's convenience.

\begin{prop}\label{prop:main_de_verdad} 
Let $x,y\in W_\mathrm{aff}^+$.
Then $x\leq y$ (in the Bruhat order), if and only if $x(-\varpi_i)\leq y(-\varpi_i)$ for all $i\in \{0,1,\ldots ,n\}$ (in the dominance order).
Equivalently, $x\leq y$ in the Bruhat order if and only if $\nu\leq\mu$ in the dominance order for all $(\nu,\mu)\in V(x)\times V(y)$ satisfying $\nu\equiv\mu\mod\bbZ\Phi$.
\end{prop}

The proof of \Cref{prop:main_de_verdad} is a translation of \Cref{criteriofinal} to our conventions.
To do this we first need the following lemma.
\begin{lemma}\label{lem: w0 bruhat}
Let $\lambda,\mu\in\Lambda$.
Then
\begin{equation*}
\mu\leq\lambda\iff w_0(\lambda)\leq w_0(\mu).
\end{equation*}
\end{lemma}

\begin{proof}
It suffices to show that $0\leq\lambda\implies w_0(\lambda)\leq0$.
It is well known that $w_0$ maps the set of fundamental weights to the set of the negative fundamental weights.
Let $\tau$ be the permutation  of  $\{1,\ldots , n\}$ defined by $w_0( \varpi_i)  =-\varpi_{\tau(i)}$.
Note that since $\mathbf{0}\leq \lambda $, for all $1\leq i\leq n$, we have $$0\leq  \lr{\lambda}{\varpi_i} =  \lr{w_0(\lambda)}{-\varpi_{\tau(i)}}.$$
This implies that $\lr{w_0(\lambda)}{\varpi_i}\leq 0$ for all $i\in \{1,\ldots , n \}$. In other words, $w_0(\lambda) \leq \mathbf{0}$. 
\end{proof}

\begin{proof}[Proof of \Cref{prop:main_de_verdad}]
We will prove only the first equivalence, as the second follows directly from it.
We denote by $(\bbZ\Phi\rtimes W_\mathrm{f})_\bullet$ the affine Weyl group defined as in \Cref{sec:preliminaries}.
Recall also that $(\bbZ\Phi\rtimes W_\mathrm{f})^\bullet$ is the affine Weyl used in \Cref{sec:isomorphisms}.
To apply the criterion in \Cref{criteriofinal}, we will need the following group isomorphisms:
\begin{equation}  \label{guatamala4}
\begin{array}{ccccc}
(\bbZ\Phi\rtimes W_\mathrm{f})_\bullet     & \xlongrightarrow{\sim} &  (\bbZ\Phi\rtimes W_\mathrm{f})^\bullet  & \xlongrightarrow{\sim}   &(W_\mathrm{f}\ltimes \bbZ\Phi)^\bullet  \\
 (\lambda,w)    & \longmapsto  &   (-\lambda,w) & \longmapsto   &(w,-w^{-1}(\lambda))
\end{array}
\end{equation}
Let $x=(\lambda,w)$ and $y= (\mu,u)$ in $(\bbZ\Phi\rtimes W_\mathrm{f})_\bullet$. Using \eqref{guatamala4} we see that $x\leq y$ (in the Bruhat order) if and only if
\begin{equation}\label{guatamala5}
(w,-w^{-1}(\lambda))\leq(u,-u^{-1}(\mu)) \ \mbox{ in } (W_\mathrm{f}\ltimes \bbZ\Phi)^\bullet.
\end{equation}

Applying \Cref{criteriofinal}, we get that \eqref{guatamala5} is equivalent to the statement that for all $i,j\in\{0,\ldots, n\}$, we have
\begin{equation}\label{eq: paso intermedio}
(-w^{-1}(\lambda)+\varpi_i-w^{-1}(\varpi_j))_{\mathrm{dom}}\leq (-u^{-1}(\mu)+\varpi_i-u^{-1}(\varpi_j))_{\mathrm{dom}},
\end{equation}
in the dominance order.
By definition, for any $v\in\Lambda$ and $z\in W_\mathrm{f}$ we have $v_\mathrm{dom}=(z\cdot v)_\mathrm{dom}$.
Hence, \eqref{eq: paso intermedio} becomes
\begin{equation*}
(-\lambda+w(\varpi_i)-\varpi_j)_{\mathrm{dom}}\leq (-\mu+u(\varpi_i)-\varpi_j)_{\mathrm{dom}}.
\end{equation*}
In turn, this last inequality is equivalent to 
\begin{equation}\label{eq: paso intermedio2}
(-x(-\varpi_i)-\varpi_j)_{\mathrm{dom}}\leq (-y(-\varpi_i)-\varpi_j)_{\mathrm{dom}}.
\end{equation}
Since $x\in W_\mathrm{aff}^+$ we have that $x(-\varpi_i)+\varpi_j\in C^+$, so that $-x(-\varpi_i)-\varpi_j\in C^-$, where $C^-=-C^+$.
On the other hand, by \cite[Corollary VI.1.6.3]{Bour46} we have $w_0(C^-)=C^+$.
It follows that
\begin{equation} \label{guatamala6}
(-x(-\varpi_i)-\varpi_j)_{\mathrm{dom}} = w_0(  -x(-\varpi_i)-\varpi_j  ) = -w_0x(\varpi_i)-w_0(\varpi_j),
\end{equation}
and the same holds for $y$. 
By \eqref{guatamala6}, we obtain that \eqref{eq: paso intermedio2} is equivalent to
\begin{equation*}
\forall i \in \{0,\ldots,n\} : w_0x(-\varpi_i)\geq w_0y(-\varpi_i).
\end{equation*}
Finally, multiplying by $w_0$ and applying \Cref{lem: w0 bruhat}, we get $ x(-\varpi_i) \leq y (-\varpi_i)$  for all $i\in \{0,1,\ldots , n\}$.
\end{proof}

\bibliography{bibliography}
\bibliographystyle{plain}
\end{document}